\documentclass[11pt,a4paper]{amsart}
\usepackage[english]{babel}
\usepackage{amssymb,amsthm,graphicx,mathtools}
\usepackage{url}
\usepackage[all]{xy}
\usepackage{longtable}
\usepackage{tikz}
\usetikzlibrary{arrows}
\definecolor{lightgrey}{rgb}{.804,.804,.756}
\usepackage[pdfauthor={V. Lebed, L. Vendramin},
pdftitle={Homology of the Yang--Baxter equation},
colorlinks=false,linkbordercolor=lightgrey,citebordercolor=lightgrey,urlbordercolor=lightgrey]{hyperref}
\usepackage{caption,longtable}
\usepackage[all]{hypcap}
\usepackage{pinlabel}

\numberwithin{equation}{section}
\numberwithin{table}{section}
\numberwithin{figure}{section}

\newcommand{\Z}{{\mathbb Z}}

\newcommand{\R}{{\mathbb R}}
\newcommand{\C}{{\mathcal C}}
\renewcommand{\O}{{\mathcal O}}

\newcommand{\Ret}{\operatorname{Ret}}

\newcommand{\Fun}{\operatorname{Fun}}
\newcommand{\End}{\operatorname{End}}
\newcommand{\Ext}{\mathcal{E}}

\newcommand{\Id}{\operatorname{Id}}

\newcommand{\im}{\operatorname{Im}}
\renewcommand{\ker}{\operatorname{Ker}}

\newcommand{\extension}[4]{(#1\xrightarrow{#2}#3,\, #4)}
\newcommand\op{\mathrel{\triangleleft}}
\newcommand\wop{\mathrel{\widetilde{\triangleleft}}}
\newcommand\ops{\mathrel{\triangleleft_\sigma}}
\newcommand\action{\mathrel{\leftthreetimes}}
\newcommand\oaction{\mathrel{\overline{\leftthreetimes}}}
\newcommand\wdot{\mathrel{\tilde{\cdot}}}
\newcommand\oX{\overline{X}}

\newcommand\oJ{\overline{J}}
\newcommand\osigma{\overline{\sigma}}

\newcommand\rcircled[1]{$\,$\tikz[baseline=(char.base)]{
            \node[shape=circle,draw,inner sep=1pt] (char) {#1};}}
\newcommand{\mapsfrom}{\mathrel{\reflectbox{\ensuremath{\mapsto}}}}
\newcommand\Degen{{\scriptstyle\mathrm{D}}}
\newcommand\Norm{{\scriptstyle\mathrm{N}}}
\newcommand\Braided{{\scriptstyle\mathrm{B}}}
\newcommand\Group{{\scriptstyle\mathrm{G}}}

\usepackage{color,xcolor}
\definecolor{choco}{rgb}{.525,.27,.075}
\definecolor{mygreen}{rgb}{0,.455,0} 
\definecolor{myviolet}{rgb}{.45,.05,.545}
\definecolor{myred}{rgb}{.545,0,0}
\definecolor{myblue}{rgb}{.024,.15,.645}           

\theoremstyle{plain}
\newtheorem{thm}{Theorem}[section]

\newtheorem{que}[thm]{Question}
\newtheorem{pro}[thm]{Proposition}
\newtheorem{lem}[thm]{Lemma}
\newtheorem{cor}[thm]{Corollary}
\theoremstyle{definition}
\newtheorem{defn}[thm]{Definition}
\newtheorem{rem}[thm]{Remark}
\newtheorem{example}[thm]{Example}

\newtheorem{notation}[thm]{Notation}

\begin{document}
\title[Homology of the Yang--Baxter equation]{Homology of left non-degenerate set-theoretic solutions to the Yang--Baxter equation}

\author{Victoria Lebed}
\address{Laboratoire de Math\'ematiques Jean Leray, Universit\'e de Nantes, 2 rue de la Houssini\`ere, BP 92208 F-44322 Nantes Cedex 3, France}
\email{lebed.victoria@gmail.com}

\author{Leandro Vendramin}
\address{Depto. de Matem\'atica, FCEN,
Universidad de Buenos Aires, Pabell\'on I, Ciudad Universitaria (1428)
Buenos Aires, Argentina}
\email{lvendramin@dm.uba.ar}

\thanks{The work of L.V. was partially supported by CONICET, PICT-2014-1376, and 
ICTP. V.L. is grateful to Henri Lebesgue Center (University of Nantes), for
warm and stimulating working atmosphere as well as for financial support via
the program ANR-11-LABX-0020-01. The authors thank Arnaud Mortier for valuable
comments on the early versions of this manuscript and for graphical assistance; Travis Schedler for the list of small involutive solutions; Patrick Dehornoy, Friedrich Wagemann, and Simon Covez for fruitful discussions; and the referee for valuable remarks allowing to put this work into better perspective, and for suggesting the methods of~\cite{MR1812049} for proving Proposition~\ref{PR:betti}.}


\date{\today}

\begin{abstract}
    This paper deals with left non-degenerate set-theoretic solutions to the
    Yang--Baxter equation (=LND solutions), a vast class of algebraic structures
    encompassing groups, racks, and cycle sets. To each such solution is associated a shelf
    (i.e., a self-distributive structure) which captures its major properties. 
    We consider two (co)homology theories for LND solutions, one of which was
    previously known, in a reduced form, for biracks only. An explicit isomorphism between these theories is described. For groups and racks we recover their classical (co)homology, whereas for cycle sets we get new constructions. For a certain type of LND 
    solutions, including quandles and non-degenerate cycle sets, the
    (co)homologies split into the degenerate and the normalized parts. We
    express $2$-cocycles of our theories in terms of group cohomology, and, in
    the case of cycle sets, establish connexions with extensions.
    This leads to a construction of cycle sets with interesting properties. 
\end{abstract}

\maketitle

\setcounter{tocdepth}{1}
\tableofcontents{}
	
\section{Introduction}

The \textit{Yang--Baxter equation} (=\textit{YBE}) plays a fundamental role in
such apparently distant fields as statistical mechanics, particle physics,
quantum field theory, quantum group theory, and low-dimensional topology; see
for instance~\cite{YBE} for a brief introduction. The study of its solutions
has been a vivid research area for the last half of a century. Following
Drinfel$'$d~\cite{MR1183474}, set-theoretic solutions, or \textit{braided
sets}, received special attention. Concretely, these are sets~$X$ endowed with
a \textit{braiding}, i.e., a not necessarily invertible map
$\sigma \colon X^{\times 2} \to X^{\times 2}$,
often written as $\sigma(a,b) = ( \prescript{a}{}{b}, a^b)$, satisfying the YBE
\begin{align}
    \label{E:YBE}
    (\sigma \times \Id) (\Id \times \sigma) (\sigma \times \Id) &= (\Id \times \sigma) (\sigma \times \Id) (\Id \times \sigma) \colon X^{\times 3} \to X^{\times 3}.
\end{align}

Two families of braided sets are particularly well explored:
\begin{itemize}
	\item The map $$\sigma(a,b) =  (b,a^b)$$ is a braiding if and only if the
		operation $a \op b := a^b$ is \textit{self-distributive}, in the
		sense of 
		\begin{align}\label{E:SD}
		(a \op b) \op c = (a \op c) \op (b \op c).
		\end{align}
		 Such datum
		$(X,\,\op)$ is called a \textit{shelf}. The term \textit{rack} is used if
		moreover the right translations $a \mapsto a \op b$ are bijections
		on~$X$ for all $b \in X$, which is equivalent to the invertibility
		of~$\sigma$. A \textit{quandle} is a rack satisfying 	
		$a \op a = a$ for all~$a$, which means that~$\sigma$ is the
		identity on the diagonal of~$X^{\times 2}$. A group with the
		conjugation operation $a \op b =
		b^{-1}ab$ yields an important example of quandles. A systematic study of
		self-distributivity dates back to Joyce~\cite{MR638121} and
		Matveev~\cite{MR672410}.
		\item A \emph{cycle set}, or \textit{right-cyclic quasigroup}, is a set~$X$
			with a binary operation~$\cdot$ satisfying 
			\begin{align}\label{E:Cyclic}
			(a\cdot b)\cdot (a\cdot c)=(b\cdot a)\cdot (b\cdot c)
			\end{align}
			and having all the left translations $a \mapsto
			b \cdot a$ bijective, the inverse operation being denoted by $a \mapsto  b
			\ast a$. As pointed out by Rump~\cite{MR2132760}, these
			give rise to involutive braidings $$\sigma(a,b) = ( (b \ast a) \cdot b, b
			\ast a),$$ and all braidings of a certain type  
			can be obtained this way.
	\end{itemize}

Racks and \textit{non-degenerate} cycle sets (i.e., for which the squaring map
$a \mapsto a \cdot a$ is bijective) can be included into the much more
general---and hence less understood---family of \textit{biracks}, introduced by
Fenn, Rourke, and Sanderson in~\cite{FRS_BirackHom}. These are sets with
invertible braidings which are \textit{left} and \textit{right non-degenerate},
i.e., their maps $a \mapsto a^b$ and $a \mapsto \prescript{b}{}{a}$ are bijective.

Important advances in knot-theoretic and Hopf-algebraic applications of
self-distributivity are due to the \textit{homological approach},
initiated by Fenn--Rourke--Sanderson \cite{RackHom} and
Carter--Jelsovsky--Kamada--Langford--Saito \cite{MR1990571}, and further developed
by Andruskiewitsch--Gra{\~n}a \cite{MR1994219}. Rack (co)homology theories were
generalized to the case of arbitrary braided sets by Carter--Elham\-dadi--Saito
\cite{MR2128041} and further developed by the first author~\cite{Lebed1}.
For biracks, Fenn, Rourke, and Sanderson \cite{FRS_BirackHom} constructed an
alternative, and more manageable, (co)homology theory, recently revived by
Ceniceros--El\-ham\-da\-di--Green--Nelson \cite{BirackHom}. The knot invariant
construction, which motivated the rack cohomology theory,
survived in all these generalized settings. Another application of these
cohomology theories is a construction of new examples of racks and braided sets via
an extension procedure using cocycles of low degree. Recently, the extension
techniques were adapted to cycle sets by the second
author~\cite{Vendramin_Ext}, resulting in counter-examples to several
conjectures concerning  involutive solutions to the Yang--Baxter equation.

The starting point of this paper is a study of the (co)homology of left non-degenerate (=LND) braided sets, including biracks. We introduce a coefficient version of the complex from~\cite{FRS_BirackHom} and extend it from biracks to all LND braided sets (Theorem~\ref{T:Birack}). It is then related to the complex from \cite{MR2128041, Lebed1} (recalled in Theorem~\ref{T:BrHom}) by an explicit isomorphism (Theorem~\ref{T:HomEquiv}). For a concrete braided set, one can thus choose between the two constructions the one more suitable for computations.

We also extend to LND braidings of a certain type the degenerate and normalized (co)homology constructions, given for quandles in~\cite{MR1990571} and for a more general class of biracks in~\cite{FRS_BirackHom}. These (co)homologies turn out to be related by a splitting theorem (Theorem~\ref{T:DegAndNormHom}), which generalizes the analogous result for quandles, obtained by Litherland and Nelson in~\cite{MR1952425}. 
More precisely, we establish a splitting theorem in the abstract context of (a skew version of) cubical homology (Theorem~\ref{PR:PreCub}), and then refine our concrete complexes into cubical structures. 

On the way we show how to associate to any LND braiding a shelf operation that captures many of its properties: invertibility, involutivity, the structure (semi)group, the action of positive braid monoids on its tensor powers, etc. (Propositions~\ref{PR:AssShelf} and~\ref{PR:Guitar}, Theorem~\ref{T:BijG}). This reduces the study of certain aspects of LND braided sets to that of shelves. 
An analogous construction for biracks was considered by Soloviev~\cite{MR1809284} and, in the more restricted case of braided groups, by Lu, Yan, and Zhu~\cite{MR1769723}. 

The last block of our results concerns cycle sets, whose (co)homologic\-al aspects remained unexplored until now. A cycle set is automatically an LND braided set (but not necessarily a birack!), which allows an application of our general (co)homology constructions described above. We provide a detailed analysis of cycle set extensions in terms of their cohomology groups (Theorem~\ref{T:Ext=H2}). Some explicit examples of extensions are given (Theorem~\ref{T:level_m}), implying the estimate $N_m \leqslant 2N_{m-1}$ for the minimal size~$N_m$ of square-free multipermutation cycle sets of level~$m$ (see Section~\ref{S:Mutliperm} for the definitions). We disprove the relation $N_m = 2^{m-1}+1$ conjectured by Cameron and Gateva-Ivanova \cite{MR2885602}, using a computer-aided computation of the~$N_m$ for small~$m$. 
 
Finally, we express the second cohomology of LND braided sets---in particular, cycle sets---in terms of the first group cohomology of their structure groups (Theorem~\ref{T:GroupCohom}), generalizing Etingof and Gra{\~n}a's result for racks~\cite{MR1948837}.      

Graphical tools play a central role in most of our constructions and proofs, making them more intuitive and concise. 

\section{Skew cubical structures and homology}\label{S:Cubical}  

The chain complexes we work with in this paper carry a much richer structure than a differential. This section reviews such enriched structures and establishes a homology splitting result for one of them.

\begin{defn}\label{D:cubical}
A \emph{pre-cubical structure} in a category~$\C$ consists of a family of objects $C_k$, $k \geqslant 0$, and of two families of morphisms $d^+_i, d^-_i \colon C_k \to C_{k-1}$ for $k \geqslant 1$, $1 \leqslant i \leqslant k$, satisfying the compatibility conditions
\begin{align}\label{E:PreCub}
d^\varepsilon_i d^\zeta_j &= d^\zeta_{j-1}d^\varepsilon_i \qquad \text{for all } i < j \:\text{and } \varepsilon,\zeta \in \{+,-\}.
\end{align}
These~$d_i$ are referred to as \emph{boundaries}. Such a structure is called \emph{weak skew cubical} if it also includes \emph{degeneracies} $s_i \colon C_k \to C_{k+1}$ for $k \geqslant 1$, $1 \leqslant i \leqslant k$, subject to relations
\begin{align}
d^\varepsilon_i s_j &= s_{j-1}d^\varepsilon_i & \text{for all } i < j \:\text{and } \varepsilon \in \{+,-\},\label{E:WeakCub}\\
d^\varepsilon_i s_j &= s_{j}d^\varepsilon_{i-1} & \text{for all } i > j+1 \:\text{and } \varepsilon \in \{+,-\},\label{E:WeakCub2}\\
d^\varepsilon_i s_i &= d^\varepsilon_{i+1}s_{i} & \text{for all } i \:\text{and } \varepsilon \in \{+,-\}.\label{E:WeakCub3}
\end{align}
A \emph{skew} (respectively, \emph{semi-strong skew}) \emph{cubical structure} satisfies moreover the property
\begin{align}
s_i s_j &= s_{j+1}s_i & \text{for all } i \leqslant j\label{E:WeakCubDeg}
\end{align}
and the upgraded version
\begin{align}
d^\varepsilon_i s_i &= d^\varepsilon_{i+1}s_{i} = \Id \label{E:WeakCub3'}
\end{align}
of condition~\eqref{E:WeakCub3} for all $\varepsilon \in \{+,-\}$ (respectively, for $\varepsilon = +$ only).
\end{defn}

In this paper we stick to a purely algebraic treatment of pre-cubical structures. For a topological interpretation, see the classical references \cite{SerreThesis,KanCubic,BH_cubes}. We sketch it in Remark~\ref{R:CubSkewcubSimpl} only. That remark also compares our structures with the much more classical cubical and simplicial ones.

The importance of the different types of structures we introduced is illustrated by the following result. For simplicity, it is stated for the category $\mathbf{Mod}_R$ of modules over a unital commutative ring~$R$.

\begin{thm}\label{PR:PreCub} 
Let $(C_k,\,d^+_i,\, d^-_i)$ be a pre-cubical structure in~$\mathbf{Mod}_R$.
\begin{enumerate}
\item The $R$-modules~$C_k$ endowed with the alternating sum maps 
\begin{align}\label{E:PreCubDiff}
\partial^{(\alpha, \beta)}_k &= \alpha\sum\nolimits_{i=1}^k (-1)^{i-1}d^+_i + \beta \sum\nolimits_{i=1}^k (-1)^{i-1} d^-_i
\end{align}  
form a chain complex for any $\alpha, \beta \in R$. 
\item If the boundaries $(d^+_i,\, d^-_i)$ can be completed with degeneracies~$s_i$, then the images $C^{\Degen}_k= \sum_{i=1}^{k-1}\im s_i$ form a sub-complex of $(C_k,\partial^{(\alpha, \beta)}_k)$ for any choice of $\alpha, \beta \in R$.
\item If the structure $(C_k,\, d^+_i,\, d^-_i,\, s_i)$ is moreover semi-strong skew cubical, then one has $R$-module decompositions
\begin{align}
C_k&= C^{\Degen}_k\oplus C^{\Norm}_k, \qquad\qquad C^{\Norm}_k = \im \eta_k, \label{E:SplittingGeneral}\\
\text{ where }\; \eta_k &= (\Id - s_1d^+_2)(\Id - s_2d^+_3)\cdots(\Id - s_{k-1}d^+_k).  \notag
\end{align} 
It yields a chain complex splitting for $(C_k,\partial^{(\alpha, \beta)}_k)$ for any $\alpha, \beta \in R$.
\end{enumerate}
\end{thm}

According to the theorem, decomposition~\eqref{E:SplittingGeneral} induces a decomposition in homology. We use the standard terminology for such splittings:

\begin{defn}
The ${\Degen}$- and ${\Norm}$- parts of the complexes and homology groups above are called \emph{degenerate} and, respectively, \emph{normalized}.
\end{defn}

\begin{proof}
The first two points are classical and can be verified by a straightforward computation. We give a detailed proof for the last point, which to our knowledge is new. Fix an arbitrary choice of $\alpha, \beta \in R$, and put $\partial_k = \partial^{(\alpha, \beta)}_k$.

We first show that the~$\eta_k$ form an endomorphism of the complex $(C_k,\partial_k)$. For this, it suffices to verify the relations 
\begin{align}\label{E:EtaComplexMap}
&\sum\nolimits_{i=1}^k (-1)^{i-1}d^\varepsilon_i \eta_k = \eta_{k-1}\sum\nolimits_{i=1}^k (-1)^{i-1}d^\varepsilon_i, \qquad \varepsilon \in \{+,-\}.
\end{align}
Put $p_i = \Id - s_{i}d^+_{i+1}$, and rewrite $\eta_k$ as $p_1 \cdots p_{k-1}$. The weak skew cubical axioms imply the following commutation rules for the~$p_i$ and the boundaries: 
\begin{align}
d^\varepsilon_i p_j = p_{j-1}d^\varepsilon_i, \quad i < j; \qquad\qquad &d^\varepsilon_i p_j = p_{j}d^\varepsilon_{i}, \quad i > j+1;\label{E:WeakCubP}\\
(d^\varepsilon_i-d^\varepsilon_{i+1})p_{i} = &d^\varepsilon_i-d^\varepsilon_{i+1}.\label{E:WeakCub3P}
\end{align}
Further, semi-strong skew cubical axioms imply the simplification rule
\begin{align}
p_id^\varepsilon_{i+1}p_{i+1} &= d^\varepsilon_{i+1}p_{i+1}.\label{E:WeakCub4P}
\end{align}
Indeed, this property rewrites as
\begin{align*}
s_{i}d^+_{i+1}d^\varepsilon_{i+1}(\Id-s_{i+1}d^+_{i+2}) &= 0,
\end{align*}
which follows from the computation
\begin{align*}
s_{i}d^+_{i+1}d^\varepsilon_{i+1}s_{i+1}d^+_{i+2} &\overset{\eqref{E:PreCub}}{=}s_{i}d^\varepsilon_{i+1}d^+_{i+2}s_{i+1}d^+_{i+2} \overset{\eqref{E:WeakCub3'}}{=} s_{i}d^\varepsilon_{i+1}d^+_{i+2} \overset{\eqref{E:PreCub}}{=} s_{i}d^+_{i+1}d^\varepsilon_{i+1}.
\end{align*}
Now, relations~\eqref{E:WeakCubP} allow one to rewrite the right-hand side of~\eqref{E:EtaComplexMap} as
\begin{align*}
p_1 &\cdots p_{k-2}\sum\nolimits_{i=1}^k (-1)^{i-1}d^\varepsilon_i \\
&= \sum\nolimits_{i=1}^{k-1} (-1)^{i-1} p_1 \cdots p_{i-1} d^\varepsilon_i p_{i+1} \cdots p_{k-1} + (-1)^{k-1}p_1 \cdots p_{k-2} d^\varepsilon_k.
\end{align*}
To conclude, we obtain the identical expression for the left-hand side of~\eqref{E:EtaComplexMap} by repeatedly using the following computation (with $s < k-1$):
\begin{align*}
\sum\nolimits_{i=s}^k& (-1)^{i-1} d^\varepsilon_i p_s \cdots p_{k-1} \\
\overset{\eqref{E:WeakCubP},\eqref{E:WeakCub3P}}{=}&\sum\nolimits_{i=s+2}^k (-1)^{i-1} p_s d^\varepsilon_i p_{s+1} \cdots p_{k-1} \\
&\qquad\qquad + ((-1)^{s-1} d^\varepsilon_s + (-1)^{s} d^\varepsilon_{s+1})p_{s+1} \cdots p_{k-1}\\ 
 \overset{\eqref{E:WeakCub4P}}{=} &(-1)^{s-1} d^\varepsilon_s p_{s+1} \cdots p_{k-1} + p_s (\sum\nolimits_{i=s+1}^k (-1)^{i-1} d^\varepsilon_i)p_{s+1} \cdots p_{k-1}.
\end{align*}

Both $C^{\Degen}_k = \sum \im s_i$ and $C^{\Norm}_k = \im \eta_k$ are thus sub-complexes of $(C_k,\partial^{(\alpha, \beta)}_k)$. It remains to establish the $R$-module decomposition $C_k= C^{\Degen}_k\oplus C^{\Norm}_k$. The definition of the map~$\eta_k$ directly gives the property 
$\im (\Id - \eta_k) \subseteq \sum\nolimits_i \im s_i$. 
Further, the semi-strong skew cubical axioms imply the following commutation rules for the~$p_i$ and the degeneracies: 
\begin{align*}
&s_i p_j = p_{j+1}s_i, \quad i \leqslant j; & &s_i p_j = p_{j}s_{i}, \quad i > j+1;\\
&p_is_i = 0; \qquad p_is_{i+1}=s_{i+1} - s_i; & & (s_i-s_{i+1})p_{i} = s_i-s_{i+1}.
\end{align*}
These relations imply that the map~$\eta_k$ vanishes on~$\im s_i$ for all $1 \leqslant i \leqslant k-1$. We conclude by applying to the data $C^{\Degen}_k \overset{\iota_k}{\hookrightarrow} C_k \overset{\eta_k}{\to} C_k$ (where~$\iota_k$ is the inclusion map) the following lemma.

\begin{lem}
Let $M,M'$ be two $R$-modules, and let $M' \overset{\alpha}{\to} M \overset{\beta}{\to} M$ be two $R$-linear maps satisfying the conditions $\beta\alpha = 0$ and $\im (\Id - \beta) \subseteq \im \alpha$. Then $M$ decomposes as $\im \alpha \oplus \im \beta$.
\end{lem}

\begin{proof}
Condition $\im (\Id - \beta) \subseteq \im \alpha$ implies that $\im \alpha + \im \beta$ covers the whole $R$-module~$M$. Let us show that this sum is direct. Relation $\beta\alpha = 0$ means $\im \alpha \subseteq \ker \beta$, implying $\im (\Id - \beta) \subseteq \im \alpha\subseteq \ker \beta$, which translates as $\beta^2=\beta$. But then the intersection $\ker \beta \cap \im \beta$ is zero, hence so is its sub-module $\im \alpha \cap \im \beta$. \qedhere\qedhere
\end{proof}
\end{proof}

\begin{rem}\label{R:CubSkewcubSimpl}
Eilenberg and MacLane~\cite{AcyclicModels} used the morphisms~$\eta_k$ to compare the full and the normalized versions of \emph{simplicial homology} (recall that simplicial structures are similar to skew cubical ones, except that they include only one family of boundaries~$d^+_i$). For \emph{cubical homology}, they employed the morphisms $\eta'_k= (\Id - s_1d_1^+)\cdots(\Id - s_{k}d_k^+)$ instead. Note that we use the classical notion of pre-cubical structure, but our skew cubical structures are different from the classical cubical ones. Concretely, a cubical structure bears $k+1$ degeneracies $s_1, \ldots, s_{k+1}$ on~$C_k$ while we stop at~$s_k$, and it satisfies conditions $d^\varepsilon_i s_i =  \Id$ and $d^\varepsilon_{i+1}s_{i} = s_{i}d^\varepsilon_{i}$ instead of~\eqref{E:WeakCub3'}. Topologically, cubical degeneracies correspond to compressing the unit cube in~$\R^n$ in the direction of one of the axes via the maps 
$$(x_1,\ldots,x_k) \mapsto (x_1, \ldots,x_{i-1},x_{i+1}, \ldots,x_k),$$ 
while our~$s_i$ can be thought of as squeezings onto the diagonal hyperplane $x_i = x_{i+1}$ via the map 
$$\varsigma_i \colon (x_1,\ldots,x_k) \mapsto (x_1, \ldots,x_{i-1},x_i+x_{i+1}-x_ix_{i+1}, \ldots,x_k).$$ 
This explains the word \emph{skew} in our terminology. More explicitly, together with the topological boundaries
\begin{align*}
\delta^+_i &\colon (x_1,\ldots,x_k) \mapsto (x_1, \ldots,x_{i-1},0,x_{i}, \ldots,x_k),\\
\delta^-_i &\colon (x_1,\ldots,x_k) \mapsto (x_1, \ldots,x_{i-1},1,x_{i}, \ldots,x_k),
\end{align*}
the~$\varsigma_i$ satisfy the relations dual to our semi-strong skew cubical axioms. Observe that for a cubical structure, the $C^{\Degen}_k$ form a sub-complex of $(C_k,\partial^{(\alpha, \beta)}_k)$ only for  $\beta=-\alpha$; no splitting of type~\eqref{E:SplittingGeneral} is known in this case. The structures we will treat in Section~\ref{S:Splitting} will be semi-strong skew cubical but neither skew cubical nor cubical.
\end{rem}

\section{Braided homology}\label{S:BrHom}  

Fix a braided set~$X$, with a braiding $\sigma \colon X^{\times 2} \to X^{\times 2}$, $(a,b) \mapsto (\prescript{a}{}{b}, a^b)$. In this section we recall and slightly extend the (co)homology theory for $(X, \, \sigma)$, developed in~\cite{Lebed1}. It is referred to as \textit{braided (co)homology} here.

We first comment on the \textit{graphical calculus}, which renders our constructions more intuitive. Braided diagrams represent here maps between sets, a set being associated to each strand; horizontal glueing corresponds to Cartesian product, vertical glueing to composition (which should be read from bottom to top), straight vertical lines to identity maps, crossings to the braiding~$\sigma$, and opposite crossings to its inverse (whenever it exists), as shown in Figure~\ref{P:YBE}\rcircled{A}. With these conventions, the Yang--Baxter equation~\eqref{E:YBE} for~$\sigma$ becomes the diagram from Figure~\ref{P:YBE}\rcircled{B}, which is precisely the braid- and knot-theoretic R$\mathrm{III}$ (= Reidemeister~$\mathrm{III}$) move. Associating \textit{colors} (i.e., arbitrary elements of the corresponding sets) to the bottom free ends of a diagram and applying to them the map encoded by the diagram, one determines the colors of the top free ends; Figure~\ref{P:YBE}\rcircled{A} contains a simple case of this process, referred to as \textit{color propagation}.        
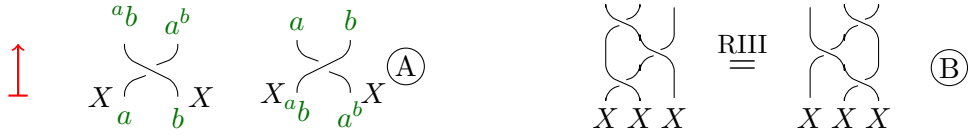
\begin{figure}[h]
\centering
\begin{tikzpicture}[scale=0.7]
\draw [rounded corners](0,0)--(0,0.25)--(0.4,0.4);
\draw [rounded corners](0.6,0.6)--(1,0.75)--(1,1);
\draw [rounded corners](1,0)--(1,0.25)--(0,0.75)--(0,1);
\node  at (0,-0.4) [mygreen]  {$a$};
\node  at (1,-0.4) [mygreen] {$b$};
\node  at (0,1) [mygreen,above] {$\prescript{a}{}{b}$};
\node  at (1,1) [mygreen,above] {$a^b$};
\node  at (0,0) [left] {$X$};
\node  at (1,0) [right] {$X$};
\draw [|->, red, thick]  (-2,0) -- (-2,1);
\node  at (2,0){ };
\end{tikzpicture}
\begin{tikzpicture}[scale=0.7]
\draw [rounded corners](0,0)--(0,0.25)--(1,0.75)--(1,1);
\draw [rounded corners](1,0)--(1,0.25)--(0.6,0.4);
\draw [rounded corners](0.4,0.6)--(0,0.75)--(0,1);
\node  at (0,-0.4) [mygreen] {$\prescript{a}{}{b}$};
\node  at (1,-0.4) [mygreen] {$a^b$};
\node  at (0,1) [mygreen,above] {$a$};
\node  at (1,1) [mygreen,above] {$b$};
\node  at (2,0.5){\rcircled{A}};
\node  at (0,0) [left] {$X$};
\node  at (1,0) [right] {$X$};
\node  at (5,0){ };
\end{tikzpicture}
\begin{tikzpicture}[xscale=0.45,yscale=0.4]
\draw [rounded corners](0,0)--(0,0.25)--(0.4,0.4);
\draw [rounded corners](0.6,0.6)--(1,0.75)--(1,1.25)--(1.4,1.4);
\draw [rounded corners](1.6,1.6)--(2,1.75)--(2,3);
\draw [rounded corners](1,0)--(1,0.25)--(0,0.75)--(0,2.25)--(0.4,2.4);
\draw [rounded corners](0.6,2.6)--(1,2.75)--(1,3);
\draw [rounded corners](2,0)--(2,1.25)--(1,1.75)--(1,2.25)--(0,2.75)--(0,3);
\node  at (0,0) [below] {$X$};
\node  at (1,0) [below] {$X$};
\node  at (2,0) [below] {$X$};
\node  at (4,1.5){\Large $\overset{\mathrm{RIII}}{=}$};
\end{tikzpicture}
\begin{tikzpicture}[xscale=0.45,yscale=0.4]
\draw [rounded corners](1,1)--(1,1.25)--(1.4,1.4);
\draw [rounded corners](1.6,1.6)--(2,1.75)--(2,3.25)--(1,3.75)--(1,4);
\draw [rounded corners](0,1)--(0,2.25)--(0.4,2.4);
\draw [rounded corners](0.6,2.6)--(1,2.75)--(1,3.25)--(1.4,3.4);
\draw [rounded corners](1.6,3.6)--(2,3.75)--(2,4);
\draw [rounded corners](2,1)--(2,1.25)--(1,1.75)--(1,2.25)--(0,2.75)--(0,4);
\node  at (0,1) [below] {$X$};
\node  at (1,1) [below] {$X$};
\node  at (2,1) [below] {$X$};
\node  at (4,2){\rcircled{B}};
\end{tikzpicture}
\caption{Color propagation through a crossing and its opposite, and the R$\mathrm{III}$ move representing the YBE.}\label{P:YBE}
\end{figure}

The role of coefficients in the braided homology will be played by the following structures:

\begin{defn}
	A \emph{right (braided) module} over a braided set $(X,\, \sigma)$ is a pair
	$(M, \,\rho)$, where $M$ is a set and $\rho \colon M \times X
	\to M$, $(m,a) \mapsto m \cdot a$, is a map compatible with~$\sigma$ in
	the sense of $$(m \cdot a) \cdot b = (m \cdot  \prescript{a}{}{b}) \cdot a^b$$ for all
	$m \in M,\, a,b \in X$. 
\emph{Left modules} $(N, \,\lambda \colon X \times N \to N)$ over $(X,\, \sigma)$ are defined similarly. See Figure~\ref{P:BrMod} for a diagrammatic version.
\end{defn}

\begin{figure}[h]
\centering
\begin{tikzpicture}[xscale=0.5,yscale=0.4]
 \draw [thick] (0,0) -- (0,2.5);
 \draw (1,0) -- (0,1);
 \draw (2,0) -- (0,2);
 \node at (0,2) [left]{$\rho$};
 \node at (0,1) [left]{$\rho$};
 \node at (2,0) [below] {$X$};
 \node at (1,0) [below] {$X$};
 \node at (0,0) [below] {$M$};
 \node  at (3.5,1.5){$=$};
\end{tikzpicture}
\begin{tikzpicture}[xscale=0.5,yscale=0.4]
 \node  at (-1.5,1.5){};
 \draw (1,0) -- (0,2);
 \draw [line width=4pt,white] (2,0) -- (0,1); 
 \draw (2,0) -- (0,1);
 \draw [thick] (0,0) -- (0,2.5);
 \node at (0,1) [left]{$\rho$};
 \node at (0,2) [left]{$\rho$};
 \node at (2,0) [below] {$X$};
 \node at (1,0) [below] {$X$};
 \node at (0,0) [below] {$M$};
 \node at (1,1) {$\sigma$}; 
\end{tikzpicture}
\begin{tikzpicture}[xscale=0.5,yscale=0.4]
 \node  at (-3.5,1.5){};
 \draw [thick] (3,0) -- (3,2.5);
 \draw (1,0) -- (3,2);
 \draw (2,0) -- (3,1);
 \node at (3,2) [right]{$\lambda$};
 \node at (3,1) [right]{$\lambda$};
 \node at (2,0) [below] {$X$};
 \node at (1,0) [below] {$X$};
 \node at (3,0) [below] {$N$};
 \node  at (5.5,1.5){$=$};
\end{tikzpicture}
\begin{tikzpicture}[xscale=0.5,yscale=0.4]
 \node  at (0.5,1.5){};
 \draw (1,0) -- (3,1); 
 \draw [line width=4pt,white] (2,0) -- (3,2);
 \draw [thick] (3,0) -- (3,2.5);
 \draw (2,0) -- (3,2);
 \node at (3,2) [right]{$\lambda$};
 \node at (3,1) [right]{$\lambda$};
 \node at (2,0) [below] {$X$};
 \node at (1,0) [below] {$X$};
 \node at (3,0) [below] {$N$};
 \node at (2,1) {$\sigma$}; 
\end{tikzpicture}
\caption{Right and left braided modules.}\label{P:BrMod}
\end{figure}
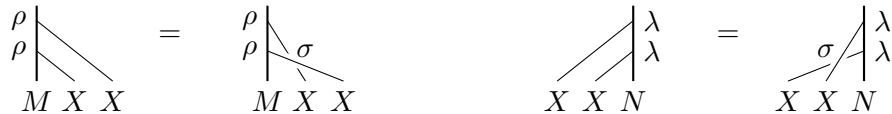

\begin{example}\label{EX:Adjoint}
The braided set $(X,\,\sigma)$ is a right and a left module over itself, with
the actions $\rho \colon (a,b) \mapsto a^b$ and $\lambda \colon (a,b) \mapsto
\prescript{a}{}{b}$. These modules are called \emph{adjoint}. More generally, any of the
powers $X^{\times n}$ is a right and a left module over $T(X)= \coprod_{i
\geqslant 0} X^{\times i}$, with the module structure adjoint to the extension
of the braiding~$\sigma$ to~$T(X)$, see Figure~\ref{P:TX}.
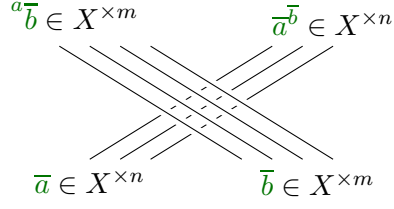
\begin{figure}[h]
\begin{tikzpicture}[xscale=0.4,yscale=0.25]
\draw (1,0)--(7,6);
\draw (2,0)--(8,6);
\draw (3,0)--(9,6);
\draw [line width=4pt,white] (6,0)--(0,6);
\draw [line width=4pt,white] (7,0)--(1,6);
\draw [line width=4pt,white] (8,0)--(2,6);
\draw [line width=4pt,white] (9,0)--(3,6);
\draw (6,0)--(0,6);
\draw (7,0)--(1,6);
\draw (8,0)--(2,6);
\draw (9,0)--(3,6);
\node [below] at (1,0) {${\color{mygreen}\overline{a}} \in X^{\times n}$};
\node [below] at (8.5,0) {${\color{mygreen}\overline{b}} \in X^{\times m}$};
\node [above] at (9,6) {${\color{mygreen}\overline{a}^{\overline{b}}} \in X^{\times n}$};
\node [above] at (0.5,6) {${\color{mygreen}\prescript{\overline{a}}{}{\overline{b}}} \in X^{\times m}$};
\end{tikzpicture}
\caption{The braiding of~$X$ extended to~$T(X)$.}\label{P:TX}
\end{figure}
\end{example}

\begin{example}\label{EX:Trivial}
The one-element set $I = \{\ast\}$ with the unique map $X \to I$ yields an example of a right and a left $(X,\,\sigma)$-module simultaneously. It is referred to as the \emph{trivial} right/left $(X,\,\sigma)$-module.
\end{example}

\begin{notation}\label{N:sigma_i}
   Let $(M,\,\rho)$ be a right module and $(N,\,\lambda)$ be
   a left module over a braided set $(X,\,\sigma)$.  We write 
\begin{align*}
   \sigma_i &= \Id_{M} \times \Id_{X}^{\times (i-1)} \times \sigma \times \Id_{X}^{\times (n-i-1)} \times \Id_N,\\
   \rho_0 &= \rho \times \Id_{X}^{\times (n-1)} \times \Id_N, \qquad \lambda_n = \Id_{M} \times \Id_{X}^{\times (n-1)} \times \lambda
\end{align*}
   (these are all maps from $M \times X^{\times n} \times N$ to $M \times X^{\times (n-1)}\times N$).
\end{notation}

The following result extends a construction from~\cite{Lebed1}:

\begin{thm}\label{T:BrHom}
   Let $(M,\,\rho)$ be a right module and $(N,\,\lambda)$ be a
    left module over a braided set $(X,\,\sigma)$. Consider the sets $C_n = M \times X^{\times n} \times N$.
    \begin{enumerate} 
        \item 
			The following maps form a pre-cubical structure on the~$C_n$:     
            \begin{align*}
                d_i^{l,+} &= \rho_0 \circ \sigma_1 \circ \cdots \circ \sigma_{i-1},
                & d_i^{r,-} &= \lambda_n \circ \sigma_{n-1} \circ \cdots \circ \sigma_{i}.
            \end{align*}
        \item 
            Now suppose that the braiding~$\sigma$ is invertible, and consider the maps
            \begin{align*}
                d_i^{l,-} &= \rho_0 \circ \sigma^{-1}_1 \circ \cdots \circ \sigma^{-1}_{i-1},
                & d_i^{r,+} &= \lambda_n \circ \sigma^{-1}_{n-1} \circ \cdots \circ \sigma^{-1}_{i}
            \end{align*}
            (Figure~\ref{P:BrHom}). For any choice of $\varepsilon,\zeta \in \{l,r\}$, the families $(d_i^{\varepsilon,+},\, d_i^{\zeta,-})$ with $n \geqslant 1$, $1 \leqslant i \leqslant n$, form a pre-cubical structure.
    \end{enumerate}
\end{thm}

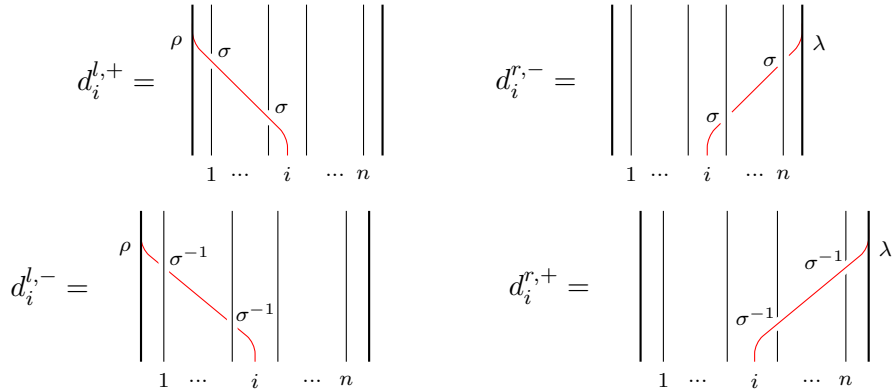
\begin{figure}[h]
\begin{tikzpicture}[scale=0.25]
 \draw [red,rounded corners] (-1,0) -- (-1,-1) -- (4,-6) -- (4,-7);
 \draw [ thick] (-1,1) --  (-1,-7);
 \draw [ thick] (9,1) --  (9,-7); 
 \draw (0,1) --(0,-1.6);
 \draw (0,-2.4) --  (0,-7);
 \node at (-5,-3) {$d_i^{l,+}=$};
 \node at (-1,-1)  [left] {$\scriptstyle \rho$};
 \draw (3,1) -- (3,-4.6);
 \draw (3,-5.4) -- (3,-7);
 \draw (5,1) -- (5,-7);
 \draw (8,1) -- (8,-7);
 \node at (-0.2,-2.3) [ above right] {$\scriptstyle \sigma$};
 \node at (2.8,-5.3) [above right] {$\scriptstyle \sigma$};
 \node at (0,-8) {$\scriptstyle 1$};
 \node at (1.5,-8) {$\scriptstyle \ldots$};
 \node at (4,-8) {$\scriptstyle i$};
 \node at (6.5,-8) {$\scriptstyle\ldots$};
 \node at (8,-8) {$\scriptstyle{n}$}; 
\end{tikzpicture}
\begin{tikzpicture}[scale=0.25]
 \draw [red,rounded corners] (9,0) -- (9,-1) -- (8.3,-1.7);
 \draw [red,rounded corners] (7.7,-2.3) -- (5.3,-4.7);
 \draw [red,rounded corners] (4.7,-5.3)--(4,-6) -- (4,-7);
 \draw [ thick] (-1,1) --  (-1,-7);
 \draw [ thick] (9,1) --  (9,-7);
 \draw (0,1) -- (0,-7);
 \node at (-12,-3) {};
 \node at (-5,-3) {$d_i^{r,-}=$};
 \node at (9,-1) [right] {$\scriptstyle \lambda$};
 \draw (3,1) -- (3,-7);
 \draw (5,1) -- (5,-7);
 \draw (8,1) -- (8,-7);
 \node at (5.2,-5) [left] {$\scriptstyle \sigma$};
 \node at (8.2,-2) [left] {$\scriptstyle \sigma$};
 \node at (0,-8) {$\scriptstyle 1$};
 \node at (1.5,-8) {$\scriptstyle \ldots$};
 \node at (4,-8) {$\scriptstyle i$};
 \node at (6.5,-8) {$\scriptstyle\ldots$};
 \node at (8,-8) {$\scriptstyle{n}$}; 
\end{tikzpicture}

\medskip
\begin{tikzpicture}[yscale=0.25, xscale=0.3]
 \draw [red,rounded corners] (-1,0) -- (-1,-1) -- (4,-6) -- (4,-7);
 \draw [ thick] (-1,1) --  (-1,-7);
 \draw [ thick] (9,1) --  (9,-7); 
 \node at (-5,-3) {$d_i^{l,-}=$};
 \node at (-1,-1)  [left] {$\scriptstyle \rho$};
 \draw [line width=4pt,white] (0,1) -- (0,-7); 
 \draw [line width=4pt,white] (3,1) -- (3,-7); 
 \draw (0,1) --  (0,-7); 
 \draw (3,1) -- (3,-7);
 \draw (5,1) -- (5,-7);
 \draw (8,1) -- (8,-7);
 \node at (-0.2,-2.4) [ above right] {$\scriptstyle\sigma^{-1}$};
 \node at (2.7,-5.4) [above right] {$\scriptstyle\sigma^{-1}$};
 \node at (0,-8) {$\scriptstyle 1$};
 \node at (1.5,-8) {$\scriptstyle \ldots$};
 \node at (4,-8) {$\scriptstyle i$};
 \node at (6.5,-8) {$\scriptstyle\ldots$};
 \node at (8,-8) {$\scriptstyle{n}$}; 
\end{tikzpicture}
\begin{tikzpicture}[yscale=0.25, xscale=0.3]
 \node at (-12,-3) {};
 \node at (-5,-3) {$d_i^{r,+}=$};
 \node at (9,-1) [right] {$\scriptstyle \lambda$};
 \draw (0,1) -- (0,-7); 
 \draw (2.8,1) -- (2.8,-7);
 \draw (5,1) -- (5,-7);
 \draw (8,1) -- (8,-7);
 \draw [line width=4pt,white,rounded corners] (9,0) -- (9,-1) -- (4,-6) -- (4,-7); 
 \draw [red,rounded corners] (9,0) -- (9,-1) -- (4,-6) -- (4,-7);
 \draw [ thick] (-1,1) --  (-1,-7);
 \draw [ thick] (9,1) --  (9,-7); 
 \node at (5.4,-4.6) [left] {$\scriptstyle\sigma^{-1}$};
 \node at (8.4,-1.6) [left] {$\scriptstyle\sigma^{-1}$};
 \node at (0,-8) {$\scriptstyle 1$};
 \node at (1.5,-8) {$\scriptstyle \ldots$};
 \node at (4,-8) {$\scriptstyle i$};
 \node at (6.5,-8) {$\scriptstyle\ldots$};
 \node at (8,-8) {$\scriptstyle{n}$}; 
\end{tikzpicture}
   \caption{Braided homology.}\label{P:BrHom}
\end{figure}

\begin{proof}
Conditions~\eqref{E:PreCub} are easily verified by diagram manipulations, using ambient isotopy, the third Reidemeister move, and the definition of braided modules (Figures \ref{P:YBE}-\ref{P:BrMod}).
\end{proof}

Theorem~\ref{PR:PreCub} now yields a collection of graphically defined differentials for a braided set, with coefficients in braided modules. Quite remarkably, they admit many alternative interpretations of completely different nature. For instance, choosing trivial modules (Example~\ref{EX:Trivial}) as coefficients and the values $\alpha = 1, \beta = -1$ as parameters, one gets the complex of Carter--Elham\-dadi--Saito \cite{MR2128041}. It was described topologically in terms of certain $n$-dimensional cubes, inspired by the preferred squares approach to rack spaces, due to Fenn--Rourke--Sanderson \cite{FRS_BirackHom,RackHom}. The cochain version of that complex can also be regarded as the diagonal part of Eisermann's Yang--Baxter cohomology, which controls deformations of the braided set \cite{Eisermann,Eisermann2}. Recently the braided (co)homology received two complementary interpretations, boasting new applications: one based on Rosso's quantum shuffle machinery \cite{Lebed1,LebedIdempot}, and one in terms of a special differential graded bialgebra of Farinati and Garc{\'{\i}}a-Galofre \cite{FarinatiGalofre}.

\begin{rem}
Theorem~\ref{T:BrHom} is easily transportable from the category of sets to a general monoidal category. Moreover, the categorical duality yields a cohomological version of our constructions. Finally, degeneracies can be built out of a comultiplication on~$X$ compatible with the braiding~$\sigma$. Details on these and other related points can be found in~\cite{Lebed1}.
\end{rem}

\section{Birack homology}\label{S:BirackHom} 

Recall that a \textit{birack} is a braided set whose braiding $(a,b) \mapsto (\prescript{a}{}{b},a^b)$ is invertible and \textit{non-degenerate}, i.e., the maps $a \mapsto a^b$ and $a \mapsto \prescript{b}{}{a}$
are bijections $X \overset{\sim}{\to} X$ for all~$b \in X$. 

\begin{notation}\label{N:sideways}
The inverses of the maps $a \mapsto a^b$ and $a \mapsto \prescript{b}{}{a}$ are denoted by $a \mapsto a^{\tilde{b}}$ and $a \mapsto \prescript{\tilde{b}}{}{a}$ respectively. We also use the notations
\begin{align*}
b \wdot a &= a^{\tilde{b}}, & a \cdot b &= \prescript{b \wdot a}{}{b}.
\end{align*}
\end{notation}

A homology theory for biracks was developed in~\cite{FRS_BirackHom,BirackHom}
as follows:

\begin{thm}\label{T:Birack}
Let $(X,\,\sigma)$ be a birack.
    \begin{enumerate}
        \item The maps $X^{\times n}\to X^{\times (n-1)}$ given by
            \begin{align*}
                d_i &\colon (x_1, \ldots, x_n) \mapsto (x_i \wdot  x_1, \ldots, x_i \wdot  x_{i-1}, x_i \cdot x_{i+1}, \ldots, x_i \cdot  x_n),\\
                d'_i & \colon (x_1, \ldots, x_n) \mapsto (x_1, \ldots, x_{i-1}, x_{i+1}, \ldots, x_n), 
            \end{align*}
            form a pre-cubical structure. 
        \item 
            The assertion remains true with the maps~$d_i$ replaced with
            \begin{align*}
                d^\star_i &\colon (x_1, \ldots, x_n) \mapsto (x_i \cdot x_1, \ldots, x_i \cdot x_{i-1}, x_i \wdot  x_{i+1}, \ldots, x_i \wdot   x_n).
            \end{align*}
    \end{enumerate}
    If the relation $a \cdot a = a \wdot a$ holds for all $a \in X$, then the maps
            \begin{align*}
                s_i &\colon (x_1,\ldots, x_n) \mapsto (x_1,\ldots, x_{i-1}, x_i, x_i, x_{i+1},\ldots, x_n)            			\end{align*}       
    enrich any of the two pre-cubical structures above into a cubical one.
\end{thm}

A proof by straightforward verifications is given in~\cite{BirackHom}, whereas
\cite{FRS_BirackHom} treats only the chain complex $(X^{\times n} ,\, \partial^{(1, -1)}_n)$ from Theorem~\ref{PR:PreCub} using its 
topological realization. We propose here a diagrammatic interpretation of the
boundary maps from the theorem, which will be instrumental in subsequent sections.

First, observe that the invertibility and the non-degeneracy of~$\sigma$ allow one to propagate colors through a crossing not only from bottom to top, but also from top to bottom (this corresponds to the map~$\sigma^{-1}$), from right to left (this is the map $(a^b,b) \mapsto (\prescript{a}{}{b},a)$, or, in our notations, $(a,b) \mapsto (a\cdot b,b \wdot a)$), and from left to right (this is the map $(\prescript{a}{}{b},a) \mapsto (a^b,b)$). The right-to-left versions of~$\sigma$ and~$\sigma^{-1}$ are presented in Figure~\ref{P:Side}. 
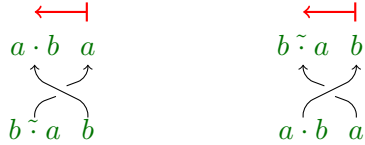
\begin{figure}[h]
\centering
\begin{tikzpicture}[scale=0.7]
\draw [rounded corners](0,0)--(0,0.25)--(0.4,0.4);
\draw [->,rounded corners](0.6,0.6)--(1,0.75)--(1,1);
\draw [->,rounded corners](1,0)--(1,0.25)--(0,0.75)--(0,1);
\node  at (0,-0.6) [mygreen,above]  {$b \wdot a$};
\node  at (1,-0.6) [mygreen,above]  {$b$};
\node  at (0,1) [mygreen,above] {$a \cdot b$};
\node  at (1,1) [mygreen,above] {$a$};
\draw [|->, red, thick]  (1,2) -- (0,2);
\node  at (4,2.5){ };
\end{tikzpicture}
\begin{tikzpicture}[scale=0.7]
\draw [->,rounded corners](0,0)--(0,0.25)--(1,0.75)--(1,1);
\draw [rounded corners](1,0)--(1,0.25)--(0.6,0.4);
\draw [->,rounded corners](0.4,0.6)--(0,0.75)--(0,1);
\node  at (0,-0.6) [mygreen,above] {$a \cdot b$};
\node  at (1,-0.6) [mygreen,above] {$a$};
\node  at (0,1) [mygreen,above] {$b \wdot a$};
\node  at (1,1) [mygreen,above] {$b$};
\draw [|->, red, thick]  (1,2) -- (0,2);
\end{tikzpicture}
\caption{Sideways maps.}\label{P:Side}
\end{figure}
These maps are fundamental in birack theory, and are often called \textit{sideways maps}. Note also that our treatment of crossings and their opposites validates the use of Reidemeister~$\mathrm{II}$ moves (Figure~\ref{P:BrInv}) in our diagrams. From now on strand orientations become relevant and are thus indicated in diagrams; all the strands in Figures~\ref{P:YBE}-\ref{P:BrHom} should be considered as oriented upwards.
          
\begin{figure}[h]
\centering
\begin{tikzpicture}[scale=0.55]
\draw [rounded corners](0,0)--(0,0.25)--(0.4,0.4);
\draw [rounded corners](0.6,0.6)--(1,0.75)--(1,1.25)--(0.6,1.4);
\draw [-<, rounded corners](0.4,1.6)--(0,1.75)--(0,2);
\draw [-<, rounded corners](1,0)--(1,0.25)--(0,0.75)--(0,1.25)--(1,1.75)--(1,2);
\node  at (2.3,1){\Large $\overset{\mathrm{RII}_1}{\rightleftarrows}$};
\node  at (2,0){ };
\end{tikzpicture}
\begin{tikzpicture}[scale=0.55]
\draw [->, rounded corners](0,0)--(0,2);
\draw [->, rounded corners](0.5,0)--(0.5,2);
\node  at (1.8,1){\Large $\overset{\mathrm{RII}_2}{\rightleftarrows}$};
\node  at (1,-0){ };
\end{tikzpicture}
\begin{tikzpicture}[scale=0.55]
\draw [rounded corners](1,0)--(1,0.25)--(0.6,0.4);
\draw [rounded corners](0.4,0.6)--(0,0.75)--(0,1.25)--(0.4,1.4);
\draw [-<, rounded corners](0.6,1.6)--(1,1.75)--(1,2);
\draw [-<, rounded corners](0,0)--(0,0.25)--(1,0.75)--(1,1.25)--(0,1.75)--(0,2);
\node  at (3,0){ };
\end{tikzpicture}
\begin{tikzpicture}[scale=0.55]
\draw [rounded corners](0,0)--(0,0.25)--(0.4,0.4);
\draw [rounded corners](0.6,0.6)--(1,0.75)--(1,1.25)--(0.6,1.4);
\draw [-<, rounded corners](0.4,1.6)--(0,1.75)--(0,2);
\draw [-<, rounded corners](1,2)--(1,1.75)--(0,1.25)--(0,0.75)--(1,0.25)--(1,0);
\node  at (2.3,1){\Large $\overset{\mathrm{RII}_3}{\rightleftarrows}$};
\node  at (2,0){ };
\end{tikzpicture}
\begin{tikzpicture}[scale=0.55]
\draw [->, rounded corners](0,0)--(0,2);
\draw [->, rounded corners](0.5,2)--(0.5,0);
\node  at (2.5,-0){ };
\end{tikzpicture}
\begin{tikzpicture}[scale=0.55]
\draw [<-, rounded corners](0,0)--(0,2);
\draw [<-, rounded corners](0.5,2)--(0.5,0);
\node  at (1.8,1){\Large $\overset{\mathrm{RII}_4}{\rightleftarrows}$};
\node  at (1,-0){ };
\end{tikzpicture}
\begin{tikzpicture}[scale=0.55]
\draw [>-, rounded corners](0,0)--(0,0.25)--(0.4,0.4);
\draw [rounded corners](0.6,0.6)--(1,0.75)--(1,1.25)--(0.6,1.4);
\draw [rounded corners](0.4,1.6)--(0,1.75)--(0,2);
\draw [>-, rounded corners](1,2)--(1,1.75)--(0,1.25)--(0,0.75)--(1,0.25)--(1,0);
\node  at (1,0){ };
\end{tikzpicture}
\caption{Reidemeister~$\mathrm{II}$ moves.}\label{P:BrInv}
\end{figure}
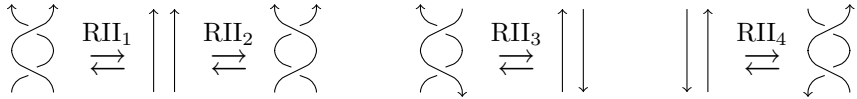

\newsavebox{\mybox}
\sbox{\mybox}{
\begin{tikzpicture}[scale=0.25,>=latex]
 \draw [->] (0,9) arc [radius=4, start angle=90, end angle= 450]; 
 \draw [line width=4pt,white] (0,2.5) circle [radius=4]; 
 \draw [->] (0,6.5) arc [radius=4, start angle=90, end angle= 450]; 
 \draw [->,red,rounded corners=15,dotted] (8.5,8) -- (8.5,10) -- (-4.5,10) -- (-4.5,-7.5) -- (8.5,-7.5) -- (8.5,8);
 \draw [line width=4pt,white] (0,-2.5) circle [radius=4]; 
 \draw [->] (0,1.5) arc [radius=4, start angle=90, end angle= 450];
 \node at (3.5,5) [right] {$\scriptstyle{{\color{red}x_3} \wdot x_1}$}; 
 \node at (3.5,2.5) [right] {$\scriptstyle{{\color{red}x_3} \wdot x_2}$}; 
 \node at (8.5,-6) [red,left] {$\scriptstyle{x_3}$}; 
 \node at (3.5,-2.5) [right] {$\scriptstyle{{\color{red}x_3} \cdot x_4}$};   
\end{tikzpicture}
}

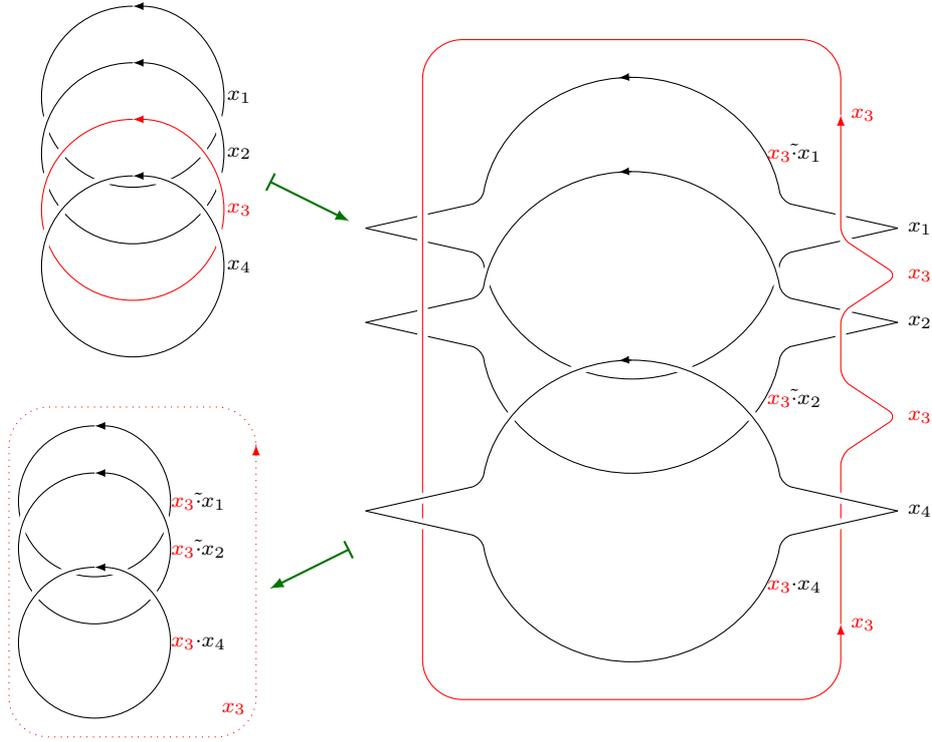
\begin{figure}[h]
\centering
\begin{tikzpicture}[scale=0.3,>=latex]
 \draw [->] (0,9) arc [radius=4, start angle=90, end angle= 450]; 
 \draw [line width=4pt,white] (0,2.5) circle [radius=4]; 
 \draw [->] (0,6.5) arc [radius=4, start angle=90, end angle= 450]; 
 \draw [line width=4pt,white] (0,0) circle [radius=4]; 
 \draw [->,red] (0,4) arc [radius=4, start angle=90, end angle= 450];
 \draw [line width=4pt,white] (0,-2.5) circle [radius=4]; 
 \draw [->] (0,1.5) arc [radius=4, start angle=90, end angle= 450];
 \node at (3.7,5) [right] {$\scriptstyle{x_1}$}; 
 \node at (3.7,2.5) [right] {$\scriptstyle{x_2}$}; 
 \node at (3.7,0) [red,right] {$\scriptstyle{x_3}$}; 
 \node at (3.7,-2.5) [right] {$\scriptstyle{x_4}$};    
 \draw [|->, mygreen, thick] (6,1.25)--(9.5,-0.5);
 \draw [|->, mygreen, thick] (9.5,-15)--(6,-16.75);
 \node at (0,-16) {\usebox{\mybox}}; 
\end{tikzpicture}
\begin{tikzpicture}[scale=0.5,>=latex]
 \draw [->] (0,9) arc [radius=4, start angle=90, end angle= 95]; 
 \path [draw, rounded corners] (0,9) arc (90:170:4) -- (-7,5);
 \path [draw, rounded corners] (0,1) arc (270:190:4) -- (-7,5); 
 \path [draw, rounded corners] (0,1) arc (270:350:4) -- (7,5); 
 \path [draw, rounded corners] (0,9) arc (90:10:4) -- (7,5);  
 \draw [line width=4pt,white] (0,2.5) circle [radius=4]; 
 \draw [->] (0,6.5) arc [radius=4, start angle=90, end angle= 95]; 
 \path [draw, rounded corners] (0,6.5) arc (90:170:4) -- (-7,2.5);
 \path [draw, rounded corners] (0,-1.5) arc (270:190:4) -- (-7,2.5); 
 \path [draw, rounded corners] (0,-1.5) arc (270:350:4) -- (7,2.5); 
 \path [draw, rounded corners] (0,6.5) arc (90:10:4) -- (7,2.5); 
 \draw [line width=4pt,white,rounded corners=10] (5.5,8) -- (5.5,10) -- (-5.5,10) -- (-5.5,0);
 \draw [line width=4pt,white,rounded corners] (5.5,-5.5) -- (5.5,-1) -- (7,0) -- (5.5,1) -- (5.5,2.75) -- (7,3.75) -- (5.5,4.75) -- (5.5,8);   
 \draw [->,red,rounded corners=15] (5.5,8) -- (5.5,10) -- (-5.5,10) -- (-5.5,-7.5) -- (5.5,-7.5) -- (5.5,-5.5);
 \draw [->,red,rounded corners] (5.5,-5.5) -- (5.5,-1) -- (7,0) -- (5.5,1) -- (5.5,2.75) -- (7,3.75) -- (5.5,4.75) -- (5.5,8);  
 \path [draw,line width=4pt,white, rounded corners] (0,1.5) arc (90:170:4) -- (-7,-2.5);
 \path [draw,line width=4pt,white, rounded corners] (0,-6.5) arc (270:190:4) -- (-7,-2.5); 
 \path [draw,line width=4pt,white, rounded corners] (0,-6.5) arc (270:350:4) -- (7,-2.5); 
 \path [draw,line width=4pt,white, rounded corners] (0,1.5) arc (90:10:4) -- (7,-2.5); 
 \draw [->] (0,1.5) arc [radius=4, start angle=90, end angle= 95]; 
 \path [draw, rounded corners] (0,1.5) arc (90:170:4) -- (-7,-2.5);
 \path [draw, rounded corners] (0,-6.5) arc (270:190:4) -- (-7,-2.5); 
 \path [draw, rounded corners] (0,-6.5) arc (270:350:4) -- (7,-2.5); 
 \path [draw, rounded corners] (0,1.5) arc (90:10:4) -- (7,-2.5); 
 \node at (7,5) [right] {$\scriptstyle{x_1}$}; 
 \node at (3.3,7) [right] {$\scriptstyle{{\color{red}x_3} \wdot x_1}$};  
 \node at (7,2.5) [right] {$\scriptstyle{x_2}$}; 
 \node at (3.3,0.5) [right] {$\scriptstyle{{\color{red}x_3} \wdot x_2}$};   
 \node at (7,0) [red,right] {$\scriptstyle{x_3}$}; 
 \node at (7,3.75) [red,right] {$\scriptstyle{x_3}$}; 
 \node at (5.5,8) [red,right] {$\scriptstyle{x_3}$};  
 \node at (5.5,-5.5) [red,right] {$\scriptstyle{x_3}$}; 
 \node at (7,-2.5) [right] {$\scriptstyle{x_4}$};
 \node at (3.3,-4.5) [right] {$\scriptstyle{{\color{red}x_3} \cdot x_4}$};   
 \node at (3.3,-8.5) {};     
\end{tikzpicture}
   \caption{A diagrammatic version of the boundary map~$d_i$: the $i$th circle inflates and then disappears. Here $n=4$, $i=3$.}\label{P:Birack_di}
\end{figure}

Now, consider the upper left diagram from Figure~\ref{P:Birack_di}. The colors $x_1, \ldots$ of its rightmost arcs (indicated in the diagram) can be propagated to the left, and uniquely determine the colors of all the remaining arcs. Probably the easiest way to see this is to start with $n$ horizontally aligned disjoint circles of the same size, colored by $x_1, \ldots$, and then to continuously bring them closer in the vertical direction, until they are piled up as on this diagram; during the stacking procedure some local Reidemeister~$\mathrm{II}$ moves occur, provoking local color changes but preserving the rightmost colors. Imagine then the $i$th circle inflating until it encloses the other ones, and then disappearing, as shown in the figure (where the circles are deformed for the sake of readability). R$\mathrm{II}$ and R$\mathrm{III}$ moves with induced local color changes happen during the inflation. From the figure one sees that the new colors of the rightmost arcs of the remaining circles yield the value of $d_i(x_1, \ldots)$.

The value of $d'_i(x_1, \ldots)$ is obtained by a similar procedure, except that the $i$th circle shrinks into the area which is interior to all the circles (Figure~\ref{P:Birack_di'}); the colors of the rightmost arcs are not affected by this procedure. Observe that changing the order of shrinking and/or inflation of different circles does not modify the colors in the resulting diagram. This implies relations~\eqref{E:PreCub}, and hence Theorem~\ref{T:Birack}. To switch from the~$d_i$ to the~$d^\star_i$, one should replace all the crossings in our diagrams with their opposites.

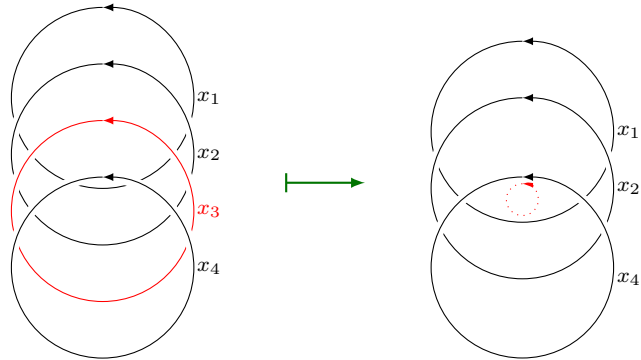
\begin{figure}[h]
\centering
\begin{tikzpicture}[scale=0.3,>=latex]
 \draw [->] (0,9) arc [radius=4, start angle=90, end angle= 450]; 
 \draw [line width=4pt,white] (0,2.5) circle [radius=4]; 
 \draw [->] (0,6.5) arc [radius=4, start angle=90, end angle= 450]; 
 \draw [line width=4pt,white] (0,0) circle [radius=4]; 
 \draw [->,red] (0,4) arc [radius=4, start angle=90, end angle= 450];
 \draw [line width=4pt,white] (0,-2.5) circle [radius=4]; 
 \draw [->] (0,1.5) arc [radius=4, start angle=90, end angle= 450];
 \node at (3.7,5) [right] {$\scriptstyle{x_1}$}; 
 \node at (3.7,2.5) [right] {$\scriptstyle{x_2}$}; 
 \node at (3.7,0) [red,right] {$\scriptstyle{x_3}$}; 
 \node at (3.7,-2.5) [right] {$\scriptstyle{x_4}$};    
 \draw [|->, mygreen, thick] (8,1.25)--++(3.5,0);
\end{tikzpicture}
\begin{tikzpicture}[scale=0.3,>=latex]
 \node at (-6,0) { }; 
 \draw [->] (0,9) arc [radius=4, start angle=90, end angle= 450]; 
 \draw [line width=4pt,white] (0,2.5) circle [radius=4]; 
 \draw [->] (0,6.5) arc [radius=4, start angle=90, end angle= 450]; 
 \draw [->,red,dotted] (0,2.7) arc [radius=0.7, start angle=90, end angle= 450];
 \draw [line width=4pt,white] (0,-1) circle [radius=4]; 
 \draw [->] (0,3) arc [radius=4, start angle=90, end angle= 450];
 \node at (3.7,5) [right] {$\scriptstyle{x_1}$}; 
 \node at (3.7,2.5) [right] {$\scriptstyle{x_2}$}; 
 \node at (3.7,-1.5) [right] {$\scriptstyle{x_4}$}; 
\end{tikzpicture}
   \caption{A diagrammatic version of the map~$d'_i$: the $i$th circle shrinks and then disappears.}\label{P:Birack_di'}
\end{figure}

\section{Associated shelves}\label{S:AssShelves}

This section contains a reminder from the zoology of braided sets. We
recall/introduce several algebraic structures (a shelf, a semigroup, and a
group) associated to every braiding, and capturing its properties. They will be
instrumental in further sections.

\begin{defn}\label{D:BraidedZoo}
A braiding $\sigma \colon X^{\times 2} \to X^{\times 2}, (a,b) \mapsto (\prescript{a}{}{b}, a^b)$ and the corresponding braided set are called
\begin{itemize}
\item \emph{left (or right) non-degenerate} if the map $a \mapsto a^b$ (respectively, $a \mapsto \prescript{b}{}{a}$) is a bijection $X \overset{\sim}{\to} X$;
\item \emph{non-degenerate} if both left and right non-degenerate;
\item \emph{involutive} if $\sigma\sigma = \Id$;
\item \emph{weakly R$\mathrm{I}$-compatible} if there exists a map $t \colon X \to X$ such that $\sigma(t(a),a) = (t(a),a)$ for all $a \in X$;
\item \emph{R$\mathrm{I}$-compatible} if weakly R$\mathrm{I}$-compatible with a bijective map~$t$.
\end{itemize}
\end{defn}

\begin{figure}[h]
\centering
\begin{tikzpicture}[scale=0.55]
\path [->,draw,rounded corners] (-1.1,1) arc (180:270:0.5) -- (-0.2,0.7) -- (0,1.1) -- (0,2); 
\path [draw,line width=4pt,white, rounded corners] (0,0) -- (0,0.9) -- (-0.2,1.3) -- (-0.6,1.5) arc (90:180:0.5) ; 
\path [draw,rounded corners] (0,0) -- (0,0.9) -- (-0.2,1.3) -- (-0.6,1.5) arc (90:180:0.5) ; 
\node  at (0.4,0.3){\color{mygreen} $x$};
\node  at (0.4,1.7){\color{mygreen} $x$};
\node  at (-1.9,1){\color{mygreen} $t(x)$};
\node  at (1.5,1){\Large $=$};
\end{tikzpicture}
\begin{tikzpicture}[scale=0.55]
\draw [->, rounded corners](0,0)--(0,2);
\node  at (0.4,0.3){\color{mygreen} $x$};
\node  at (0.4,1.7){\color{mygreen} $x$};
\node  at (1.5,1){\Large $=$};
\end{tikzpicture}
\begin{tikzpicture}[scale=0.55]
\path [draw,rounded corners] (0,0) -- (0,0.9) -- (0.2,1.3) -- (0.6,1.5) arc (90:0:0.5) ; 
\path [draw,line width=4pt,white, rounded corners] (1.1,1) arc (0:-90:0.5) -- (0.2,0.7) -- (0,1.1) -- (0,2); 
\path [->,draw,rounded corners] (1.1,1) arc (0:-90:0.5) -- (0.2,0.7) -- (0,1.1) -- (0,2); 
\node  at (-0.4,0.3){\color{mygreen} $x$};
\node  at (-0.4,1.7){\color{mygreen} $x$};
\node  at (2.3,1){\color{mygreen} $t^{-1}(x)$};
\end{tikzpicture}
\caption{Reidemeister~$\mathrm{I}$ move and the map~$t$.}\label{P:RI}
\end{figure}

Observe that the left (or right) non-degeneracy is equivalent to the sideways
map (or its right version) being well defined. Also note that the
R$\mathrm{I}$-compatibility is related to the color propagation through the
kink of a Reidemeister~$\mathrm{I}$ move (Figure~\ref{P:RI}), hence the name.

\begin{example}\label{EX:rack}
Recall that a \emph{shelf} is a set~$X$ with a self-distributive
operation~$\op$, in the sense of~\eqref{E:SD}. It gives rise to a braiding
$\sigma_{\op}(a,b) =  (b,a \op b)$, which is
\begin{itemize}
\item invertible if and only if the right translations $a \mapsto a \op b$ are bijective, i.e., $(X,\,\op)$ is a \emph{rack};
\item left non-degenerate if and only if $(X,\,\op)$ is a rack;
\item always right non-degenerate;
\item involutive if and only if the shelf is \emph{trivial}, i.e., $a \op b = a$ for all $a,b$;
\item (weakly) R$\mathrm{I}$-compatible if and only if one has $a \op a = a$ for all~$a$, implying $t = \Id_X$; in this case $(X,\,\op)$ is called a \emph{spindle}.
\end{itemize}
The notion of braided module over $(X,\, \sigma_{\op})$ recovers the classical notion of module over the shelf $(X,\,\op)$. Another braiding on~$X$ is defined by $\sigma'_{\op}(a,b) =  (b \op a,a)$. It should be thought of as the braiding~$\sigma_{\op}$ with the entries read from right to left. It has the same properties as~$\sigma_{\op}$, except that the right and left non-degeneracies change places.
\end{example}

\begin{example}\label{EX:birack}
A \emph{birack} can be seen as an invertible, left and right non-degenerate braided set. A birack---or, more generally, any left non-degenerate braided set---is weakly R$\mathrm{I}$-compatible if and only if one has  $a \cdot a = a \wdot a$ for all $a \in X$ (Notation~\ref{N:sideways}), and R$\mathrm{I}$-compatible if and only if the map $t \colon a \mapsto a \cdot a = a \wdot a$ is bijective (note that its injectivity follows from the left non-degeneracy, so for finite biracks weak and usual R$\mathrm{I}$-compatibility properties are equivalent).
\end{example}

\begin{example}\label{EX:CycleSet}
Recall that a \emph{cycle set} is a set~$X$ with an operation~$\cdot$ satisfying the cycle property~\eqref{E:Cyclic}, such that all the translations $a \mapsto b \cdot a$ admit inverses $a \mapsto  b \ast a$. Its associated braiding $\sigma_{\cdot}(a,b) = ( (b \ast a) \cdot b, b \ast a)$ is
\begin{itemize}
\item always involutive, left non-degenerate, and weakly R$\mathrm{I}$-compatible, with $t(a)= a \cdot a$;
\item right non-degenerate if and only if $(X,\,\cdot)$ is \emph{non-degenerate}, i.e., the squaring map $a \mapsto a \cdot a$ is bijective;
\item R$\mathrm{I}$-compatible if and only if $(X,\,\cdot)$ is non-degenerate.
\end{itemize}
\end{example}

\begin{example}\label{EX:monoid}
As noticed in~\cite{Lebed1}, any \emph{monoid}~$X$, with an associative operation $(a,b) \mapsto a \star b$ and a unit element $e \in X$, carries the following braiding:
$$\sigma_{\star} (a,b) = (e, a \star b).$$
Even better: the YBE for~$\sigma_{\star}$ is equivalent to the associativity of~$\star$, if one admits the unit property of~$e$. This braiding is
\begin{itemize}
\item almost never invertible, nor right non-degenerate, nor involutive, nor R$\mathrm{I}$-compatible;
\item left non-degenerate if and only if right translations are bijective, which holds for instance when $X$ is a group;
\item weakly R$\mathrm{I}$-compatible, with $t(a)=e$;
\item idempotent, in the sense of $\sigma_{\star}\sigma_{\star}=\sigma_{\star}$.
\end{itemize}
A module over our monoid is automatically a braided module over $(X,\, \sigma_{\star})$, with the same action.
\end{example}

We now show how to associate a shelf to any left non-degenerate braided set. 
Our result generalizes that of Soloviev~\cite{MR1809284}, see also~\cite[Prop. 5.4]{MR1994219}.
Recall the notations $b \wdot a$, $a \cdot b$ (Notation~\ref{N:sideways}) and the sideways map (Figure~\ref{P:Side}) for biracks, which still make sense in our more general context. Moreover, observe that the left non-degeneracy is sufficient for performing
\begin{itemize}
\item all the oriented versions of the R$\mathrm{II}$ move (Figure~\ref{P:BrInv}) in the direction~$\leftarrow$ of the equivalences~$\rightleftarrows$, and
\item the disentangling $\rightarrow$-directed moves R$\mathrm{II}_1$ and R$\mathrm{II}_3$ provided that for each strand, the colors of its lower and upper ends coincide.
\end{itemize}

\begin{defn}\label{D:AllowedRII}
These directed R$\mathrm{II}$ moves are called \textit{allowed}.
\end{defn}

\begin{pro}\label{PR:AssShelf}
\begin{enumerate}
    \item A left non-degenerate braided set $(X,\,\sigma)$ carries the
        following self-distributive operation (Figure~\ref{P:AssShelf}):
\begin{align*}
a \ops b = (b \cdot a)^b,
\end{align*}
\item It defines a rack structure if and only if~$\sigma$ is invertible.
\item The resulting shelf is trivial if and only if~$\sigma$ is involutive.
\item The following assertions are equivalent:
\begin{enumerate}
\item\label{I:a} The resulting shelf is a spindle.
\item\label{I:b} The braiding~$\sigma$ is weakly R$\mathrm{I}$-compatible.
\item\label{I:c} The braiding~$\sigma$ is weakly R$\mathrm{I}$-compatible with $t(a)=a\cdot a = a\wdot a$.
\end{enumerate}
\end{enumerate}
\end{pro}

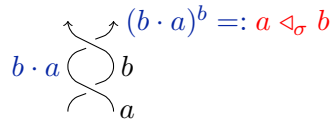
\begin{figure}[h]
\begin{tikzpicture}[scale=0.6]
\draw [rounded corners](0,0)--(0,0.25)--(0.4,0.4);
\draw [->, rounded corners](0.6,0.6)--(1,0.75)--(1,1.25)--(0,1.75)--(0,2);
\draw [rounded corners](1,0)--(1,0.25)--(0,0.75)--(0,1.25)--(0.4,1.4);
\draw [->, rounded corners](0.6,1.6)--(1,1.75)--(1,2);
\node  at (1.3,0){$a$};
\node  at (1.3,1){$b$};
\node  at (3.5,2){$\color{myblue} (b \cdot a)^b =: \color{red} a \ops b$};
\node  at (-0.7,1){$\color{myblue} b \cdot a$};
\end{tikzpicture}
\caption{The colors $a,b$ are propagated to the left, and then upwards; $a \ops
b$ is defined as the induced upper right color.}\label{P:AssShelf}
\end{figure} 

\begin{defn}
    The structure from the proposition is called the \emph{associated
    shelf/rack structure} for $(X,\,\sigma)$.
\end{defn}

The associated shelf operation can also be interpreted in terms of colored
circles (in the spirit of Figures~\ref{P:Birack_di}-\ref{P:Birack_di'}), as
shown in Figure~\ref{P:AssShelfCircles}. Note that this passing-through
procedure involves only allowed R$\mathrm{II}$ moves.

\begin{figure}[h]
\labellist
\small\hair 2pt
\pinlabel $b$ at 133 240
\pinlabel $a$ at 136 62

\pinlabel $b$ at 327 240
\pinlabel $a$ at 328 144
\pinlabel ${\color{myblue} a\cdot b}$ at 256 223

\pinlabel ${\color{myblue} a\cdot b}$ at 453 258
\pinlabel $b$ at 538 190

\pinlabel $b$ at 730 190
\pinlabel ${\color{myblue} \left( a\cdot b\right) ^b}$ at 658 370
\pinlabel ${\color{myblue} a\cdot b}$ at 658 208

\pinlabel $b$ at 924 190
\pinlabel ${\color{red} a\ops b}$ at 944 399
\pinlabel ${\color{myblue} =}$ at 944 368
\pinlabel ${\color{myblue} \left( a\cdot b\right) ^b}$ at 944 343

\pinlabel ${\color{mygreen} \mapsto}$ at 160 215
\pinlabel ${\color{mygreen} \mapsto}$ at 351 215
\pinlabel ${\color{mygreen} \mapsto}$ at 566 215
\pinlabel ${\color{mygreen} \mapsto}$ at 759 215

\endlabellist
\centering 
\hspace{5pt}
\includegraphics[scale=0.35]{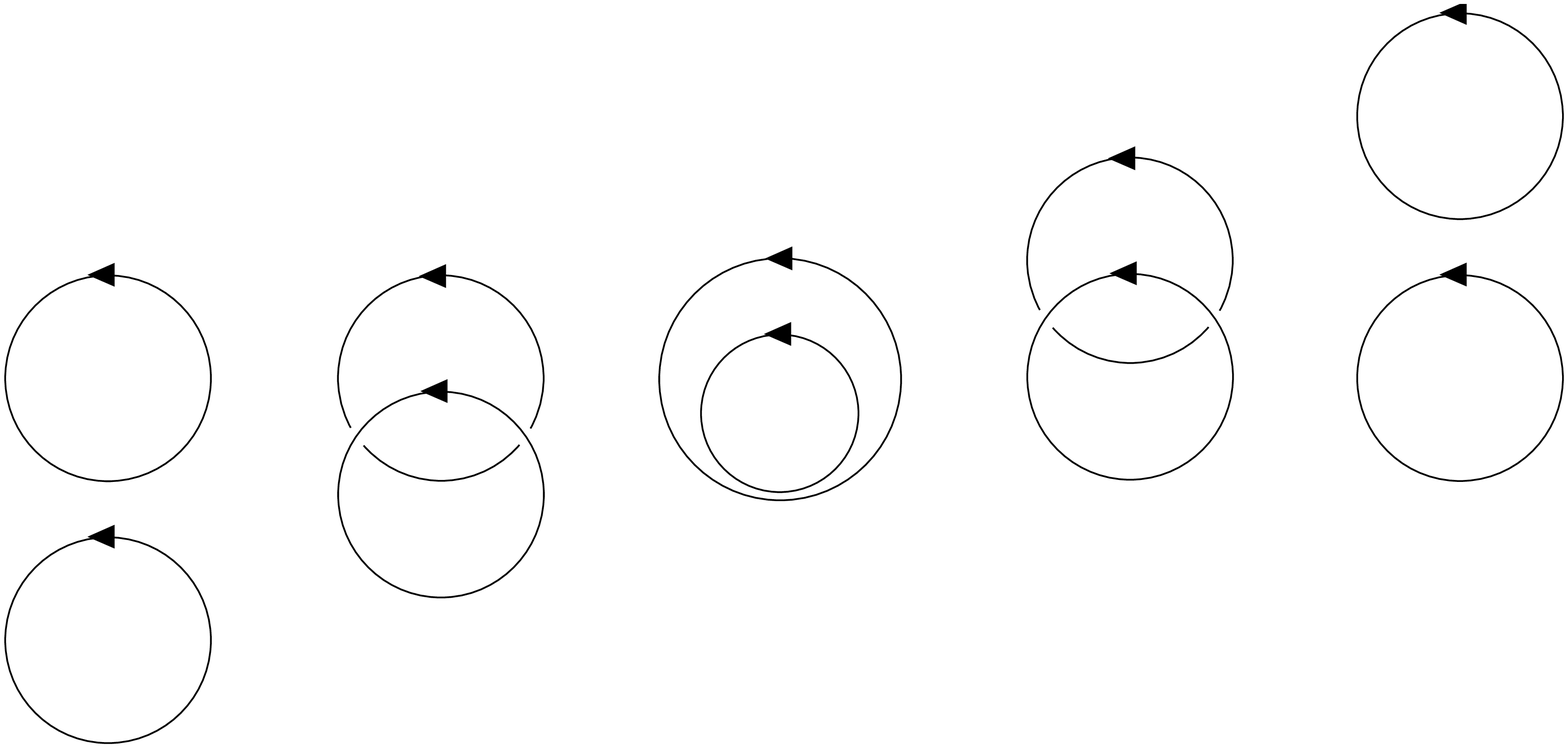}
\caption{The $a$-colored circle passes through the $b$-colored one; its color changes to~$a \ops b$.}\label{P:AssShelfCircles}
\end{figure} 

\begin{proof}[Proof of Proposition~\ref{PR:AssShelf}]
\begin{enumerate}
\item Figure~\ref{P:AssShelfProof} contains a diagrammatic proof of the self-distributivity of~$\ops$; only  allowed R$\mathrm{II}$ moves and nicely oriented R$\mathrm{III}$ moves are used. Certainly, algebraic manipulations also do the trick.
\begin{figure}[h]
\vspace{10pt}

\labellist
\small\hair 2pt

\pinlabel $c$ at 127 419
\pinlabel $b$ at 127 264
\pinlabel $a$ at 127 109

\pinlabel ${\color{myblue} \left( a\ops b \right)\ops c }$ at 265 450
\pinlabel $c$ at 332 309
\pinlabel ${\color{myblue} a \ops b}$ at 320 235
\pinlabel $b$ at 332 128
\pinlabel $a$ at 297 14

\pinlabel ${\color{myblue} \left( a \ops b \right)\ops c }$ at 462 450
\pinlabel ${\color{myblue} b \ops c}$ at 541 315
\pinlabel $c$ at 525 225
\pinlabel $b$ at 523 128
\pinlabel $a$ at 494 14

\pinlabel ${\color{red} \left( a\ops b \right)\ops c =}$ at 705 475
\pinlabel ${\color{red} \left( a\ops c \right)\ops \left( b \ops c \right)}$ at 705 445
\pinlabel ${\color{myblue} b \ops c}$ at 742 309
\pinlabel ${\color{myblue} a \ops c}$ at 710 225
\pinlabel $c$ at 724 128
\pinlabel $a$ at 689 14

\pinlabel ${\color{mygreen} \mapsto}$ at 160 215
\pinlabel ${\color{mygreen} \mapsto}$ at 362 215
\pinlabel ${\color{mygreen} \mapsto}$ at 566 215

\endlabellist
\centering 
\hspace{5pt}
\includegraphics[scale=0.4]{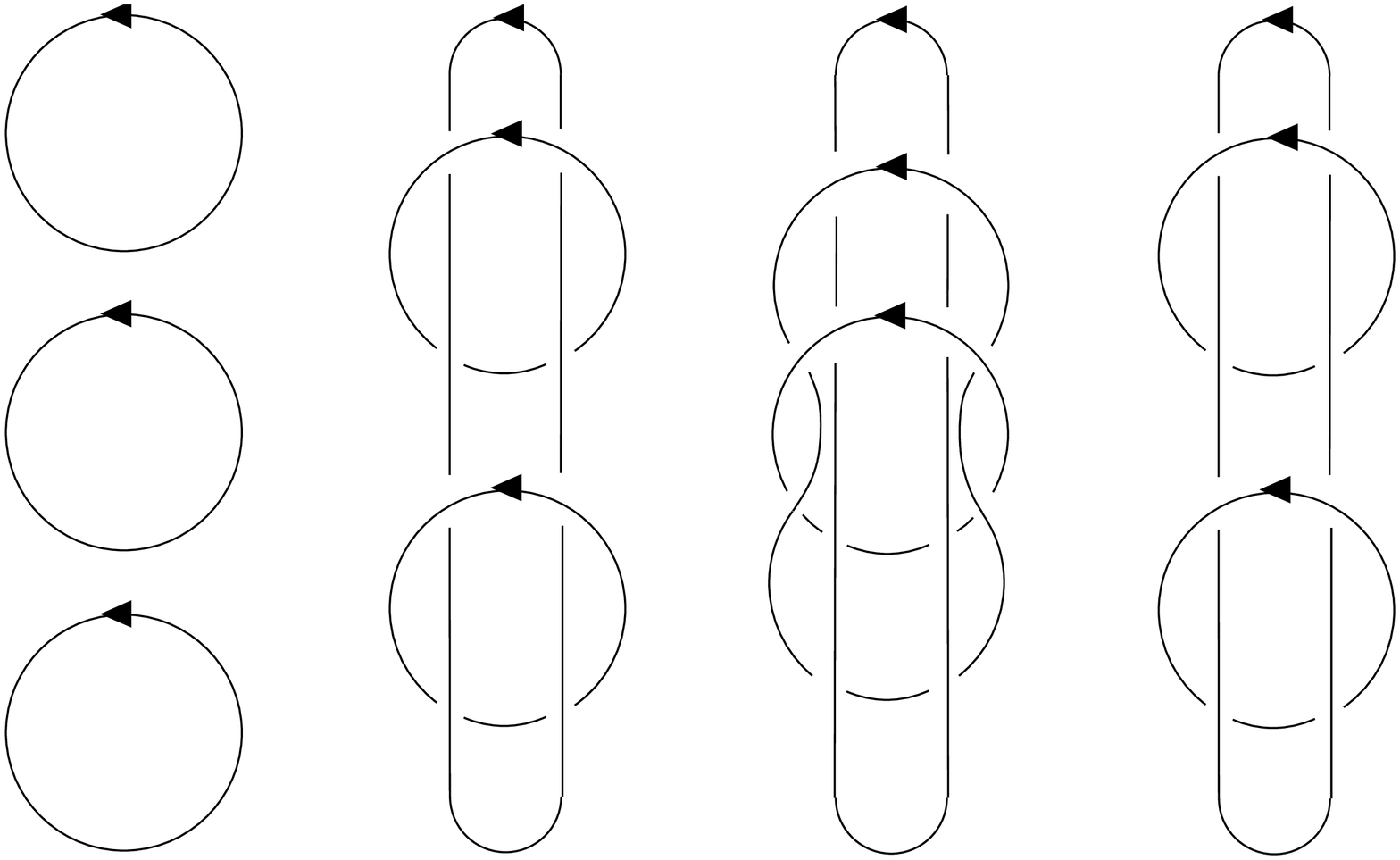}
\caption{A proof of $(a \ops b) \ops c = (a \ops c) \ops (b \ops c)$. The $b$-colored circle passes through the $c$-colored one, while both are encircling the elongated $a$-colored circle.}\label{P:AssShelfProof}
\end{figure} 

\item Suppose that~$\sigma$ is invertible. This implies that, for a given $b\in X$, the map $a \mapsto b \cdot a$ is bijective. The map $a \mapsto a^b$ being bijective by left non-degeneracy, so is $a \mapsto a \ops b = (b \cdot a)^b$. In the other direction, the bijectivity of $a \mapsto a \ops b$ and of $a \mapsto a^b$ yields the bijectivity of $a \mapsto b \cdot a$. Then $\sigma^{-1}$ is defined by $\sigma^{-1}(c,b)=(a \wdot b,a)$ where~$a$ is the unique element of~$X$ satisfying $b \cdot a = c$.

\item Suppose that $a \ops b = a$ for all $a,b$. Then the color of the upper left arc in Figure~\ref{P:AssShelf} has to be $a\cdot b$. Further, the color of the middle left arc is $b \cdot a$ when determined from the lower crossing, and $b \wdot a$ when determined from the upper one. Hence the operations~$\cdot$ and~$\wdot$ coincide, and from the figure one reads $\sigma^2(a \wdot b,a) = (a \wdot b,a)$, for all $a,b$. Since the map $b \mapsto a \wdot b = b^{\tilde{a}}$ is bijective, one obtains $ \sigma^2 = \Id$. The opposite direction is obvious.

\item Implications \ref{I:c} $\Rightarrow$ \ref{I:b} $\Rightarrow$ \ref{I:a} are clear. Let us prove \ref{I:a} $\Rightarrow$ \ref{I:c}. Relation $a \ops a = a$ means $(a \cdot a)^a = a$, or equivalently, using the left non-degeneracy, $a \cdot a = a^{\tilde{a}} = a \wdot a$, which implies $\sigma(a \cdot a, a) = (a \cdot a, a)$ as desired.
 \qedhere
\end{enumerate}
\end{proof}

\begin{example}
The self-distributive operation associated to the braiding~$\sigma_{\op}$ or~$\sigma'_{\op}$ for a shelf $(X,\,\op)$ is simply its original operation~$\op$.
\end{example}

\begin{example}
For the braiding coming from a right-invertible monoid
(Example~\ref{EX:monoid}), the associated operation is $a \ops b = b$ for all
$a,b$.
\end{example}

We finish this section by recalling how to associate a (semi)group to
a braiding. This construction appeared in~\cite{MR1722951} and~\cite{MR1637256}, and since then
became a key tool in the study of the YBE.

\begin{defn}\label{D:StructureSemiGroup}
The \emph{structure (semi)group of a braided set} $(X,\, \sigma)$ is the
(semi)group $G_{(X, \sigma)}$ (respectively, $SG_{(X, \sigma)}$), defined by
its generators $a \in X$ and relations $ab = \prescript{a}{}{b} a^b$, for all $a,b \in X$.
The \emph{structure (semi)group of a shelf} $(X,\, \op)$, denoted by $(S)G_{(X,
\op)}$, is simply the structure (semi)group of the associated braided set
$(X,\, \sigma'_{\op})$. Similarly, for a cycle set $(X,\, \cdot)$, one puts
$(S)G_{(X, \cdot)} = (S)G_{(X, \sigma_{\cdot})}$.
\end{defn}

The importance of these constructions comes, among others, from the following elementary property:

\begin{lem}\label{L:StructureSemiGroup}
For a right $(X,\, \sigma)$-module $(M,\, \cdot)$, the assignment $(m,a) \mapsto m \cdot a$, $a\in X$, $m \in M$ extends to a unique $SG_{(X, \sigma)}$-module structure on~$M$. If~$X$ acts on~$M$ by bijections (such modules are called \emph{solid}), then this assignment also defines a unique $G_{(X, \sigma)}$-module structure. This yields a bijection between (solid) $(X,\, \sigma)$-module and $SG_{(X, \sigma)}$- (respectively, $G_{(X, \sigma)}$-) module structures on~$M$. Analogous properties hold true for right modules.
\end{lem}

\begin{example}
In rack theory, the \emph{associated group} of a rack is a widely used notion. In our language, it is the group $G_{(X, \sigma_{\op})}$, isomorphic to $G_{(X, \op)}$ via the order-reversion map $(x_1,x_2,\ldots,x_n) \mapsto (x_n,\ldots,x_2,x_1)$. 
\end{example} 
  
\begin{example}
For the braided set associated to a monoid~$X$ (Example~\ref{EX:monoid}), the structure semigroup is isomorphic to the monoid itself, via the map sending $x_1\cdots x_n \in SG_X$ to the product $x_1 \star \cdots \star x_n \in X$.
\end{example}

\section{Guitar map}\label{S:Guitar}

This section is devoted to the remarkable guitar map and its applications to 
the study of braided sets and their structure (semi)groups.
Applications to homology will be treated in the next section.

\begin{defn}
The \emph{$n$-guitar map} for a braided set $(X,\, \sigma)$ is the map $J \colon X^{\times n} \to X^{\times n}$, defined by
\begin{align*}
J(\overline{x}) &= (J_1(\overline{x}),\ldots,J_n(\overline{x})), \qquad \overline{x} = (x_1, \ldots, x_n),\\
J_i(\overline{x}) &= x_i^{x_{i+1} \cdots x_n},
\end{align*}
with the notation $x_i^{x_{i+1} \cdots x_n} = (\ldots (x_i^{x_{i+1}})\ldots)^{x_{n}}$. We will often abusively talk about the \emph{guitar map}~$J$ meaning the family of $n$-guitar maps for all $n \geqslant 1$.
\end{defn}

The name comes from the resemblance of a diagrammatic version of the map~$J$ (Figure~\ref{P:Guitar}) with the position of guitar strings when a chord is played.

\begin{figure}[h]
\centering
\begin{tikzpicture}[xscale=0.9,>=latex]
 \draw [->, rounded corners=10] (0,0) -- (5,2.5) -- (4,3);
 \draw [line width=4pt,white, rounded corners=10] (1,0) -- (5,2) -- (3,3); 
 \draw [->, rounded corners=10] (1,0) -- (5,2) -- (3,3); 
 \draw [line width=4pt,white, rounded corners=10] (2,0) -- (5,1.5) -- (2,3); 
 \draw [->, rounded corners=10] (2,0) -- (5,1.5) -- (2,3);
 \draw [line width=4pt,white, rounded corners=10] (3,0) -- (5,1) -- (1,3); 
 \draw [->, rounded corners=10] (3,0) -- (5,1) -- (1,3); 
 \draw [line width=4pt,white, rounded corners=10] (4,0) -- (5,0.5) -- (0,3);   
 \draw [->, rounded corners=10] (4,0) -- (5,0.5) -- (0,3);  
 \node at (0,-0.4) [above] {$\scriptstyle{x_1}$};  
 \node at (2,-0.4) [above] {$\scriptstyle{\cdots}$}; 
 \node at (4,-0.4) [above] {$\scriptstyle{x_n}$}; 
 \node at (4.7,2.5) [right] {$\scriptstyle{J_1(\overline{x})}$};   
 \node at (4.7,2) [right] {$\scriptstyle{J_2(\overline{x})}$}; 
 \node at (5,1.25) [right] {$\;\scriptstyle{\vdots}$};  
 \node at (4.7,0.5) [right] {$\scriptstyle{J_n(\overline{x})}$};  
\end{tikzpicture} 
\hspace*{15pt}
\begin{tikzpicture}[xscale=0.9,>=latex]
 \draw [->, rounded corners=10] (0,0) -- (5,2.5) -- (4,3);
 \draw [line width=4pt,white, rounded corners=10] (1,0) -- (5,0.25)  -- (5.5,1.1) -- (5,2) -- (3,3); 
 \draw [->, rounded corners=10,red] (1,0) -- (5,0.25)  -- (5.5,1.1) -- (5,2) -- (3,3); 
 \draw [line width=4pt,white, rounded corners=10] (2,0) -- (5,1.5) -- (2,3); 
 \draw [->, rounded corners=10] (2,0) -- (5,1.5) -- (2,3);
 \draw [line width=4pt,white, rounded corners=10] (3,0) -- (5,1) -- (1,3); 
 \draw [->, rounded corners=10] (3,0) -- (5,1) -- (1,3); 
 \draw [line width=4pt,white, rounded corners=10] (4,0) -- (5,0.5) -- (0,3);   
 \draw [->, rounded corners=10] (4,0) -- (5,0.5) -- (0,3);  
 \node at (1,-0.4) [above] {$\scriptstyle{x_2}$};
 \node at (2,-0.4) [above] {$\scriptstyle{x_3}$};  
 \node at (3,-0.4) [above] {$\scriptstyle{x_4}$}; 
 \node at (4,-0.4) [above] {$\scriptstyle{x_5}$}; 
 \node at (4.9,2) [right] {$\scriptstyle{J_2(\overline{x})=}$}; 
 \node at (5.3,1.6) [right] {$\scriptstyle{x_2^{x_3 x_4 x_5}}$};    
\end{tikzpicture} 
   \caption{The guitar map (left) and a formula for calculating its components (right); the two diagrams are related by a sequence of R$\mathrm{III}$ moves.}\label{P:Guitar}
\end{figure}
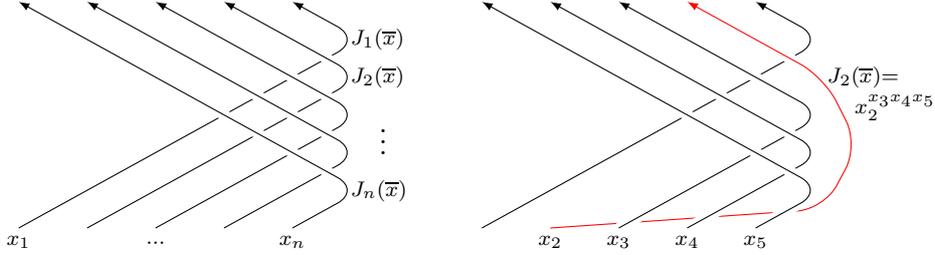

Recall the notation~$\sigma_i$ (Notation~\ref{N:sigma_i}), which we will use here with trivial coefficients. Also recall the extension of~$\sigma$ to $T(X)= \coprod_{i \geqslant 0} X^{\times i}$, and the right braided action $(\overline{a},\overline{b}) \mapsto \overline{a}^{\overline{b}}$ of $T(X)$ on itself (Example~\ref{EX:Adjoint}). One easily checks that both the extended braiding and its adjoint action descend from~$T(X)$ to the semigroup~$SG_{(X,\sigma)}$. For $\overline{a},\overline{b} \in T(X)$, their concatenation product will be denoted by~$\overline{a}\overline{b}$. 

\begin{pro}\label{PR:Guitar}
Let $(X,\,\sigma)$ be a left non-degenerate braided set, $\ops$ its associated shelf operation, and $\sigma'=\sigma'_{\ops} \colon (a,b) \mapsto (b \ops a,a)$ the braiding extracted from~$\ops$.
\begin{enumerate}
\item\label{I:Bij} The $n$-guitar map~$J$ for $(X,\,\sigma)$ is bijective for all $n \geqslant 1$.
\item\label{I:Cocycle} $J$ satisfies the cocycle property $J(\overline{a}\overline{b}) = (J(\overline{a})\action {\overline{b}}) J(\overline{b})$,  $\overline{a},\overline{b} \in T(X)$, where the action~$\action$ of $T(X)$ on itself is defined by
\begin{align}\label{E:FirstAction}
(a_1,\ldots,a_n) \action \overline{b} = (a_1^{\overline{b}},\ldots,a_n^{\overline{b}}).
\end{align} 
\item\label{I:AdjAction} $J$ entwines the adjoint action and the action~$\action$ above, in the sense of $J(\overline{a}^{\overline{b}}) = J(\overline{a})\action {\overline{b}}$.
\item\label{I:Entwine} $J$ entwines~$\sigma$ and~$\sigma'$, in the sense of $J \sigma_i = \sigma'_i J$.
\item\label{I:InducedAction} The operation~$\action$ induces an action of~$SG_{(X,\sigma)}$ on~$SG_{(X,\ops)}$.
\item\label{I:BijSG} $J$ induces a bijective cocycle $J^{SG} \colon SG_{(X,\sigma)} \overset{\sim}{\to} SG_{(X,\ops)}$, with the $SG_{(X,\sigma)}$-action on~$SG_{(X,\ops)}$ from the previous point. 
\end{enumerate}
\end{pro}

In practice, the braiding~$\sigma'$ and the semigroup $SG_{(X,\ops)}$ turn out to be much simpler than~$\sigma$ and~$SG_{(X,\sigma)}$. For instance, for an involutive~$\sigma$, one obtains the \textit{flip} $\sigma' \colon (a,b) \mapsto (b,a)$ (cf.~Proposition~\ref{PR:AssShelf}), and $SG_{(X,\ops)}$ becomes the free abelian semigroup on~$X$. This reduction to shelves simplifies the study of certain aspects of braided sets. For example, the proposition implies that the action on~$X^{\times n}$ of the positive braid monoid~$B_n^+$ (or, for invertible~$\sigma$, of the whole braid group~$B_n$) induced by~$\sigma$ is conjugated to the action induced by~$\ops$. Thus, as far as $B_n^{(+)}$-actions are concerned, left non-degenerate braided sets yield nothing new compared to shelves.

Historically, the map~$J$ seems to be first considered by Etingof, Schedler,
and Soloviev~\cite{MR1722951} for involutive braidings (and thus with $\sigma'$
being the flip). In their setting, they showed Point~\ref{I:Entwine} of the
proposition, without assuming the left non-degeneracy of~$\sigma$. 
Their construction was extended to non-involutive braidings by
Soloviev~\cite{MR1809284} and by Lu, Yan, and Zhu~\cite{MR1769723}; the latter
used a mirror version of the guitar map and denoted it by~$T_n$.
For cycle sets, Dehornoy~\cite{MR3374524} developed a right-cyclic calculus, used to
obtain short proofs for the existence of the Garside structure, of the
$I$-structure, and of Coxeter-like groups for structure groups $G_{(X,
\cdot)}$. Certain maps~$\Omega_i$, crucial in his calculus, can in fact be
expressed as $\Omega(x_1, \ldots,x_n) = J^{-1}(x_n, \ldots, x_1)$, and their
properties follow from those of the guitar map. In yet another particular
case, that of braidings associated to shelves, the guitar map was reintroduced
by Przytycki~\cite{Prz1} under the name ``the remarkable map $f$'', and used in
a study of multi-term distributive homology. 

\begin{proof}[Proof of Proposition~\ref{PR:Guitar}]
In Point~\ref{I:Bij}, the bijectivity of~$J$ is equivalent to the left non-degeneracy of~$\sigma$.  
Points~\ref{I:Cocycle} and~\ref{I:AdjAction} can be checked by playing with the diagrammatic definition of~$J$. 
Point~\ref{I:Entwine} is proved in Figure~\ref{P:GuitarBr}. 
In~\ref{I:InducedAction}, the only non-trivial property to check is $(\sigma'_i (\overline{c}))\action {\overline{b}} = \sigma'_i (\overline{c}\action {\overline{b}})$. Since~$J$ is bijective, we will verify it for the tuples~$\overline{c}$ of the form $J(\overline{a})$ only: 
\begin{align*}
(\sigma'_i J(\overline{a}))\action {\overline{b}} &= (J\sigma_i(\overline{a}))\action {\overline{b}} = J((\sigma_i(\overline{a}))^{\overline{b}}) = J\sigma_i(\overline{a}^{\overline{b}}) = \sigma'_i J(\overline{a}^{\overline{b}}) \\
&= \sigma'_i (J(\overline{a})\action {\overline{b}}).
\end{align*} 
The last point is a consequence of the preceding ones.
\end{proof}

\begin{figure}[h]
\centering
\begin{tikzpicture}[xscale=0.7,yscale=0.9,>=latex]
 \draw [->, rounded corners=10] (0,-0.5) -- (0,0) -- (5,2.5) -- (4,3);
 \draw [line width=4pt,white, rounded corners=10] (1,0) -- (5,2) -- (3,3); 
 \draw [->, rounded corners=10] (2,0.5) -- (5,2) -- (3,3);
 \draw [line width=4pt,white, rounded corners=10] (2,0) -- (5,1.5) -- (2,3); 
 \draw [->, rounded corners=10] (1,-0.5) -- (2,0) -- (5,1.5) -- (2,3); 
 \draw [line width=4pt,white, rounded corners=10] (4,0) -- (5,0.5) -- (0,3);   
 \draw [->, rounded corners=10] (4,-0.5) -- (4,0) -- (5,0.5) -- (0,3);  
 \draw [line width=4pt,white, rounded corners=10] (2,-0.5) -- (1,0) -- (2,0.5); 
 \draw [rounded corners=10] (2,-0.5) -- (1,0) -- (2,0.5); 
 \draw [myviolet, dashed] (-1,0)--(5,0); 
 \node at (-1,-0.5) [myviolet] {${\sigma_2}$};  
 \node at (-1,1.5) [myviolet]  {${J}$}; 
 \node at (0,-0.9) [above] {$\scriptstyle{x_1}$};  
 \node at (1,-0.9) [above] {$\scriptstyle{x_2}$}; 
 \node at (2,-0.9) [above] {$\scriptstyle{x_3}$}; 
 \node at (4,-0.9) [above] {$\scriptstyle{x_4}$};
 \node at (4.7,2.5) [right] {$\scriptstyle{\color{myblue}J_1(\sigma_2(\overline{x}))}$};   
 \node at (4.7,2) [right] {$\scriptstyle{\color{myblue}J_2(\sigma_2(\overline{x}))}$};   
 \node at (4.7,1.5) [right] {$\scriptstyle{\color{myblue}J_3(\sigma_2(\overline{x}))}$};  
 \node at (4.7,0.5) [right] {$\scriptstyle{\color{myblue}J_4(\sigma_2(\overline{x}))}$};  
\end{tikzpicture} 
\hspace*{20pt}
\begin{tikzpicture}[xscale=0.7,yscale=0.9,>=latex]
 \draw [->, rounded corners=10] (0,-0.5) -- (0,0) -- (5,2.5) -- (4,3);
 \draw [line width=4pt,white, rounded corners=10] (2,-0.5) -- (5,1) -- (2,2.5) -- (3,3);
 \draw [->, rounded corners=10] (2,-0.5) -- (5,1) -- (2,2.5) -- (3,3);
 \draw [line width=4pt,white, rounded corners=10] (2,0) -- (5,1.5) -- (2,3); 
 \draw [->, rounded corners=10] (1,-0.5) -- (2,0) -- (5,1.5) -- (2,3); 
 \draw [line width=4pt,white, rounded corners=10] (4,0) -- (5,0.5) -- (0,3);   
 \draw [->, rounded corners=10] (4,-0.5) -- (4,0) -- (5,0.5) -- (0,3);  
 \draw [line width=4pt,white, rounded corners=10] (4.75,1.125) -- (4.25,1.375);   
 \draw [rounded corners=10] (4.75,1.125) -- (4.25,1.375); 
 \draw [myviolet, dashed] (-0.5,2.5)--(5.5,2.5); 
 \node at (-0.5,0.5) [myviolet]  {${J}$}; 
 \node at (0,-0.9) [above] {$\scriptstyle{x_1}$};  
 \node at (1,-0.9) [above] {$\scriptstyle{x_2}$}; 
 \node at (2,-0.9) [above] {$\scriptstyle{x_3}$}; 
 \node at (4,-0.9) [above] {$\scriptstyle{x_4}$};
 \node at (4.7,2.25) [right] {$\scriptstyle{\color{red} J_1(\overline{x})}$};   
 \node at (4.7,1.5) [right] {$\scriptstyle{\color{red} J_2(\overline{x})}$};  
 \node at (4.7,1) [right] {$\scriptstyle{\color{red} J_3(\overline{x})}$};   
 \node at (4.7,0.5) [right] {$\scriptstyle{\color{red} J_4(\overline{x})}$};   
\end{tikzpicture}
\begin{tikzpicture}[xscale=0.7,yscale=0.9,>=latex]
 \draw [latex->, rounded corners=10,mygreen] (-2,3.5) -- (-1.5,3)node[below left]
{R$\mathrm{III}$} -- (-0.5,2.5);
 \draw [latex->, rounded corners=10,mygreen] (11.5,3.5) -- (11,3)node[below right]
{R$\mathrm{III}$} -- (10,2.5);
 \draw [->, rounded corners=10] (0,-0.5) -- (0,0) -- (5,2.5) -- (4,3);
 \draw [line width=4pt,white, rounded corners=10] (1,0) -- (5,2) -- (3,3); 
 \draw [->, rounded corners=10] (2,-0.5) -- (5,1) -- (4,1.5) -- (5,2) -- (3,3);
 \draw [line width=4pt,white, rounded corners=10] (2,0) -- (5,1.5) -- (2,3); 
 \draw [->, rounded corners=10] (1,-0.5) -- (2,0) -- (5,1.5) -- (2,3); 
 \draw [line width=4pt,white, rounded corners=10] (4,0) -- (5,0.5) -- (0,3);   
 \draw [->, rounded corners=10] (4,-0.5) -- (4,0) -- (5,0.5) -- (0,3);  
 \draw [line width=4pt,white, rounded corners=10] (4.75,1.125) -- (4.25,1.375);   
 \draw [rounded corners=10] (4.75,1.125) -- (4.25,1.375); 
 \node at (0,-0.9) [above] {$\scriptstyle{x_1}$};  
 \node at (1,-0.9) [above] {$\scriptstyle{x_2}$}; 
 \node at (2,-0.9) [above] {$\scriptstyle{x_3}$}; 
 \node at (4,-0.9) [above] {$\scriptstyle{x_4}$};
 \node at (4.7,2.5) [right] {$\scriptstyle{{\color{myblue}J_1(\sigma_2(\overline{x}))} = \color{red} J_1(\overline{x})}$};   
 \node at (4.7,2) [right] {$\scriptstyle{{\color{myblue}J_2(\sigma_2(\overline{x}))} = \color{red} J_3(\overline{x}) \ops  J_2(\overline{x})}$};   
 \node at (4.7,1.5) [right] {$\scriptstyle{{\color{myblue}J_3(\sigma_2(\overline{x}))} = \color{red} J_2(\overline{x})}$};  
 \node at (4.7,1) [right] {$\scriptstyle{\color{red} J_3(\overline{x})}$};   
 \node at (4.7,0.5) [right] {$\scriptstyle{{\color{myblue}J_4(\sigma_2(\overline{x}))} = \color{red} J_4(\overline{x})}$};   
\end{tikzpicture}
   \caption{The entwining relation $J \sigma_i = \sigma'_i J$ (here $n=4$, $i=2$) is established by comparing the colors of the rightmost arcs in the bottom diagram as calculated from the upper left (blue labels) and the upper right diagrams (red labels).}\label{P:GuitarBr}
\end{figure}
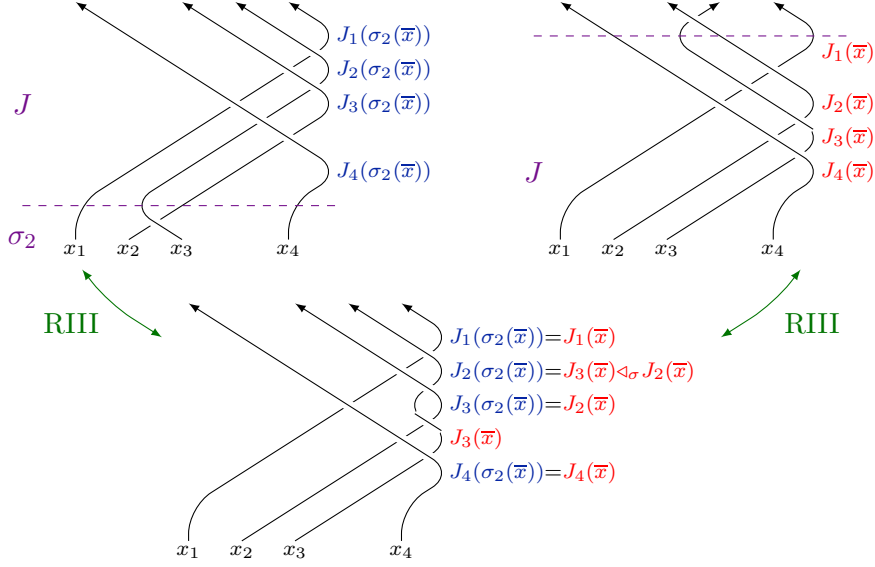

Our next aim is to upgrade the guitar map so that it induces a bijective cocycle $G_{(X,\sigma)} \overset{\sim}{\to} G_{(X,\ops)}$. This will be done for the case of a non-degenerate, invertible, and R$\mathrm{I}$-compatible braided set $(X,\,\sigma)$. In particular, one is now allowed to use all Reidemeister moves with all possible orientations.

Glue together two copies of~$X$ into $\oX = \{\,a^{+1},a^{-1} \,|\, a \in X\,\}$, called the \textit{double} of~$X$. The braiding~$\sigma$ extends to~$\oX$ via
\begin{itemize}
\item $\osigma(a^{+1},b^{+1}) = (c^{+1},d^{+1})$, where $\sigma(a,b)=(c,d)$;
\item $\osigma(a^{-1},b^{-1}) = (c^{-1},d^{-1})$, where $\sigma(d,c)=(b,a)$;
\item $\osigma(a^{+1},b^{-1}) = (c^{-1},d^{+1})$, where $\sigma(d,b)=(c,a)$;
\item $\osigma(a^{-1},b^{+1}) = (c^{+1},d^{-1})$, where $\sigma(a,c)=(b,d)$.
\end{itemize}
 This definition is best seen graphically: the diagram for~$\osigma$ on $(a_1^{\varepsilon_1}, a_2^{\varepsilon_2})$ is the one for~$\sigma$, with the orientation of the $i$th  strand reversed whenever $\varepsilon_i = -1$. The new braiding~$\osigma$ inherits the non-degeneracy, invertibility, and R$\mathrm{I}$-compatibility properties of~$\sigma$ (with $t(a^{+1}) = t(a)^{+1}$, and $t(a^{-1}) = (t^{-1}(a))^{-1}$).

We also need the \emph{toss map} $K\colon \oX^{\times n} \to \oX^{\times n}$ defined by $K(a^{+1})=a^{+1}$, $K(a^{-1})=t(a)^{-1}$, and $K(x_1^{\varepsilon_1}, \ldots, x_n^{\varepsilon_n}) = (K(x_1^{\varepsilon_1}), \ldots, K(x_n^{\varepsilon_n}))$. The change from~$a$ to~$t(a)$ for ``negatively oriented'' elements can be regarded as the color change happening to a negatively oriented circle when it flips and gets positive orientation (Figure~\ref{P:Toss}), hence the name.

\begin{figure}[h]
\labellist
\small\hair 2pt

\pinlabel $a$ at 122 113
\pinlabel $a$ at 240 113
\pinlabel ${\color{red}t(a)}$ at 555 113
\pinlabel ${\color{red}t(a)}$ at 788 113

\pinlabel ${\color{mygreen} \mapsto}$ at 183 65
\pinlabel ${\color{mygreen} \mapsto}$ at 596 65

\endlabellist
\centering 
\hspace{5pt}
\includegraphics[scale=0.3]{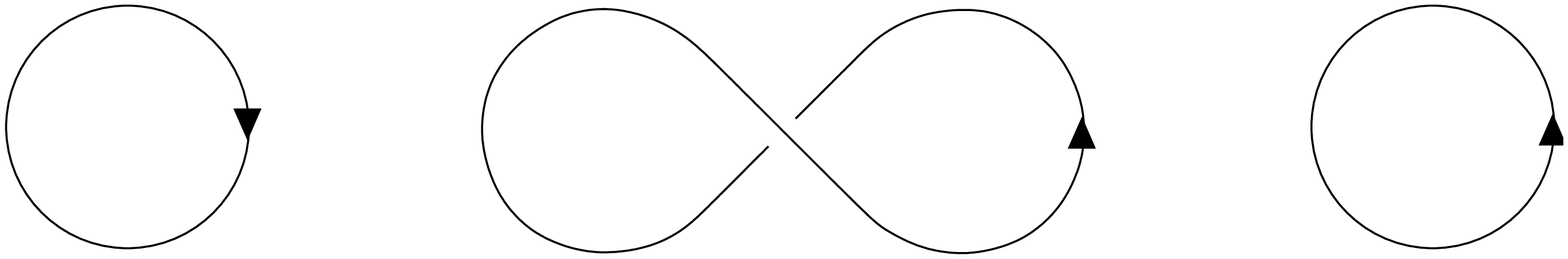}
   \caption{The toss map.}\label{P:Toss}
\end{figure}

\begin{thm}\label{T:BijG}
Let $(X,\,\sigma)$ be a non-degenerate invertible R$\mathrm{I}$-compatible braided set, $\ops$ its associated shelf operation, $(\oX,\,\osigma)$ its double, $J$ the guitar map of $(\oX,\,\osigma)$, and~$K$ the toss map. The map $\oJ = K J \colon \oX^{\times n} \to \oX^{\times n}$ induces a bijective cocycle $G_{(X,\sigma)} \overset{\sim}{\to} G_{(X,\ops)}$, where the group $G_{(X,\sigma)}$ act on $G_{(X,\ops)}$ via 
\begin{align}\label{E:SecondAction}
(a_1^{\varepsilon_1} \cdots a_n^{\varepsilon_n}) \oaction \overline{b} = (a_1^{\overline{b}})^{\varepsilon_1} \cdots (a_n^{\overline{b}})^{\varepsilon_n},
\end{align} 
for $a_i \in X$, $\varepsilon_i \in \{1,-1\}$, $\overline{b} \in G_{(X,\sigma)}$.
\end{thm} 

This result first appeared in~\cite{MR1722951} for involutive braidings, in~\cite{MR1769723} for braided groups, and in~\cite{MR1809284} in the general case. Its known proofs are rather technical and do not use the guitar map.

Note that the operation~$\oaction$ above is in general different from the
operation~$\action$ from~\eqref{E:FirstAction}, considered here for the braided
set $(\oX,\,\osigma)$.

\begin{proof}
Figure~\ref{P:JonInverse} shows that the map~$\oJ$ behaves well on the ``inverse pairs'', in the sense of
\begin{align*}
\oJ(a^{+1},a^{-1}) &= (t(a)^{+1},t(a)^{-1}), & \oJ(a^{-1},a^{+1}) &= (a^{-1},a^{+1}).
\end{align*}

\begin{figure}[h]
\labellist
\small\hair 2pt

\pinlabel ${\color{myviolet} J}$ at -20 44
\pinlabel $\scriptstyle{\color{myblue}a^{+1}}$ at 33 20
\pinlabel $\scriptstyle{\color{myblue}a^{-1}}$ at 130 20

\pinlabel ${\color{myviolet} K}$ at 297 26
\pinlabel $\scriptstyle{\color{myblue}a^{-1}}$ at 188 87
\pinlabel $\scriptstyle{\color{myblue}t(a)^{+1}}$ at 202 262
\pinlabel $\scriptstyle{\color{red}t(a)^{-1}}$ at 360 83
\pinlabel $\scriptstyle{\color{red}t(a)^{+1}}$ at 360 262

\pinlabel ${\color{myviolet} J}$ at 521 44
\pinlabel $\scriptstyle{\color{myblue}a^{-1}}$ at 571 20
\pinlabel $\scriptstyle{\color{myblue}a^{+1}}$ at 668 20

\pinlabel ${\color{myviolet} K}$ at 834 26
\pinlabel $\scriptstyle{\color{myblue}a^{+1}}$ at 726 87
\pinlabel $\scriptstyle{\color{myblue}\left( t^{-1}(a)\right) ^{-1}}$ at 766 262
\pinlabel $\scriptstyle{\color{red}a^{+1}}$ at 865 83
\pinlabel $\scriptstyle{\color{red}a^{-1}}$ at 870 262

\endlabellist
\centering 
\hspace{5pt}
\includegraphics[width=0.9\linewidth, height=0.25\linewidth]{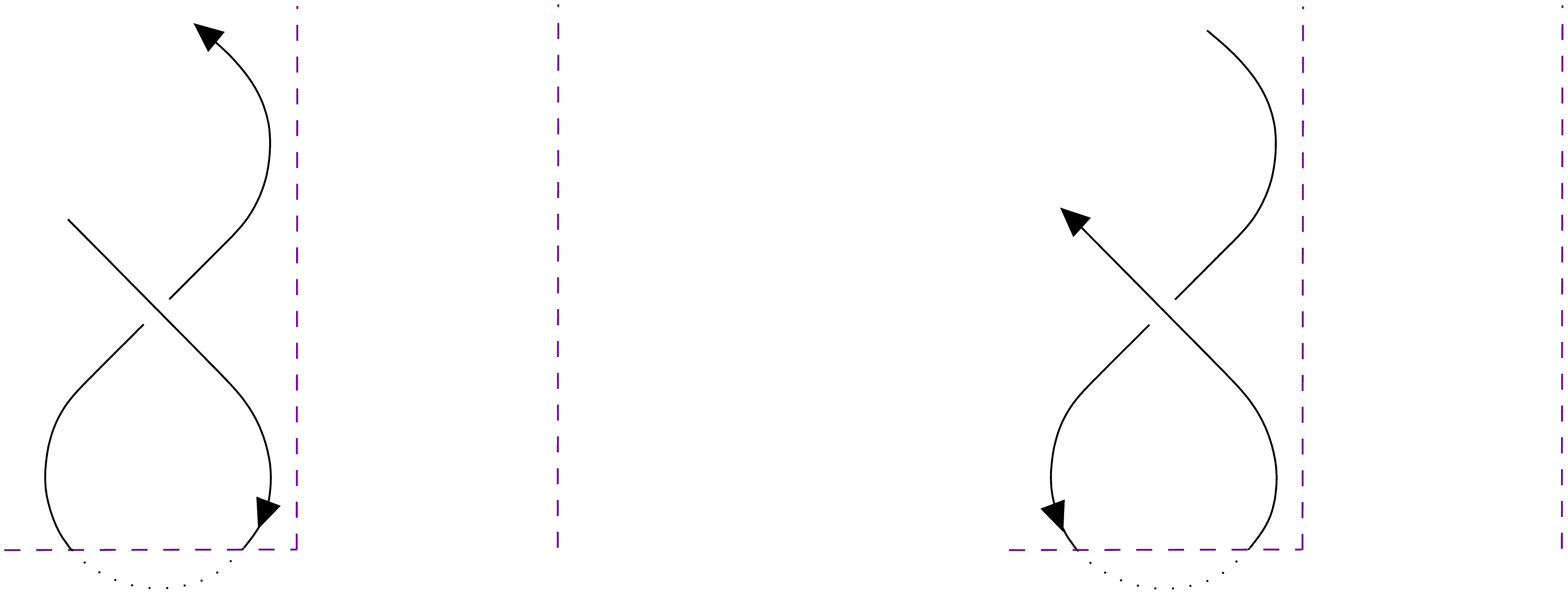}
   \caption{Computation of $\oJ(a^{+1},a^{-1})$ and $\oJ(a^{-1},a^{+1})$ in two steps: first applying the graphical guitar map~$J$, next modifying the result by the toss map~$K$.}\label{P:JonInverse}
\end{figure}

The adjoint right action also preserves the ``inverse pairs'': one has
\begin{align*}
(a^{+1},a^{-1})^{\overline{b}} &= (c^{+1},c^{-1}), & (a^{-1},a^{+1})^{\overline{b}} &= ((a^{\overline{b}})^{-1},(a^{\overline{b}})^{+1}),
\end{align*}
where $\overline{b}$ lives in~$\oX^{\times n}$, and $c \in X$ is defined via $c^{-1} = (a^{-1})^{\overline{b}}$ in~$\oX$. 
The second identity is proved in Figure~\ref{P:ActOnInverse}; the first one is similar.

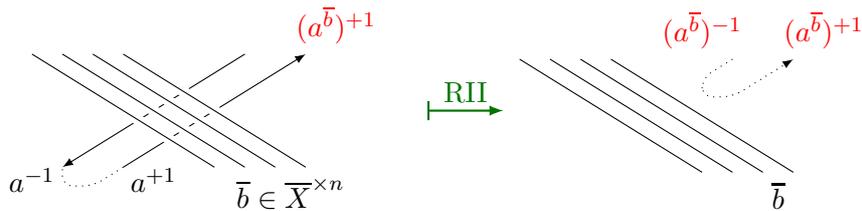
\begin{figure}[h]
\centering
\begin{tikzpicture}[xscale=0.4,yscale=0.25,>=latex]
\draw [->] (7,6)--(1,0);
\draw [->] (3,0)--(9,6);
\draw [dotted, rounded corners] (3,0)--(2,-1)--(1,-1)--(1,0);
\draw [line width=4pt,white] (6,0)--(0,6);
\draw [line width=4pt,white] (7,0)--(1,6);
\draw [line width=4pt,white] (8,0)--(2,6);
\draw [line width=4pt,white] (9,0)--(3,6);
\draw (6,0)--(0,6);
\draw (7,0)--(1,6);
\draw (8,0)--(2,6);
\draw (9,0)--(3,6);
\node [below] at (0,0.5) {$a^{-1}$};
\node [below] at (4,0.5) {$a^{+1}$};
\node [below] at (8.5,0) {$\overline{b} \in \oX^{\times n}$};
\node [above] at (10,6) {$\color{red} (a^{\overline{b}})^{+1}$};
\node [mygreen] at (14.25,4) {$\mathrm{RII}$};
 \draw [|->, mygreen, thick] (13,3)--++(2.5,0);
\end{tikzpicture}
\begin{tikzpicture}[xscale=0.4,yscale=0.25,>=latex]
\draw [->,dotted, rounded corners] (7,6)--(6,5)--(6,4)--(7,4)--(9,6);
\draw [line width=4pt,white] (6,0)--(0,6);
\draw [line width=4pt,white] (7,0)--(1,6);
\draw [line width=4pt,white] (8,0)--(2,6);
\draw [line width=4pt,white] (9,0)--(3,6);
\draw (6,0)--(0,6);
\draw (7,0)--(1,6);
\draw (8,0)--(2,6);
\draw (9,0)--(3,6);
\node [below] at (8.5,0) {$\overline{b}$};
\node [above] at (6,6) {$\color{red} (a^{\overline{b}})^{-1}$};
\node [above] at (10,6) {$\color{red} (a^{\overline{b}})^{+1}$};
\end{tikzpicture}
   \caption{Computation of the adjoint action $(a^{-1},a^{+1})^{\overline{b}}$ via the capping trick.}\label{P:ActOnInverse}
\end{figure}

The same graphical capping trick yields the properties
\begin{align}\label{E:ActionOfInv}
\overline{a}^{(b^{+1},b^{-1})} &= \overline{a} =  \overline{a}^{(b^{-1},b^{+1})}.
\end{align}

Now, using Proposition~\ref{PR:Guitar}, one has
\begin{align*}
\oJ(\overline{a}\overline{b})&= KJ(\overline{a}\overline{b}) = K((J(\overline{a}) \action \overline{b})J(\overline{b}))= K(J(\overline{a}^{\overline{b}})J(\overline{b}))= KJ(\overline{a}^{\overline{b}}) KJ(\overline{b})\\
&= \oJ(\overline{a}^{\overline{b}}) \oJ(\overline{b}),
\end{align*}
and thus
\begin{align*}
\oJ(\overline{a}b^{+1}b^{-1}\overline{d})&= \oJ((\overline{a}b^{+1}b^{-1})^{\overline{d}}) \oJ(\overline{d}) = \oJ(\overline{a}^{b^{+1}b^{-1}\overline{d}}) \oJ((b^{+1}b^{-1})^{\overline{d}}) \oJ(\overline{d}) \\
&= \oJ(\overline{a}^{\overline{d}}) \oJ(c^{+1}c^{-1}) \oJ(\overline{d}) = \oJ(\overline{a}^{\overline{d}}) t(c)^{+1}t(c)^{-1} \oJ(\overline{d}),
\end{align*}
where $c \in X$ is defined via $c^{-1} = (b^{-1})^{\overline{d}}$. Together with the invertibility of the maps~$J$, $K$, and~$t$ and of the adjoint action on~$T(\oX)$, this implies that the map $\oJ \colon T(\oX) \to T(\oX)$ survives when on both sides one mods out the relations $\overline{a}b^{+1}b^{-1}\overline{d}=\overline{a}\overline{d}$, $b \in X$, $\overline{a},\overline{d} \in \oX$. Relations $\overline{a}b^{-1}b^{+1}\overline{d}=\overline{a}\overline{d}$ can be treated analogously.  Moreover, $J$ entwines~$\sigma$ and $\sigma'=\sigma'_{\ops}$ (Proposition~\ref{PR:Guitar}), and~$K$ does not alter the elements $a^{+1} \in \oX$, so $\oJ = K J$ still survives when one mods out the relations $\overline{a}b^{+1}c^{+1}\overline{d}=\overline{a}\sigma(b^{+1},c^{+1})\overline{d}$ on the left and
$\overline{a}b^{+1}c^{+1}\overline{d}=\overline{a}\sigma'(b^{+1},c^{+1})\overline{d}$ on the right. 
Hence~$\oJ$ induces a bijection $G_{(X,\sigma)} \overset{\sim}{\to} G_{(X,\ops)}$. 

Next, formula~\eqref{E:SecondAction} defines an action of the semigroup $T(\oX)$ on itself. It behaves well with respect to the inverse pairs: the property 
\begin{align*}
(a_1^{\varepsilon_1},\ldots,a_i^{\varepsilon_i}, a_i^{-\varepsilon_i},\ldots, a_n^{\varepsilon_n}) \oaction \overline{b} = ((a_1^{\overline{b}})^{\varepsilon_1},\ldots, (a_i^{\overline{b}})^{\varepsilon_i}, (a_i^{\overline{b}})^{-\varepsilon_i},\ldots  (a_n^{\overline{b}})^{\varepsilon_n})
\end{align*} 
is clear from the definition, and the property 
\begin{align*}
\overline{a} \oaction (\overline{b}c^{\varepsilon}c^{-\varepsilon}\overline{d}) = \overline{a} \oaction (\overline{b}\overline{d})
\end{align*}
follows from~\eqref{E:ActionOfInv}. Nice behavior with respect to the braidings~$\sigma$ and~$\sigma'$ is proved in the same way as for the action~$\action$ in Proposition~\ref{PR:Guitar}. Altogether, this shows that~$\oaction$ induces an action of $G_{(X,\sigma)}$ on $G_{(X,\ops)}$. It remains to check for~$\oJ$ the cocycle property with respect to this action. 
It will follow from the relation $\oJ(\overline{a}\overline{b})= \oJ(\overline{a}^{\overline{b}}) \oJ(\overline{b})$ from the previous paragraph if we manage to prove the identity 
\begin{align}\label{E:ActionEntw}
\oJ(\overline{a}^{\overline{b}}) &= \oJ(\overline{a}) \oaction \overline{b}.
\end{align} 
The relation $J(\overline{a}^{\overline{b}}) = J(\overline{a}) \action \overline{b}$ from Proposition~\ref{PR:Guitar} reduces~\eqref{E:ActionEntw} to 
\begin{align*}
K(J(\overline{a}) \action \overline{b}) &= (KJ(\overline{a})) \oaction \overline{b}.
\end{align*}
 The toss map~$K$ and the operations~$\action$ and~$\oaction$ acting component-wise, it suffices to consider the relation $K(a^{+1} \action \overline{b}) = K(a^{+1}) \oaction \overline{b}$, $a \in X$, in which case both sides equal $(a^{\overline{b}})^{+1}$; and $K(a^{-1} \action \overline{b}) = K(a^{-1}) \oaction \overline{b}$, which translates as $t(a)^{\overline{b}} = t(c)$, where $c \in X$ is defined via $c^{-1} = (a^{-1})^{\overline{b}}$. This latter property is verified graphically in Figure~\ref{P:KinkTrick}. \qedhere

\begin{figure}[h]
\begin{tikzpicture}[xscale=0.8,>=latex]
 \draw [rounded corners=10,>-] (-1,0.6) -- (-1,0.9) -- (-0.3,1.5) -- (0.5,1.7) -- (0.5,1.25);
 \draw [line width=5pt,white, rounded corners=10] (0.5,1.25) -- (0.5,0.8) -- (-0.3,1) -- (-1,1.9) -- (-1,2.5);
 \draw [rounded corners=10,<-] (0.5,1.25) -- (0.5,0.8) -- (-0.3,1) -- (-1,1.9) -- (-1,2.5);
 \draw [line width=5pt,white, rounded corners=10]  (1,0.6) -- (1,1.7) -- (-1.5,2.2) -- (-1.5,2.5);
 \draw [ultra thick, rounded corners=10,-<] (1,0.6) -- (1,1.7) -- (-1.5,2.2) -- (-1.5,2.5);
 \node  at (2.5,1.5){\Large \color{mygreen} $\overset{\mathrm{RII}}{\longleftrightarrow}$};
 \node at (-1,0.4) {$\scriptstyle{a^{-1}}$}; 
 \node at (1.1,0.4) {$\scriptstyle{\overline{b}}$}; 
 \node at (-0.9,2.7) {$\scriptstyle{c^{-1}}$}; 
 \node at (-1.2,1.6) {$\scriptstyle{a^{-1}}$}; 
 \node at (0.4,0.7) {$\scriptstyle{t(a)}$}; 
\end{tikzpicture}
\begin{tikzpicture}[xscale=0.8,>=latex]
 \draw [rounded corners=10,>-] (-1,0.3) -- (-1,0.6) -- (-0.3,1.5) -- (0.5,1.7) -- (0.5,1.25);
 \draw [line width=5pt,white, rounded corners=10] (0.5,1.25) -- (0.5,0.8) -- (-0.3,1) -- (-1,1.9) -- (-1,2.2);
 \draw [rounded corners=10,<-] (0.5,1.25) -- (0.5,0.8) -- (-0.3,1) -- (-1,1.9) -- (-1,2.2);
 \draw [line width=5pt,white, rounded corners=10] (1,0.3) -- (1,1.2) -- (-1.5,1.8) -- (-1.5,2.5);
 \draw [ultra thick, rounded corners=10,->] (1,0.3) -- (1,1.2) -- (-1.5,1.8) -- (-1.5,2.5);
 \node  at (2.5,1.2){\Large \color{mygreen} $\overset{\mathrm{RII}}{\longleftrightarrow}$};
 \node at (-1,0.1) {$\scriptstyle{a^{-1}}$}; 
 \node at (1.1,0.1) {$\scriptstyle{\overline{b}}$}; 
 \node at (-0.9,2.4) {$\scriptstyle{c^{-1}}$}; 
 \node at (0.4,0.7) {$\scriptstyle{t(a)}$}; 
 \node at (1,1.7) [myblue] {$\scriptstyle{t(a)^{\overline{b}}}$}; 
\end{tikzpicture}
\begin{tikzpicture}[xscale=0.8,>=latex]
 \node  at (-0.5,1.2){};
 \draw [rounded corners=10,<-] (-1,0) -- (-1,0.6) -- (-0.3,1.5) -- (0.5,1.7) -- (0.5,1.25);
 \draw [line width=5pt,white, rounded corners=10] (0.5,1.25) -- (0.5,0.8) -- (-0.3,1) -- (-1,1.6) -- (-1,1.9);
 \draw [rounded corners=10] (0.5,1.25) -- (0.5,0.8) -- (-0.3,1) -- (-1,1.6) -- (-1,1.9);
 \draw [line width=5pt,white, rounded corners=10] (1,0) -- (1,0.3) -- (-1.5,0.8) -- (-1.5,1.9);
 \draw [ultra thick, rounded corners=10,->] (1,0) -- (1,0.3) -- (-1.5,0.8) -- (-1.5,1.9);
 \node at (-1,-0.2) {$\scriptstyle{a^{-1}}$}; 
 \node at (1.1,-0.2) {$\scriptstyle{\overline{b}}$}; 
 \node at (-0.9,2.1) {$\scriptstyle{c^{-1}}$}; 
 \node at (1,1.7) [myblue] {$\scriptstyle{t(a)^{\overline{b}}}$}; 
 \node at (1.1,1.2) [red] {$\scriptstyle{=t(c)}$}; 
\end{tikzpicture}
   \caption{Comparing $t(a)^{\overline{b}}$ and $t(c)$ under the assumption $c^{-1} = (a^{-1})^{\overline{b}}$. Thick strands stand here for bundles of parallel strands.}\label{P:KinkTrick}
\end{figure}
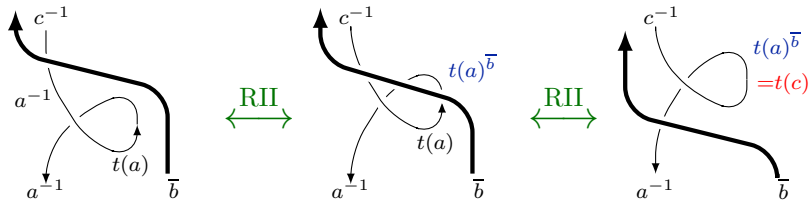	
\end{proof}

\section{The two homology theories coincide}\label{S:HomCoincide} 

Recall that for biracks we have seen two homology constructions: the general braided homology (Section~\ref{S:BrHom}), and a specific theory (Section~\ref{S:BirackHom}). We now establish the equivalence of these theories using the guitar map~$J$. Moreover, we extend the birack homology (and our equivalence of complexes) to the more general left non-degenerate (=LND) braided sets, and add coefficients to the complexes involved. 

As usual, for LND braidings we make use of the notations $b \wdot a$ and $a \cdot b$ (Notation~\ref{N:sideways}), and of the graphical calculus involving allowed R$\mathrm{II}$ moves (Definition~\ref{D:AllowedRII}) and nicely oriented R$\mathrm{III}$ moves. See Figure~\ref{P:Side} for the coloring rules expressed in terms of the operations~$\cdot$ and~$\wdot$. We will also need the maps
\begin{align}\label{E:chi}
\chi_i (y_1, \ldots, y_k) &= J^{-1}_1 ((y_i \ops y_{i-1}) \cdots \ops y_1, y_1, \ldots, y_{i-1}, y_{i+1}, \ldots, y_k)
\end{align}
from~$X^{\times k}$ to~$X$, where we write the inverse of the $k$-guitar map~$J$ as $J^{-1} = (J^{-1}_1, \ldots,J^{-1}_k)$.

\begin{thm}\label{T:HomEquiv}
Let $(X,\,\sigma)$ be a left non-degenerate braided set. Let $(M,\,\cdot)$ be a right module and $(N,\,\cdot)$ be a left module over $(X,\,\sigma)$.
    \begin{enumerate} 
        \item 
			A pre-cubical structure on $C_k=M \times X^{\times k} \times N$ can be given by the maps            
            \begin{align*}
                d_i &\colon (m, y_1, \ldots, y_k, n) \mapsto (m, y_i \wdot  y_1, \ldots, y_i \wdot  y_{i-1}, y_i \cdot y_{i+1}, \ldots, y_i \cdot  y_k, y_i \cdot n),\\
                d'_i & \colon (m, y_1, \ldots, y_k, n) \mapsto (m \cdot \chi_i (\overline{y}), y_1, \ldots, y_{i-1}, y_{i+1}, \ldots, y_k, n).
            \end{align*}
        \item 
            The extended guitar map $J := \Id_M \times J \times \Id_N$ yields an isomorphism between the pre-cubical structure $(d^{r,-}_i, d^{l,+}_i)$ from Theorem~\ref{T:BrHom} and the structure $(d_i, d'_i)$ above.
    \end{enumerate}
\end{thm} 

As a consequence, the chain complexes obtained from these pre-cubical
structures via Theorem~\ref{PR:PreCub} are isomorphic.

Note that when the braided set is a birack and the coefficients $M,N$ are
trivial (Example~\ref{EX:Trivial}), the pre-cubical structure from our theorem
specializes to that from Theorem~\ref{T:Birack}. For braided sets associated to
cycle sets, the exotic map~$\chi_i$ takes the simpler form $ J^{-1}_1 (y_i,
y_1, \ldots, y_{i-1}, y_{i+1}, \ldots, y_k)$, and appears in Dehornoy's
right-cyclic calculus~\cite{MR3374524}.

As usual, the theorem remains true when the operations~$\cdot$ and~$\wdot$
exchange places; this yields a generalization of the structure $(d^\star_i,
d'_i)$ from Theorem~\ref{T:Birack}.

\begin{proof}
Since the guitar map is bijective, it suffices to prove that it entwines the structures in question, i.e. that one has 
\begin{align*}
J \circ d^{r,-}_i &= d_i \circ J, & J \circ d^{l,+}_i &= d'_i \circ J
\end{align*}
for all $n \geqslant 1$, $1 \leqslant i \leqslant n$. 
This would imply in particular that $(d_i,d'_i)$ is indeed a pre-cubical structure. A graphical proof is presented in Figure~\ref{P:HomIso}.    

\begin{figure}[h]
\begin{tikzpicture}[scale=0.45,>=latex]
 \draw [->] (0,9) arc [radius=4, start angle=90, end angle= 450]; 
 \draw [line width=4pt,white] (0,2.5) circle [radius=4]; 
 \draw [->] (0,6.5) arc [radius=4, start angle=90, end angle= 450]; 
 \path [draw,line width=4pt,white, rounded corners] (2,10) -- (0,4) arc (90:270:4) -- (1,-7.5);
 \draw [>-,red] (-4,0) arc [radius=4, start angle=180, end angle= 190]; 
 \path [draw, rounded corners,red] (2,10) -- (0,4) arc (90:270:4) -- (1,-7.5);
 \draw [line width=4pt,white] (0,-2.5) circle [radius=4]; 
 \draw [->] (0,1.5) arc [radius=4, start angle=90, end angle= 450];
 \draw [->, mygreen, thick] (-5,-7.5)--++(12,0);
 \draw [->, choco, thick] (-0.7,1.25)--(0.7,1.25);  
 \draw [myviolet, dashed] (-0.6,1.25)--(-2,-7.5); 
 \draw [myviolet, dashed] (0.6,1.25)--(2,-7.5); 
 \draw [myviolet, dashed] (4,5)--(4,-7.5);
 \node at (3.8,5) [right] {$\scriptstyle{{\color{red}y_3} \wdot y_1}$}; 
 \node at (3.8,2.5) [right] {$\scriptstyle{{\color{red}y_3} \wdot y_2}$}; 
 \node at (3.8,-2.5) [right] {$\scriptstyle{{\color{red}y_3} \cdot y_4}$}; 
 \node at (2.2,9.5) {$\scriptstyle{{\color{red}y_3}}$};    
 \node at (-0.2,0.6) {$\scriptstyle{x_1}$}; 
 \node at (-0.6,-1.8) {$\scriptstyle{x_2}$};  
 \node at (-0.9,-3.5) [red] {$\scriptstyle{x_3}$}; 
 \node at (-1.2,-6) {$\scriptstyle{x_4}$};  
 \node at (-4,-7.6) [mygreen,above] {$\scriptstyle{n}$}; 
 \node at (3.8,-7.2) [mygreen,above,right] {$\scriptstyle{{\color{red}y_3} \cdot n}$}; 
 \node at (0.6,-7.5) [red,above] {$\scriptstyle{y_3}$};  
 \node at (-2,-8) {$\scriptstyle{\overline{v}}$};   
 \node at (2,-8) {$\scriptstyle{d^{r,-}_3(\overline{v})}$};  
 \node at (5.5,-8) {$\scriptstyle{J(d^{r,-}_3(\overline{v}))}$};   
 \node at (5.5,-9.1) {$\scriptstyle{= d_3(J(\overline{v}))}$}; 
\end{tikzpicture}
\hspace*{20pt}
\begin{tikzpicture}[scale=0.45,>=latex]
 \node at (-6,0) { }; 
 \draw [->] (0,9) arc [radius=4, start angle=90, end angle= 450]; 
 \draw [line width=4pt,white] (0,2.5) circle [radius=4]; 
 \draw [->] (0,6.5) arc [radius=4, start angle=90, end angle= 450]; 
 \path [draw, rounded corners,line width=4pt,white]  (-1.2,1.8) arc (100:300:2.5) -- (0.7,2); 
 \path [->,draw, rounded corners,red]  (-1.2,1.8) arc (100:200:2.5); 
 \path [draw, rounded corners,red]  (-1.2,1.8) arc (100:300:2.5) -- (0.7,2); 
 \draw [line width=4pt,white] (0,-1) circle [radius=4]; 
 \draw [->] (0,3) arc [radius=4, start angle=90, end angle= 450];
 \draw [->, choco, thick] (-0.5,2)--(1.9,2); 
 \draw [->, mygreen, thick] (-5,-6)--++(11,0); 
 \draw [myviolet, dashed] (0,2)--(0,-6);  
 \draw [myviolet, dashed] (1.2,2)--(1.2,-6);  
 \draw [myviolet, dashed] (4,5)--(4,-6);   
 \node at (3.9,5) [right] {$\scriptstyle{y_1}$}; 
 \node at (3.9,2.5) [right] {$\scriptstyle{y_2}$}; 
 \node at (3.9,-1.5) [right] {$\scriptstyle{y_4}$}; 
 \node at (-0.4,0.6) {$\scriptstyle{x_1}$}; 
 \node at (-0.4,-1.8) {$\scriptstyle{x_2}$};  
 \node at (-0.4,-3.5) [red] {$\scriptstyle{x_3}$}; 
 \node at (-0.4,-5.4) {$\scriptstyle{x_4}$};
 \node at (-4,-6.1) [mygreen,above] {$\scriptstyle{n}$}; 
 \node at (-0.2,1.9) [choco,above] {$\scriptstyle{m}$};  
 \node at (1.4,1.9) [choco,above] {$\scriptstyle{m'}$};  
 \node at (-0.4,-6.5) {$\scriptstyle{\overline{v}}$};   
 \node at (1.7,-6.5) {$\scriptstyle{d^{l,+}_3(\overline{v})}$};  
 \node at (5.1,-6.5) {$\scriptstyle{J(d^{l,+}_3(\overline{v}))}$};   
 \node at (5.1,-7.6) {$\scriptstyle{= d'_3(J(\overline{v}))}$}; 
 \node at (3.8,-5.7) [mygreen,above,right] {$\scriptstyle{n}$}; 
\end{tikzpicture}
   \caption{Comparing~$d^{r,-}_i$ with~$d_i$, and~$d^{l,+}_i$ with~$d'_i$, $n=4$, $i=3$. Here $\overline{v} = (m,x_1,\ldots,x_4,n)$, and $J(\overline{v}) = (m,y_1,\ldots,y_4,n)$. On the left the brown $M$-colored strand has a constant color~$m$, and on the right its color changes from~$m$ to $m' = m \cdot ({}^{x_1  x_{2}}x_3) = m \cdot \chi_3 (\overline{y})$.}\label{P:HomIso}
\end{figure}
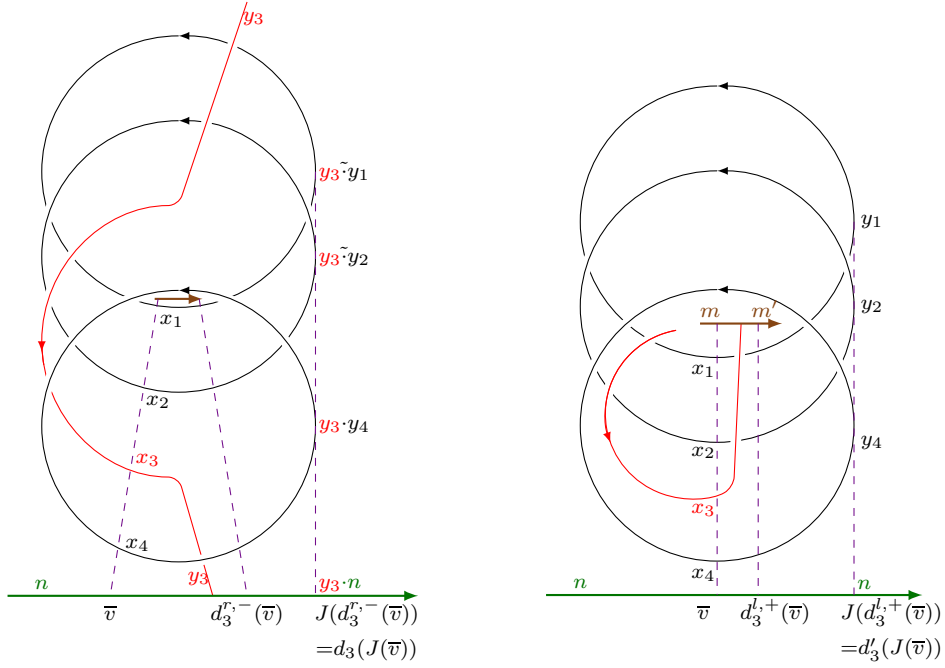
In this proof we worked with a slightly modified definition of the maps~$\chi_i$, in the sense of the following lemma.

\begin{lem}\label{L:chi}
The families of maps $\chi_i$ and $\chi'_i \colon X^{\times k} \to X$ defined by 
\begin{align}
\chi_i (\overline{y}) &= J^{-1}_1 ((y_i \ops y_{i-1}) \cdots \ops y_1, y_1, \ldots, y_{i-1}, y_{i+1}, \ldots, y_k),\notag\\
\chi'_i(\overline{y}) &= {}^{x_1\ldots x_{i-1}}x_i, \qquad \overline{x} = J^{-1}(\overline{y})\label{E:chi'}
\end{align}
(with notations from Figure~\ref{P:TX}), coincide. 
\end{lem}          
\begin{proof}
The properties of the guitar map (Proposition~\ref{PR:Guitar}) legitimize the following calculation:
\begin{align*}
\sigma_1 \cdots \sigma_{i-1} (\overline{x}) &= J^{-1}J \sigma_1 \cdots \sigma_{i-1} (\overline{x}) = J^{-1} \sigma'_1 \cdots \sigma'_{i-1} J(\overline{x}) = J^{-1} \sigma'_1 \cdots \sigma'_{i-1} (\overline{y})  
\\& = J^{-1} ((y_i \ops y_{i-1}) \cdots \ops y_1, y_1, \ldots, y_{i-1}, y_{i+1}, \ldots, y_k).
\end{align*}
The desired relation is obtained by comparing the first components of the resulting $k$-tuples. 
\qedhere\qedhere
\end{proof}
\end{proof}

\begin{rem}\label{R:AdjActionBis}
Note that the symbol~$\cdot$ in the expressions $y_i \cdot  y_k$ and $y_i \cdot n$ from the theorem has different meaning. In order to motivate this abuse of notation, we remark that the operations $y_i \cdot  y_k$ and $y_i \wdot  y_k$ both define a left action of $(X,\,\sigma)$ on itself; this is an easy consequence of the Yang--Baxter relation.
\end{rem}

\section{Degeneracies and a homology splitting}\label{S:Splitting} 

For a weakly R$\mathrm{I}$-compatible (Definition~\ref{D:BraidedZoo}) left non-degenerate braided set, we  now enrich the pre-cubical structure from Theorem~\ref{T:HomEquiv} into a semi-strong skew cubical structure. Theorem~\ref{PR:PreCub} then yields a decomposition of the corresponding chain complexes into the degenerate and the normalized parts, generalizing the homology decomposition for quandles from~\cite{MR1952425}.

Recall that for LND braided sets, weak R$\mathrm{I}$-compatibility is equivalent to the condition $a \cdot a = a \wdot a$ for all $a \in X$; in this case the map~$t$ from the definition is necessarily the squaring map $t(a)= a \cdot a$ (Example~\ref{EX:birack}). Recall also the maps~$\chi_i$ defined by~\eqref{E:chi} or, equivalently, by~\eqref{E:chi'}.

\begin{defn}
A right module $(M,\,\cdot)$ over a braided set $(X,\,\sigma)$ is called \emph{solid} if every $a \in X$ acts on~$M$ bijectively. In this case, the inverse of the bijection $m \mapsto m \cdot a$ is denoted by $m \mapsto m \cdot a^{-1}$. Solid left modules are defined similarly. 
\end{defn}

\begin{thm}\label{T:DegAndNormHom}
Let $(X,\,\sigma)$ be a weakly R$\mathrm{I}$-compatible left non-degenerate braided set. Let $(M,\,\cdot)$ be a solid right module and $(N,\,\cdot)$ a left module over $(X,\,\sigma)$. Then the pre-cubical structure $(M \times X^{\times k} \times N,\, d'_i,\, d_i)$ from Theorem~\ref{T:HomEquiv} can be enriched into a semi-strong skew cubical one by the degeneracies            
            \begin{align*}
                s_i &\colon (m, y_1, \ldots, y_k, n) \mapsto (m \cdot t(\chi_i(\overline{y}))^{-1}, y_1, \ldots, y_{i-1}, y_i, y_i, y_{i+1}, \ldots,  y_k, n).
            \end{align*}
Further, given an abelian group~$A$, the abelian groups $C_k(X,M,N;A) = A M \times X^{\times k} \times N$, $k \geqslant 0$, decompose as
        		\begin{align}
        		C_k(X,M,N;A) &= C^{\Degen}_k(X,M,N;A) \oplus C^{\Norm}_k(X,M,N;A),\label{E:Splitting}\\
        		C^{\Degen}_k(X,M,N;A) &= \sum\nolimits_{i=1}^{k-1} A\im s_i, \qquad
        		C^{\Norm}_k(X,M,N;A) = \im \eta_k, \notag
        		\end{align} 
        		where~$\eta_k$ is the $A$-linearization of the map
        		\begin{align*}
        		\eta_k &= (\Id - s_1d'_2)(\Id - s_2d'_3)\cdots(\Id - s_{k-1}d'_k).
        		\end{align*}
For any $\alpha, \beta \in \Z$, this decomposition is preserved by the differentials 
$$\partial^{(\alpha, \beta)}_k =\alpha\sum\nolimits_i (-1)^{i-1}d'_i + \beta \sum\nolimits_i (-1)^{i-1} d_i.$$
\end{thm}

As usual, decomposition~\eqref{E:Splitting} induces homology splittings.

Recall that the guitar map sends the pre-cubical structure $(d^{l,+}_i,\, d^{r,-}_i)$ isomorphically onto the structure $(d'_i,\, d_i)$ (Theorem~\ref{T:HomEquiv}). Hence the maps $J^{-1}s_iJ$ yield degeneracies for  $(d^{l,+}_i,\, d^{r,-}_i)$, explicitly written as
    \begin{align*}
	s^{\Braided}_i \colon (m, \overline{x}_{-}, x, \overline{x}_{+}, n) &\mapsto (m \cdot t(\prescript{(\overline{x}_{-})}{}{x})^{-1}, t(x) \wdot \overline{x}_{-}, t(x), x, \overline{x}_{+}, n), 
	\end{align*} 
	where $\overline{x}_{-} \in X^{\times (i-1)}$, $\overline{x}_{+} \in X^{\times (k-i)}$, and the operation~$\wdot$ is extended from~$X$ to~$T(X)$ using the extension to~$T(X)$ of the braiding~$\sigma$ (Example~\ref{EX:Adjoint}). Here and afterwards we use the subscript~$\Braided$ or the prefix~$B$ when referring to the braided homology. The modification of the $M$-component, which seemed surprising in the definition of~$s_i$, becomes more intuitive on the level of~$s^{\Braided}_i$: indeed, it can be read off from Figure~\ref{P:si_braided}. This figure is presented here for giving intuition only; to make thing rigorous, one should explain the use of badly oriented R$\mathrm{III}$ moves and the coloring rules around a crossing between an $X$-strand and an $M$-strand (which make possible appropriate R$\mathrm{II}$ and R$\mathrm{III}$ moves).
	
\begin{figure}[h]
\begin{tikzpicture}[xscale=0.8,>=latex]
 \draw [->, thick] (-1,0) -- (-1,2.5);
 \draw [->, ultra thick] (0.5,0) -- (0.5,2.5);
 \draw [line width=5pt,white, rounded corners=10] (-1.7,1.5) -- (-1.7,1) -- (1.5,1) -- (2,2) -- (2,2.5);
 \draw [->,red, rounded corners=10] (-1.7,1.5) -- (-1.7,1) -- (1.5,1) -- (2,2) -- (2,2.5);
 \draw [line width=5pt,white, rounded corners=10] (2,0) -- (2,1) -- (1.5,2) -- (-1.7,2) -- (-1.7,1.5);
 \draw [red, rounded corners=10] (2,0) -- (2,1) -- (1.5,2) -- (-1.7,2) -- (-1.7,1.5);
 \draw [->, ultra thick] (3,0) -- (3,2.5);
 \draw [->, thick] (4,0) -- (4,2.5);
 \draw [myviolet, dashed] (-1,0.3) -- (4.5,0.3); 
 \draw [myviolet, dashed] (-1,1.2) -- (4.5,1.2); 
 \node at (-1,-0.2) {$\scriptstyle{m}$};  
 \node at (0.5,-0.2) {$\scriptstyle{\overline{x}_-}$};
 \node at (2,-0.2) {$\scriptstyle{x}$};
 \node at (3,-0.2) {$\scriptstyle{\overline{x}_+}$};
 \node at (4,-0.2) {$\scriptstyle{n}$};
 \node at (-1,2.7) {$\scriptstyle{m}$};  
 \node at (0.5,2.7) {$\scriptstyle{\overline{x}_-}$};
 \node at (2,2.7) {$\scriptstyle{x}$};
 \node at (3,2.7) {$\scriptstyle{\overline{x}_+}$};
 \node at (4,2.7) {$\scriptstyle{n}$};
 \node at (-1.9,1.5) [myblue] {$\scriptstyle{a}$}; 
 \node at (-0.4,1.7) [myblue] {$\scriptstyle{m \cdot a^{-1}}$};
 \node at (-0.5,2.15) [myblue] {$\scriptstyle{a}$};   
 \node at (-0.5,0.85) [myblue] {$\scriptstyle{a}$}; 
 \node at (1.2,2.2) [myblue] {$\scriptstyle{t(x)}$};   
 \node at (1.2,0.8) [myblue] {$\scriptstyle{t(x)}$}; 
 \node at (4.5,0.1) [myviolet] {$\scriptstyle{\overline{t}}$};
 \node at (4.7,1.1) [myviolet] {$\scriptstyle{s^{\Braided}_i(\overline{t})}$};
 \node  at (6.6,1.5){\Large \color{mygreen} $\overset{\mathrm{RIII}}{\longleftrightarrow}$};
\end{tikzpicture}
\begin{tikzpicture}[xscale=0.8,>=latex]
 \node  at (-2,1.5){ };
 \draw [->, thick] (-1.5,0) -- (-1.5,2.5);
 \draw [->, ultra thick, rounded corners] (1.25,0) --  (1.7,1.25) -- (1.25,2.5);
 \draw [line width=5pt,white, rounded corners=10] (-0.5,1.25) -- (-0.5,0.8) -- (0.3,1) -- (2,2.2) -- (2,2.5);
 \draw [-<,red, rounded corners=10] (-0.5,1.25) -- (-0.5,0.8) -- (0.3,1) -- (2,2.2) -- (2,2.5);
 \draw [line width=5pt,white, rounded corners=10] (2,0) -- (2,0.3) -- (0.3,1.5) -- (-0.5,1.7) -- (-0.5,1.25);
 \draw [red, rounded corners=10] (2,0) -- (2,0.3) -- (0.3,1.5) -- (-0.5,1.7) -- (-0.5,1.25);
 \draw [->, ultra thick] (2.75,0) -- (2.75,2.5);
 \draw [->, thick] (3.5,0) -- (3.5,2.5);
 \node at (-1.5,-0.2) {$\scriptstyle{m}$};  
 \node at (1.25,-0.2) {$\scriptstyle{\overline{x}_-}$};
 \node at (2,-0.2) {$\scriptstyle{x}$};
 \node at (2.75,-0.2) {$\scriptstyle{\overline{x}_+}$};
 \node at (3.5,-0.2) {$\scriptstyle{n}$};
 \node at (-1.5,2.7) {$\scriptstyle{m}$};  
 \node at (1.25,2.7) {$\scriptstyle{\overline{x}_-}$};
 \node at (2,2.7) {$\scriptstyle{x}$};
 \node at (2.75,2.7) {$\scriptstyle{\overline{x}_+}$};
 \node at (3.5,2.7) {$\scriptstyle{n}$};
 \node at (-0.7,1.4) [myblue] {$\scriptstyle{=}$}; 
 \node at (-0.7,1.1) [myblue] {$\scriptstyle{a}$};
 \node at (-0.3,1.8) [red] {$\scriptstyle{t(\prescript{(\overline{x}_{-})}{}{x})}$};  
 \node at (0.8,0.7) [red] {$\scriptstyle{\prescript{(\overline{x}_{-})}{}{x}}$};  
\end{tikzpicture}
   \caption{A graphical definition of~$s^{\Braided}_i$ (left) and a computation of the element $a\in X$ acting on $m \in M$ (right).}\label{P:si_braided}
\end{figure}
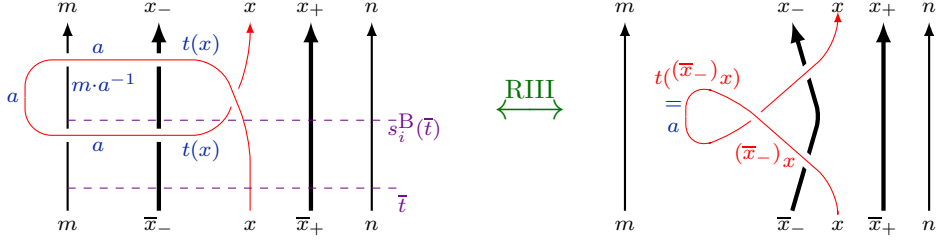	

Further, decomposition~\eqref{E:Splitting} implies the decomposition
\begin{align*}
C_k(X,M,N;A) &= J^{-1}(C^{\Degen}_k(X,M,N;A)) \oplus J^{-1}(C^{\Norm}_k(X,M,N;A))
\end{align*}
(the maps~$J$ and~$J^{-1}$ are extended by linearity), preserved by the differentials 
$$\partial^{\Braided,(\alpha, \beta)}_k =\alpha\sum\nolimits_i (-1)^{i-1}d^{l,+}_i + \beta \sum\nolimits_i (-1)^{i-1} d^{r,-}_i.$$ 
Explicitly, one has the following result:
\begin{cor}
In the context of Theorem~\ref{T:DegAndNormHom}, for every $\alpha,\beta \in \Z$ the chain complex $(C_k(X,M,N; A),\, \partial^{\Braided,(\alpha, \beta)}_k)$ admits the following direct summand:
    \begin{align*}
    BC^{\Degen}_k(X,M,N;A) &= \sum AM \times X^{\times (i-1)} \times (x \cdot x,x)\times X^{\times (k-i-1)} \times N,
	\end{align*}     
	where the sum is over all $1 \leqslant i \leqslant k-1$ and $x \in X$.
\end{cor}

For the proof of Theorem~\ref{T:DegAndNormHom} we shall need the following lemma.

\begin{lem}\label{L:Kink}
Let $(X,\,\sigma)$ be a weakly R$\mathrm{I}$-compatible LND braided set. Consider the relation $(a_1,a_2)^b=(c_1,c_2)$ in $T(X)$. Then condition $a_1 = t(a_2)$ is equivalent to $c_1 = t(c_2)$. The same equivalence holds for the relation $\prescript{b}{}{(a_1,a_2)}=(c_1,c_2)$.
\end{lem}

\begin{proof}
The first equivalence is established in Figure~\ref{P:KinkProof}, the second one is similar. \qedhere

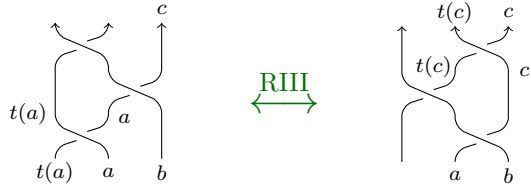
\begin{figure}[h]
\begin{tikzpicture}[xscale=0.7,yscale=0.6]
\draw [rounded corners](0,0)--(0,0.25)--(0.4,0.4);
\draw [rounded corners](0.6,0.6)--(1,0.75)--(1,1.25)--(1.4,1.4);
\draw [rounded corners,->](1.6,1.6)--(2,1.75)--(2,3);
\draw [rounded corners](1,0)--(1,0.25)--(0,0.75)--(0,2.25)--(0.4,2.4);
\draw [rounded corners,->](0.6,2.6)--(1,2.75)--(1,3);
\draw [rounded corners,->](2,0)--(2,1.25)--(1,1.75)--(1,2.25)--(0,2.75)--(0,3);
\node  at (0,-0.3) {$\scriptstyle t(a)$};
\node  at (1,-0.3) {$\scriptstyle a$};
\node  at (2,-0.3) {$\scriptstyle b$};
\node  at (2,3.3) {$\scriptstyle c$};
\node  at (4.3,1.5){\Large \color{mygreen} $\overset{\mathrm{RIII}}{\longleftrightarrow}$};
\node  at (-0.5,1) {$\scriptstyle t(a)$};
\node  at (1.3,0.9) {$\scriptstyle a$};
\end{tikzpicture}
\begin{tikzpicture}[xscale=0.7,yscale=0.6]
\node  at (-1,1.5){$ $};
\draw [rounded corners](1,1)--(1,1.25)--(1.4,1.4);
\draw [rounded corners,->](1.6,1.6)--(2,1.75)--(2,3.25)--(1,3.75)--(1,4);
\draw [rounded corners](0,1)--(0,2.25)--(0.4,2.4);
\draw [rounded corners](0.6,2.6)--(1,2.75)--(1,3.25)--(1.4,3.4);
\draw [rounded corners,->](1.6,3.6)--(2,3.75)--(2,4);
\draw [rounded corners,->](2,1)--(2,1.25)--(1,1.75)--(1,2.25)--(0,2.75)--(0,4);
\node  at (1,4.3) {$\scriptstyle t(c)$};
\node  at (1,0.7) {$\scriptstyle a$};
\node  at (2,0.7) {$\scriptstyle b$};
\node  at (2,4.3) {$\scriptstyle c$};
\node  at (0.6,3.1) {$\scriptstyle t(c)$};
\node  at (2.3,3) {$\scriptstyle c$};
\end{tikzpicture}
\caption{The colors on the left-hand side have the indicated pattern if and only if they do so on the right. This shows that the neighbouring colors $(t(x),x)$ remain of the same type when passing (in any direction) under another strand.}\label{P:KinkProof}
\end{figure}
\end{proof}

\begin{proof}[Proof of Theorem~\ref{T:DegAndNormHom}]
The verification of the semi-strong skew cubical relations \eqref{E:WeakCub}-\eqref{E:WeakCub3'} is easy on the $X$- and $N$-components of $M \times T(X) \times N$. One should be more careful with how the left- and the right-hand sides of these relations modify the $M$-component. For example, on the level of the $M$-components, relation $d'_i s_j = s_{j}d'_{i-1}$ for $i > j+1$ reads
$$(m \cdot t(\chi_j(\overline{y}))^{-1}) \cdot \chi_i(s_j(\overline{y})) = (m \cdot \chi_{i-1}(\overline{y})) \cdot  t(\chi_j(d'_{i-1}(\overline{y})))^{-1}$$
for all $\overline{y} \in X^{\times k}$, $m \in M$ (here the maps~$s_j$ and~$d'_{i-1}$ are used with trivial coefficients). This is equivalent to the relation
$$(m \cdot \chi_i(s_j(\overline{y}))) \cdot t(\chi_j(d'_{i-1}(\overline{y}))) = (m \cdot t(\chi_j(\overline{y}))) \cdot \chi_{i-1}(\overline{y}).$$
Since~$M$ is a braided module, it is sufficient to show the property
\begin{align}\label{E:ChiProperty}
\sigma(t(\chi_j(\overline{y})), \chi_{i-1}(\overline{y})) &= (\chi_i(s_j(\overline{y})), t(\chi_j(d'_{i-1}(\overline{y})))),
\end{align}
which we establish in Figure~\ref{P:ChiProperty}.

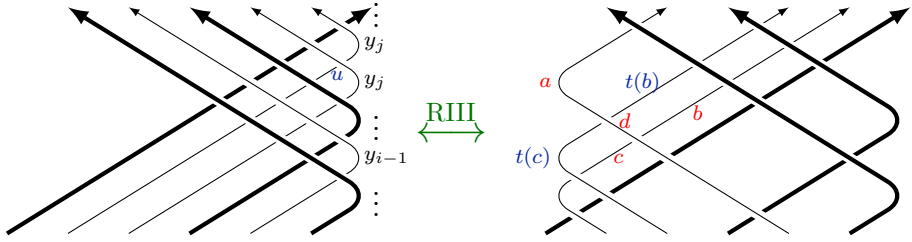
\begin{figure}[h]
\begin{tikzpicture}[xscale=0.8,>=latex]
 \draw [->, rounded corners=10, ultra thick] (-1,0) -- (5,3);
 \draw [line width=4pt,white, rounded corners=10] (0,0) -- (5,2.5) -- (4,3); 
 \draw [->, rounded corners=10] (0,0) -- (5,2.5) -- (4,3);
 \draw [line width=4pt,white, rounded corners=10] (1,0) -- (5,2) -- (3,3); 
 \draw [->, rounded corners=10] (1,0) -- (5,2) -- (3,3); 
 \draw [line width=4pt,white, rounded corners=10] (2,0) -- (5,1.5) -- (2,3); 
 \draw [->, rounded corners=10, ultra thick] (2,0) -- (5,1.5) -- (2,3);
 \draw [line width=4pt,white, rounded corners=10] (3,0) -- (5,1) -- (1,3); 
 \draw [->, rounded corners=10] (3,0) -- (5,1) -- (1,3); 
 \draw [line width=4pt,white, rounded corners=10] (4,0) -- (5,0.5) -- (0,3);   
 \draw [->, rounded corners=10, ultra thick] (4,0) -- (5,0.5) -- (0,3);  
 \node at (4.7,2.5) [right] {$\scriptstyle{y_j}$};   
 \node at (4.7,2) [right] {$\scriptstyle{y_j}$}; 
 \node at (4.15,2.1) [right,myblue] {$\scriptstyle{u}$}; 
 \node at (4.7,3) [right] {$\;\scriptstyle{\vdots}$};  
 \node at (4.7,1.5) [right] {$\;\scriptstyle{\vdots}$}; 
 \node at (4.7,0.5) [right] {$\;\scriptstyle{\vdots}$}; 
 \node at (4.7,1) [right] {$\scriptstyle{y_{i-1}}$};  
 \node  at (6.3,1.5){\Large \color{mygreen} $\overset{\mathrm{RIII}}{\longleftrightarrow}$};
\end{tikzpicture}
\begin{tikzpicture}[xscale=0.8,>=latex]
 \draw [->, rounded corners=10, ultra thick] (-1,0) -- (5,3);
 \draw [line width=4pt,white, rounded corners=10] (0,0) -- (-1,0.5) -- (4,3); 
 \draw [->, rounded corners=10] (0,0) -- (-1,0.5) -- (4,3);
 \draw [line width=4pt,white, rounded corners=10] (1,0) -- (-1,1) -- (3,3); 
 \draw [->, rounded corners=10] (1,0) -- (-1,1) -- (3,3); 
 \draw [line width=4pt,white, rounded corners=10] (2,0) -- (5,1.5) -- (2,3); 
 \draw [->, rounded corners=10, ultra thick] (2,0) -- (5,1.5) -- (2,3);
 \draw [line width=4pt,white, rounded corners=10] (3,0) -- (-1,2) -- (1,3); 
 \draw [->, rounded corners=10] (3,0) -- (-1,2) -- (1,3); 
 \draw [line width=4pt,white, rounded corners=10] (4,0) -- (5,0.5) -- (0,3);   
 \draw [->, rounded corners=10, ultra thick] (4,0) -- (5,0.5) -- (0,3);  
 \node at (-1,2) [red] {$\scriptstyle{a}$};
 \node at (-1.2,1) [myblue] {$\scriptstyle{t(c)}$};  
 \node at (0.2,1) [red] {$\scriptstyle{c}$};  
 \node at (0.3,1.5) [red] {$\scriptstyle d$}; 
 \node at (1.5,1.6) [red] {$\scriptstyle b$};  
 \node at (0.6,2) [myblue] {$\scriptstyle t(b)$};
\end{tikzpicture}
   \caption{A proof of relation~\eqref{E:ChiProperty}. Here the thick lines stand for half-twisted bundles of strands. On the left, the $j$th strand is doubled (under the action of~$s_j$). The color identifications $a= \chi_i(s_j(\overline{y}))$, $b= \chi_j(d'_{i-1}(\overline{y}))$, $c= \chi_j(\overline{y})$, $d= \chi_{i-1}(\overline{y})$ are established using the expression~\eqref{E:chi'} for the maps~$\chi$. The colors $t(b)$ and $t(c)$ are obtained by a repeated application of Lemma~\ref{L:Kink}, starting from the colors $u=t(y_j)$ and $y_j$ on the left.}\label{P:ChiProperty}
\end{figure}	

The remaining semi-strong cubical relations follow from the properties
\begin{align*}
\sigma(\chi_i(s_j(\overline{y})), t(\chi_{j-1}(d'_{i}(\overline{y})))) &= (t(\chi_j(\overline{y})), \chi_{i}(\overline{y})), &i < j\\
\sigma(t(\chi_i(s_j(\overline{y}))), t(\chi_j(\overline{y}))) &= (t(\chi_{j+1}(s_i(\overline{y}))), t(\chi_{i}(\overline{y}))), &i \leqslant j\\
\chi_i(s_i(\overline{y})) = \chi_{i+1}(s_i(\overline{y})) &= t(\chi_i(\overline{y})), &\\
\chi_j(\overline{y}) &= \chi_j(d_{i-1}(\overline{y})), &i > j+1\\
\chi_j(\overline{y}) &= \chi_{j-1}(d_{i}(\overline{y})), &i < j.
\end{align*}
These are established by a similar graphical procedure: in the guitar map diagram, one pulls to the left the strings responsible for the degeneracies and/or boundaries involved, and determines the induced colors.

Alternatively, one could show the semi-strong skew cubical relations for the data $(d^{l,+}_i, d^{r,-}_i, s^{\Braided}_i)$ using the graphical calculus, based on the diagrammatic interpretations of these maps from Figures~\ref{P:BrHom} and~\ref{P:si_braided}. 

By linearization, one obtains a semi-strong skew cubical structure on $C_k(X,M,N;A)$, which we regard as $\Z$-modules. The desired decomposition and its compatibility with the differentials now follow from Theorem~\ref{PR:PreCub}. 
\end{proof}

Let us now explore the applications of Theorems~\ref{T:HomEquiv} and~\ref{T:DegAndNormHom} to two particular cases of braided sets. 

\begin{example}
A rack $(X,\,\op)$ can be seen as a LND braided set, with the braiding $\sigma_{\op}(a,b) =  (b,a \op b)$. The operation~$\cdot$ and~$\wdot$ become here
$a \cdot b = b$, $a \wdot b = b \wop a$, 
where the operation $b \mapsto b \wop a$ is defined as the inverse of $b \mapsto b \op a$. The maps~$\chi_i$ from~\eqref{E:chi} simplify as
\begin{align*}
\chi_i(y_1,\ldots,y_k) = (\cdots(y_i \wop y_{i+1}) \wop \cdots) \wop y_k,
\end{align*}
denote here by $y_i \wop y_{i+1} \cdots y_k$ for the sake of readability. 
For a right module~$M$ and a left module~$N$ over $(X,\,\op)$ (which are thus braided modules, cf. Example~\ref{EX:rack}), the pre-cubical structure from Theorem~\ref{T:HomEquiv} writes
            \begin{align*}
                d_i &\colon (m, y_1, \ldots, y_k, n) \mapsto (m, y_1 \wop y_i, \ldots, y_{i-1} \wop y_i, y_{i+1}, \ldots, y_k, y_i \cdot n),\\
                d'_i & \colon (m, y_1, \ldots, y_k, n) \mapsto (m \cdot (y_i \wop y_{i+1} \cdots y_k), y_1, \ldots, y_{i-1}, y_{i+1}, \ldots, y_k, n).
            \end{align*}
The inverse
\begin{align*}
J^{-1}(m,y_1,\ldots,y_{k-1},y_k,n) &= (m, y_1\wop y_{2} \cdots y_k,\ldots, y_{k-1} \wop y_{k},y_k,n).
\end{align*}
of the guitar map sends these boundary maps to            
            \begin{align*}
                d^{r,-}_i &\colon (m, x_1, \ldots, x_k, n) \mapsto (m, x_1, \ldots, x_{i-1}, x_{i+1}, \ldots, x_k, (x_i \op x_{i+1} \cdots x_k) \cdot n),\\
                d^{l,+}_i & \colon (m, x_1, \ldots, x_k, n) \mapsto (m \cdot x_i, x_1 \op x_i, \ldots, x_{i-1} \op x_i, x_{i+1}, \ldots, x_k, n).
            \end{align*}
For trivial coefficients, this result was established by Przytycki~\cite{Prz1}; see his work for the meaning of the corresponding isomorphisms of complexes in the self-distributive homology theory. 

Suppose now that $X$ acts on~$M$ by bijections (which is a standard assumption in rack theory). If our rack is a quandle, then the braiding~$\sigma_{\op}$ is R$\mathrm{I}$-compatible, with $t(a)=a$. Theorem~\ref{T:DegAndNormHom} then yields the degeneracies
            \begin{align*}
                s_i \colon (m,& y_1, \ldots, y_k, n) \mapsto\\
                & (m \cdot (y_i \wop y_{i+1} \ldots y_k)^{-1}, y_1, \ldots, y_{i-1}, y_i, y_i, y_{i+1}, \ldots,  y_k, n)
            \end{align*}
for $(d'_i,\, d_i)$. The decomposition~\ref{E:Splitting} then generalizes the splitting known in the case of trivial coefficients since the work of Litherland and Nelson~\cite{MR1952425}.
\end{example}

\begin{example}
A group $(X,\, \star,\, e)$ is also an R$\mathrm{I}$-compatible LND braided set, with the braiding $\sigma_{\star} (a,b) = (e, a \star b)$, and the constant map $t(a)=e$ (Example~\ref{EX:monoid}). For this structure, one calculates
\begin{align*}
a \cdot b &= e, & \chi_i(y_1,\ldots,y_k) &= e, \quad i \geqslant 2,\\
a \wdot b &= b \star a^{-1}, & \chi_1(y_1,\ldots,y_k) &=y_1 \star y_2^{-1}
\end{align*}
(we declare $y_2=e$ if $k=1$). Take also a right module~$M$ and a left module~$N$ over the group~$X$ (which are thus solid braided modules). Theorem~\ref{T:HomEquiv} then says that the pre-cubical structures
            		\begin{align*}
                d_i &\colon (m, y_1, \ldots, y_k, n) \mapsto (m, y_1 \star y_i^{-1}, \ldots, y_{i-1}\star y_i^{-1}, e, \ldots, e, y_i \cdot n),\\
                d'_i & \colon (m, y_1, \ldots, y_k, n) \mapsto (m, y_1, \ldots, y_{i-1}, y_{i+1}, \ldots, y_k, n), \; i \geqslant 2,\\
                d'_1 & \colon (m, y_1, \ldots, y_k, n) \mapsto ((m \cdot y_1)\cdot y_2^{-1}, y_{2}, \ldots, y_k, n);\\                
              \text{and }  d^{r,-}_i &\colon (m, x_1, \ldots, x_k, n) \mapsto (m, x_1, \ldots, x_{i-1}, e, \ldots, e, x_i \cdot \ldots \cdot (x_k \cdot n)),\\
                d^{l,+}_i & \colon (m, x_1, \ldots, x_k, n) \mapsto (m, x_1, \ldots, x_{i-2}, x_{i-1} \star x_i, x_{i+1}, \ldots, x_k, n), \; i \geqslant 2,\\
                d^{l,+}_1 & \colon (m, x_1, \ldots, x_k, n) \mapsto (m \cdot x_1, x_2 \ldots, x_k, n)
                \end{align*}
are connected by the isomorphisms
\begin{align*}
J(m,x_1,\ldots,x_{k-1},x_k,n) &=(m,x_1 \star \cdots \star x_k,\ldots, x_{k-1}\star x_k, x_k),\\
J^{-1}(m,y_1,\ldots,y_{k-1},y_k,n) &= (m, y_1\star y_2^{-1},\ldots, y_{k-1}\star y_k^{-1},y_k,n).
\end{align*}
One recognizes the two equivalent forms $\sum_i (-1)^{i-1}d'_i$ and $\sum_i (-1)^{i-1} d^{l,+}_i$ of the bar differential for groups. Przytycki~\cite{Prz1} noticed the resemblance between this equivalence of differentials and the corresponding phenomenon in the self-distributive situation. Our unified braided interpretation of the two homology theories offers a conceptual explanation of these parallels.

According to Theorem~\ref{T:DegAndNormHom}, the partial diagonal maps
            \begin{align*}
                s_i \colon (m,& y_1, \ldots, y_k, n) \mapsto  (m, y_1, \ldots, y_{i-1}, y_i, y_i, y_{i+1}, \ldots,  y_k, n)
            \end{align*}
are degeneracies for $(d'_i,\, d_i)$, and the sum of their images forms a direct summand of any of the complexes constructed in Theorem~\ref{PR:PreCub}.
\end{example}

\section{Cycle sets: cohomology and extensions}

In this section we specialize our (co)homology study above to cycle sets (Example~\ref{EX:CycleSet}), and apply it to an analysis of cycle set extensions. In particular, we interpret the latter in terms of certain $2$-cocycles.

Recall that a cycle set is a set~$X$ with a binary operation~$\cdot$ satisfying 
$(a\cdot b)\cdot (a\cdot c)=(b\cdot a)\cdot (b\cdot c)$ 
and having all the translations $a \mapsto b \cdot a$ bijective, with the inverses $a \mapsto  b \ast a$. A cycle set carries the involutive left non-degenerate braiding $\sigma_{\cdot}(a,b) = ( (b \ast a) \cdot b, b \ast a)$. The operations~$\cdot$ and~$\wdot$ (Notation~\ref{N:sideways}) for this braiding both coincide with the original operation~$\cdot$, making the sideways map (Figure~\ref{P:Side}) symmetric: it takes the form $(a,b) \mapsto (a \cdot b, b \cdot a)$. 

We now apply Theorem~\ref{T:HomEquiv} to the braided set $(X,\, \sigma_{\cdot})$, with trivial coefficients (Example~\ref{EX:Trivial}) on the left, and adjoint coefficients $(X, \, \cdot)$ (Remark~\ref{R:AdjActionBis}) on the right. 
More precisely, for our data we consider the chain complex obtained from the pre-cubical structure $(d_i, d'_i)$ via Theorem~\ref{PR:PreCub} with $\alpha = 1, \beta = -1$, and its cohomological counterpart:

\begin{defn}\label{D:CycleSetHom}
The \emph{cycles} / \emph{boundaries} / \emph{homology groups of a cycle set $(X,\, \cdot)$} with coefficients in an abelian group~$A$ are the cycles / boundaries / homology groups of the chain complex $C_n(X,A) = AX^{\times n}$, $n \geqslant 0$, with
\begin{multline*}
	\partial_{n}(x_1,\dots,x_{n})=\sum\nolimits_{i=1}^{n-1}(-1)^{i}( (x_1,\dots,\widehat{x_i},\dots,x_{n})\\
	-(x_i\cdot x_1,\dots,x_i\cdot x_{i-1},x_i\cdot x_{i+1},\dots,x_i\cdot x_{n})),
\end{multline*}
where $(x_1,\dots,\widehat{x_i},\dots,x_{n})=(x_1,\dots,x_{i-1},x_{i+1},\dots,x_n)$, and $\partial_1=0$. The \emph{cocycles} / \emph{coboundaries} / \emph{cohomology groups of $(X,\, \cdot)$} are defined by the differentials $\partial^n \colon f \mapsto f \circ \partial_{n+1}$ on $C^n(X,A) = \Fun(X^n,A)$. The constructed cycle / boundary / homology groups are denoted by $Z_n(X,A)$, $B_n(X,A)$, and $H_n(X,A)$ respectively, with the analogous notations $\ldots^n(X,A)$ in the co-case.
\end{defn}

Note the subscript shift in~$C_n$ with respect to previous sections.

\begin{example}
	For a trivial cycle set $(X,\, x\cdot y=y)$ all the differentials $\partial_n$ and $\partial^n$ vanish, hence one has $H_n(X,A) = AX^{\times n}$, $H^n(X,A) = \Fun(X^n,A)$.
\end{example}

\begin{example}\label{EX:H1}
	The first differentials read $\partial_1 = 0$ and $\partial_2(x,y)= y - x\cdot y$. Thus the homology group $H_1(X,A)$ is the $A$-module freely generated by the \emph{orbits} of our cycle set $(X,\, \cdot)$, i.e. by the classes of the equivalence relation on~$X$ generated by $x \sim y \cdot x$, $x,y \in X$. The cohomology group $H^1(X,A)$ is the group of those maps $X \to A$ which are constant on every orbit.
\end{example}

We now turn to a study of the $2$-cocycles of $(X,\, \cdot)$, i.e., maps $f\colon X\times X\to A$ such that for all $x,y,z\in X$ one has
\begin{align}\label{E:2Cocycle}
f(x,z)+f(x\cdot y,x\cdot z) &= f(y,z)+f(y\cdot x,y\cdot z).
\end{align}

\begin{example} 
	Let $f,g$ be two commuting endomorphisms of an abelian group~$X$ such that $g$ is 
	invertible and $f$ squares to zero. Then $X$ is a cycle set with 
	\[
	x\cdot y=-f(g(x))+g(y),\quad
	x,y\in X.
	\]
	For $h \in\End(X)$, the map $X\times X\to X$,
	$(x,y)\mapsto h(y-x)$, is a $2$-cocycle if and
	only if $h$ satisfies $h =h g$. 
\end{example}

\begin{example}\label{EX:DeltaCocycles}	
	Fix two distinct elements $\alpha_0 \neq \alpha_1$ in an abelian group~$A$.
	Then the map $f(x,y) = \begin{cases} \alpha_1 & \text{if } x=y, \\ \alpha_0
		& \text{if } x \neq y, \end{cases}$\, is a $2$-cocycle of the cycle set
		$(X,\, \cdot)$. Indeed, relations $x=z$ and $y\cdot x = y\cdot z$ are
		equivalent in~$X$ (since the left translations are invertible), which
		yields the desired property~\eqref{E:2Cocycle}.
\end{example}

In the remaining part of this section we will show that $2$-cocycles are
closely related to cycle set extensions. This was one of the motivations behind
our definition of cycle set cohomology.

\begin{lem}\label{L:ExtFromCocycle}
	Let $(X,\, \cdot)$ be a cycle set, $A$ an abelian group, and $f\colon X\times X\to A$ be
	a map. Then $A\times X$ with $(\alpha,x)\cdot (\beta,y)=(\beta+f(x,y),x\cdot y)$ for
	$\alpha,\beta\in A$ and $x,y\in X$ is a cycle set if and only if $f\in Z^2(X,A)$. 
\end{lem}

\begin{notation}\label{N:Ext}
The cycle set from the lemma is denoted by $A\times_fX$.
\end{notation}

\begin{proof}
The left translation invertibility for $A\times X$ follows from the same property for~$X$. Indeed, one can define inverses as $(\alpha,x) \ast (\beta,y)=(\beta-f(x,x \ast y),x \ast y)$. Further, the cycle property
	\begin{equation*}
		((\alpha,x)\cdot (\beta,y))\cdot ((\alpha,x)\cdot (\gamma,z))
		=((\beta,y)\cdot (\alpha,x))\cdot ((\beta,y)\cdot (\gamma,z)),
	\end{equation*}
for $A\times X$ reads 
	\begin{equation*}
	\gamma+f(x,z)+f(x\cdot y,x\cdot z) = \gamma+f(y,z)+f(y\cdot x,y\cdot z),
	\end{equation*}
which is equivalent to~$f$ being a $2$-cocycle.	
\end{proof}

\begin{rem}\label{R:LNDHom}
One can mimic Definition~\ref{D:CycleSetHom} (with the preceding argument) for a general left non-degenerate braided set $(X,\, \sigma)$. In this situation, $2$-cocycles are defined by the property 
$$f(x,z)+f(x\cdot y,x\cdot z)=f(y,z)+f(y\wdot x,y\cdot z).$$
Changing the pre-cubical structure $(d_i, d'_i)$ to $(d^{\star}_i, d'_i)$ (Theorem~\ref{T:Birack}), one gets an alternative (co)homology theory, with the $2$-cocycles, called \emph{star $2$-cocycles} here, defined by  
$$f^{\star}(x,z)+f^{\star}(x\wdot y,x\wdot z)=f^{\star}(y,z)+f^{\star}(y\cdot x,y\wdot z).$$
Further, observe that a left non-degenerate map $(a,b) \mapsto (\prescript{a}{}{b},a^b)$ satisfies the Yang--Baxter equation if and only if the associated maps $\cdot, \wdot$ (Notation~\ref{N:sideways}) obey the following three properties:
\begin{align*}
(a\wdot b)\cdot (a\cdot c) &=(b\cdot a)\cdot (b\cdot c),\\
(a\cdot b)\wdot (a\wdot c) &=(b\wdot a)\wdot (b\wdot c),\\
(a\wdot b)\cdot (a\wdot c) &=(b\cdot a)\wdot (b\cdot c)
\end{align*}
(this is classical for biracks, and the proof extends directly to general left non-degenerate braided sets). 
Now, developing the argument from the lemma above, one shows that the formulas 
\begin{align*}
(\alpha,x)\cdot (\beta,y) &=(\beta+f(x,y),x\cdot y), & (\alpha,x)\wdot (\beta,y) &=(\beta+f^{\star}(x,y),x\wdot y)
\end{align*}
 are associated to a left non-degenerate braiding on $A \times X$ if and only $f$ is a $2$-cocycle, $f^{\star}$ is a star $2$-cocycle, and the two are compatible in the sense of
$$f(x,z)+f^{\star}(x\cdot y,x\cdot z)=f^{\star}(y,z)+f(y\wdot x,y\wdot z).$$
\end{rem}

Inspired by the theory of abelian extensions of quandles \cite{MR2008876,MR1954330}, we define extensions of cycle sets by abelian groups, of which the structure from Lemma~\ref{L:ExtFromCocycle} will be a fundamental example.

\begin{defn}\label{D:Ext}
An \emph{(abelian) extension} of a cycle set $(X,\, \cdot)$ by an abelian
group~$A$ is the data $\extension{Y}{p}{X}{A}$, where $(Y,\, \cdot)$ is a cycle set endowed with a left $A$-action (denoted by $(\alpha,y) \mapsto \alpha y$), and $p\colon Y\to X$ is a surjective cycle set homomorphism, such that the following hold:
\begin{enumerate}
	\item $A$ acts regularly on each fiber $p^{-1}(x)$ (i.e., for all $y,z$ from the same fiber there is a unique $\alpha \in A$ such that $\alpha z=y$), and
	\item for all $\alpha \in A$ and $y,z\in Y$, one has $(\alpha y)\cdot z=y\cdot z$ and $y\cdot (\alpha z)= \alpha (y\cdot z)$. 
\end{enumerate}
\end{defn}

\begin{example}\label{E:central}
	Let $A$ be an abelian group, $(X,\, \cdot)$ a cycle set, and $f$ a cocycle from $Z^2(X,A)$. Form the cycle set $A\times_fX$ (Lemma~\ref{L:ExtFromCocycle}), and consider the canonical surjection $p_X\colon A\times_fX\to X$, $(\alpha,x)\mapsto x$. Let $A$ act on $A\times_fX$ by $\alpha(\beta,x)=(\alpha+\beta,x)$. One readily sees that $(p_X\colon A\times_fX\to X,\, A)$ is an extension of $X$ by $A$. 
\end{example}

\begin{defn}\label{D:ExtEquiv}
Extensions $\extension{Y}{p}{X}{A}$ and $\extension{Y'}{p'}{X}{A}$ are called \emph{equivalent} if there
exists a cycle set isomorphism $F\colon Y \overset{\sim}{\to} Y'$ satisfying $p=p'\circ F$ and $F(\alpha y)=\alpha F(y)$ for all $\alpha\in A$, $y\in Y$. 
\end{defn}

\begin{lem}\label{L:section->cocycle}
	Let $(X,\, \cdot)$ be a cycle set and $\extension{Y}{p}{X}{A}$ be its extension.
	Every set-theoretic section $s\colon X\to Y$ induces a $2$-cocycle $f\in Z^2(X,A)$ such that 
	\begin{align}\label{E:SectionCocycle}
	f(x_1,x_2)s(x_1\cdot x_2)&=s(x_1)\cdot s(x_2)
	\end{align}
	for all $x_1,x_2\in X$.  Furthermore, if $s'\colon X\to Y$ is another section
	and $f'\in Z^2(X,A)$ is its associated $2$-cocycle, then $f$ and $f'$ are
	cohomologous. 
\end{lem}

\begin{proof}
	Take a section $s\colon X\to Y$ to $p\colon Y\to X$
	and two elements $x_1,x_2\in X$. Since $p$ is a cycle set
	homomorphism, both $s(x_1\cdot x_2)$ and $s(x_1)\cdot s(x_2)$ belong to
	$p^{-1}(x_1\cdot x_2)$. By the regularity of the $A$-action on fibers, there exists a unique $f(x_1\cdot x_2)\in A$ verifying~\eqref{E:SectionCocycle}. 
	This defines a map $f\colon X\times X\to A$, which we claim to be a cocycle. Indeed, for $x_1,x_2,x_3\in X$ one calculates
	\begin{align*}
		(f(x_1,x_3)&+f(x_1\cdot x_2,x_1\cdot x_3))\, s( (x_1\cdot x_2)\cdot (x_1\cdot x_3) )\\
		&=f(x_1,x_3)\left( s(x_1\cdot x_2)\cdot s(x_1\cdot x_3) \right)\\
		&=s(x_1\cdot x_2)\cdot \left( f(x_1,x_3)s(x_1\cdot x_3) \right)\\
		&=s(x_1\cdot x_2)\cdot \left( s(x_1)\cdot s(x_3) \right)\\
		&=\left(- f(x_1,x_2)( s(x_1)\cdot s(x_2) \right)\cdot \left( s(x_1)\cdot s(x_3) \right)\\
		&=\left( s(x_1)\cdot s(x_2) \right)\cdot \left( s(x_1)\cdot s(x_3) \right).
	\end{align*}
		Permuting the arguments, one obtains
	\begin{align*}
		(f(x_2,x_3)&+f(x_2 \cdot x_1,x_2\cdot x_3)\, s( (x_2\cdot x_1)\cdot (x_2\cdot x_3) )\\				&=\left( s(x_2)\cdot s(x_1) \right)\cdot \left( s(x_2)\cdot s(x_3) \right).
	\end{align*}
	Now, the cycle property for~$X$ and~$Y$ and the regularity of the $A$-action on fibers imply  $f(x_1,x_3)+f(x_1\cdot x_2,x_1\cdot x_3) = f(x_2,x_3)+f(x_2 \cdot x_1,x_2\cdot x_3)$ as desired.

	Suppose now that $s'$ is another section, and take an $x\in X$. Since $s(x)$ and $s'(x)$ both belong to the fiber $p^{-1}(x)$, there exists a unique $\gamma(x)\in A$ such that
	$s'(x)=\gamma(x) s(x)$. Let us prove that $f(x_1,x_2)-f'(x_1,x_2)=\gamma(x_1\cdot
	x_2)-\gamma(x_2)$ for all $x_1,x_2\in X$, which means that $f$ and $f'$ are
	cohomologous. One has:
	\begin{align*}
		\gamma({x_2})(s(x_1)\cdot s(x_2))
		&=s(x_1)\cdot (\gamma({x_2})s(x_2))\\
		&=(\gamma({x_1})s(x_1))\cdot (\gamma({x_2})s(x_2))\\
		&=s'(x_1)\cdot s'(x_2)\\
		&=f'(x_1,x_2)s'(x_1\cdot x_2)\\
		&=f'(x_1,x_2)(\gamma(x_1\cdot x_2)s(x_1\cdot x_2))\\
		&=(f'(x_1,x_2)+\gamma(x_1\cdot x_2))s(x_1\cdot x_2)\\
		&=(f'(x_1,x_2)+\gamma(x_1\cdot x_2))(-f(x_1,x_2)(s(x_1)\cdot s(x_2)))\\
		&=(f'(x_1,x_2)+\gamma(x_1\cdot x_2)-f(x_1,x_2))(s(x_1)\cdot s(x_2)).
	\end{align*}
	As usual, the regularity of the $A$-action on fibers  allows one to conclude.
\end{proof}

\begin{lem}\label{lem:extension=AxX}
	Let $(X,\, \cdot)$ be a cycle set and $A$ an abelian group. Every extension
	$\extension{Y}{p}{X}{A}$ of $X$ by $A$ is equivalent to an extension
	$\extension{A\times_fX}{p_X}{X}{A}$ for some $f\in Z^2(X,A)$.
\end{lem}

\begin{proof}
	Let $s\colon X\to Y$ be any set-theoretic section to $p\colon Y\to X$. By Lemma~\ref{L:section->cocycle}, it induces a cocycle $f\in Z^2(X,A)$.  Consider the map 
	\[
	F\colon A\times_fX\to Y,
	\quad (\alpha,x)\mapsto \alpha s(x).
	\]
	It is a homomorphism of cycle sets, since one has
	\begin{align*}
		F((\alpha,x)(\beta,y))&=F((\beta+f(x,y),x\cdot y))\\
		&=(\beta+f(x,y))s(x\cdot y)=\beta((s(x)\cdot s(y))\\
		&=(\alpha s(x))\cdot (\beta s(y)) = F(\alpha,x)\cdot F(\beta,y).
	\end{align*}

	Let us prove that $F$ is bijective. It is injective since relation
	 $\alpha s(x)=\beta s(y)$ implies $x=p(s(x))=p(\alpha s(x))=p(\beta s(y))=p(s(y))=y$, and
	 $\alpha=\beta$ follows by the regularity of the $A$-action. 
	 It is surjective since for every $y\in Y$ there exists
	a unique $\alpha \in A$ such that $\alpha sp(y)=y$, implying
	$F(\alpha,p(y))=y$. 

	It remains to prove that $F$ is a map of extensions. One has 
	\[
		(p\circ F)(\alpha,x)=p(\alpha s(x))=ps(x)=x=p_X(\alpha,x),
	\]
	\[
		\alpha F(\beta,x)=\alpha(\beta s(x))=(\alpha+\beta)s(x)=F(\alpha+\beta,x). \qedhere
	\]
\end{proof}

\begin{lem}\label{L:AxfX=AxgX}
	Let $A$ be an abelian group, $(X,\, \cdot)$ a cycle set, and $f,g$ cocycles in $Z^2(X,A)$.  The
	extensions $\extension{A\times_fX}{p_X}{X}{A}$ and
	$\extension{A\times_gX}{p_X}{X}{A}$ are equivalent if and only if $f$ and $g$
	are cohomologous.
\end{lem}

\begin{proof}
	Suppose that $F$ is an equivalence between $\extension{A\times_fX}{}{X}{A}$ and 
	$\extension{A\times_gX}{}{X}{A}$,
	i.e. $F\colon A\times_f X \to A\times_g X$ is an
	isomorphism of cycle sets such that $p_X\circ F=p_X$ 
	and $\alpha F(\beta,x)=F(\alpha+\beta,x)$ for all $\alpha,\beta\in A$ and $x\in X$.
	Let $\gamma\colon X\to A$ be defined by $x\mapsto p_A(F(0,x))$, where
	$p_A\colon A\times X\to A$, $(\alpha,x)\mapsto \alpha$, is the canonical surjection.  Then one has
	\[
		F(\alpha,x)=F(\alpha(0,x))=\alpha F(0,x)=\alpha(\gamma(x),x)=(\alpha+\gamma(x),x).
	\]
	This implies
	\begin{align}
			F((\alpha,x)\cdot(\beta,y)) &=(\beta+f(x,y)+\gamma(x\cdot y),x\cdot y),\label{E:F1}\\
			F(\alpha,x)\cdot F(\beta,y) &=(\beta+\gamma(y)+g(x,y),x\cdot y).\label{E:F2}
	\end{align}
 	Since $F$ is a cycle set morphism, one obtains 
	 $g(x,y)-f(x,y)=\gamma(x\cdot y)-\gamma(y)$ for all 
	 $x,y\in X$, thus $f$ and $g$ are cohomologous. 

	Conversely, if $f$ and $g$ are cohomologous, there exists $\gamma\colon
	X\to A$ such that $g(x,y)-f(x,y)=\gamma(x\cdot y)-\gamma(y)$ for all $x,y\in
	X$. Consider the map 
	\[
		F\colon A\times_fX\to A\times_g X,\quad
		(\alpha,x)\mapsto (\alpha+\gamma(x),x).
	\]
	Computations \eqref{E:F1}-\eqref{E:F2} remain valid and show that $F$ is a cycle set morphism. It is bijective with the inverse $F^{-1}(\alpha,x)=(\alpha-\gamma(x),x)$, and clearly satisfies $p_X\circ F=p_X$ and
	$\alpha F(\beta,x)=F(\alpha+\beta,x)$ for all $\alpha,\beta\in A$, $x\in X$. 
\end{proof}

Put together, the preceding lemmas yield:

\begin{thm}\label{T:Ext=H2}
	Let $(X,\, \cdot)$ be a cycle set and $A$ an abelian group. The
	construction from Lemma~\ref{L:section->cocycle} yields a bijective
	correspondence between the set $\Ext(X,A)$ of equivalence classes of
	extensions of $X$ by $A$, and the cohomology group $H^2(X,A)$. 
\end{thm} 

\begin{rem}
	The extension procedure allows the construction of new cycle sets, and thus new left non-degenerate involutive braided sets, out of simpler ones. Another enhancement procedure for braidings is their algebraic deformation, in the spirit of Gerstenhaber. It was extensively studied by Eisermann \cite{Eisermann,Eisermann2}. Except for the diagonal case, a deformation transforms a set-theoretic solution to the YBE into an intrinsically linear one, and thus forces one outside the realm of cycle sets. For instance, deformations of the flip $(a,b) \mapsto (b,a)$ include all the braidings coming from quantum groups. The interaction of these two enhancements reserves many open questions:
\begin{enumerate}
	\item How can one relate the deformation theories of the braidings corresponding to a cycle set and its extension?
	\item Do cycle set extensions form a class of deformations of the corresponding braiding?
	\item Can the cycle set cohomology, responsible for extensions, be recovered inside Eisermann's Yang--Baxter cohomology, which controls deformations?
\end{enumerate}	
The last two phenomena do hold for the braidings associated to racks \cite{Eisermann,Eisermann2}.
\end{rem}

We conclude this section with an estimation of the \emph{Betti numbers} $\beta_n(X)$ of a cycle set $(X,\, \cdot)$---that is, the ranks of the free part of its integral homology groups $H_n(X, \Z)$. Recall the notion of \emph{orbits} of $(X,\, \cdot)$ from Example~\ref{EX:H1}.

\begin{pro}\label{PR:betti}
    Let $(X,\, \cdot)$ be a finite cycle set with $m$ orbits. Then the inequality
    $\beta_n(X)\geqslant m^n$ holds for all $n \geqslant 0$. 
\end{pro}

\begin{proof}
Consider the set $\operatorname{Orb}(X)$ of orbits of~$X$, endowed with the trivial cycle set operation $\O \cdot \O' = \O'$. The quotient map $X \twoheadrightarrow \operatorname{Orb}(X)$ is a cycle set morphism, and thus induces a chain complex surjection $C_n(X, \Z) \twoheadrightarrow C_n(\operatorname{Orb}(X), \Z)$ and a map in homology $\phi \colon H_n(X, \Z) \to H_n(\operatorname{Orb}(X), \Z)$. For a trivial cycle set the differentials $\partial_n$ are all zero. So the abelian groups $H_n(\operatorname{Orb}(X), \Z)$ are free, with $m^n$ generators $[(\O_1, \ldots, \O_n)]$, $\O_i \in \operatorname{Orb}(X)$. For such an $n$-tuple of orbits, put $e_{\O_1, \ldots, \O_n} = \sum_{x_i \in \O_i} (x_1, \ldots, x_n)$. The differential $\partial_n$ vanishes on this element of $C_n(X, \Z)$, since all its $n-1$ terms do. One gets a class $[e_{\O_1, \ldots, \O_n}] \in H_n(X, \Z)$, with 
$\phi([e_{\O_1, \ldots, \O_n}]) = |\O_1| \cdots |\O_n|[(\O_1, \ldots, \O_n)]$. 
The linear independence of the $[(\O_1, \ldots, \O_n)]$ now implies that of the $m^n$ elements $[e_{\O_1, \ldots, \O_n}]$ of $H_n(X, \Z)$.
\end{proof}

This proposition and its proof are inspired by the analogous result for racks, due to Carter--Jelsovsky--Kamada--Saito \cite{MR1812049}. Following them, one can extend the proposition to a certain class of infinite cycle sets. But this analogy does not go much further. For instance, for a wide class of shelves including all finite racks one actually has the equality $\beta_n(X) = |\operatorname{Orb}(X)|^n$ \cite{MR1948837,Lebed_Betti}, which fails for many small cycle sets. Indeed, while preparing the paper~\cite{MR1722951}, Etingof, Schedler, and Soloviev computed a complete list of non-degenerate involutive braidings of size $\leqslant 8$. Thanks to Schedler we could access this list and convert it into a readable database for \textsf{Magma}~\cite{MR1484478} and~\textsf{GAP}~\cite{GAP4}. This database (available from the authors immediately on request), and Rump's identification between such braidings and cycle sets in the finite setting, allowed us to write a computer program for calculating the homologies $H_n(X, \Z)$ for small~$n$ and~$X$. The results motivated 

\begin{que}
What information about a cycle set is contained in its Betti numbers?
\end{que}

\section{Applications to multipermutation braided sets}\label{S:Mutliperm}

In this section we will apply the extension techniques developed above for constructing cycle sets with prescribed properties---namely, the multipermutation level.  We will freely use notations from the previous section.

Let us first recall some notions and results from \cite{MR1722951,MR2132760}. 

\begin{defn}\label{D:SquareFree}
A cycle set $(X,\, \cdot)$ is called
\begin{itemize}
\item \emph{non-degenerate} if its squaring map $a \mapsto a \cdot a$ is bijective;
\item \emph{square-free} if it satisfies $a = a \cdot a$ for all $a \in X$.
\end{itemize}
\end{defn} 

Of course, square-free cycle sets are automatically non-degenerate.

\begin{pro}\label{PR:Reduction}
For a non-degenerate cycle set $(X,\, \cdot)$, consider the equivalence relation
$$a \approx a' \quad \Longleftrightarrow \quad a \cdot b = a' \cdot b \; \text{ for all } b \in X.$$
The operation~$\cdot$ then induces a non-degenerate cycle set  structure on the quotient set $\oX = X /  \approx$.
\end{pro}

\begin{proof}
To show that the induced operation is well defined, one should prove $a \cdot b \approx a' \cdot b'$ under the assumptions $a \approx a'$, $b \approx b'$. For any $c \in X$, one has
\begin{align*}
(a \cdot b) \cdot (a \cdot c) &\overset{a \approx a'}{=} (a' \cdot b) \cdot (a' \cdot c) = (b \cdot a') \cdot (b \cdot c) \overset{b \approx b'}{=} (b' \cdot a') \cdot (b' \cdot c)\\
&= (a' \cdot b') \cdot (a' \cdot c) \overset{a \approx a'}{=} (a' \cdot b') \cdot (a \cdot c).
\end{align*}
Since every element of~$X$ can be written in the form $a \cdot c$, we are done.

The cycle set property~\eqref{E:Cyclic} for~$\oX$ and the surjectivity of the left translations and of the squaring map follow from the analogous properties for~$X$. Let us now prove that the left translations on~$\oX$ are injective, i.e. that the relation $a \cdot b \approx a \cdot b'$ implies $b \approx b'$. Indeed, for any $c \in X$ one has
\begin{align*}
(b \cdot a) \cdot (b \cdot c) &= (a \cdot b) \cdot (a \cdot c) \overset{a \cdot b \approx a \cdot b'}{=} (a \cdot b') \cdot (a \cdot c) = (b' \cdot a) \cdot (b' \cdot c).
\end{align*}
For $c=a$ this yields, using the injectivity of the squaring map, the equality $b \cdot a = b' \cdot a$. The injectivity of the left translations on~$X$ then extracts from the computation above the desired property $b \cdot c = b' \cdot c$ for all $c \in X$.

The injectivity of the induced squaring map demands more work, and is proved
in~\cite{MR2132760}. Note that it is automatic in two important cases: the
square-free case ($\oX$ is square-free since~$X$ is so), and the finite case
(where surjectivity implies injectivity).
\end{proof}

\begin{defn}
\begin{itemize}
\item The induced structure from the proposition is called the
    \emph{retraction} of $(X,\, \cdot)$, denoted by $\Ret(X,\, \cdot)$.
\item  A non-degenerate cycle set $(X,\, \cdot)$ is called
    \emph{multipermutation (=MP) of level $n\geqslant 0$} if $n$ is the minimal
    number of retractions necessary to turn it into a one-element
    set (in the sense of $|\Ret^n(X,\, \cdot)| = 1$).
\item For an integer $m \geqslant 0$, the number~$N_m$ denotes the minimal size
    of square-free MP cycle sets of level~$m$. 
\end{itemize}
\end{defn} 

\begin{rem}
    The non-degeneracy is essential for the retraction 
    construction to work: Rump~\cite{MR2132760} exhibited an example of a
    degenerate cycle set such that the left translations for the induced
    operation are not injective.
\end{rem}

\begin{example}
The only possibility for a MP cycle set of level~$0$ is a one-element set with
its unique binary operation. Level~$1$ consists of the structures $(X,\, a\cdot
b = \theta(b))$, where $\theta$ is an arbitrary bijection $X
\overset{\sim}{\to} X$ and~$X$ has at least two element; they are naturally
called \emph{permutation cycle sets}. Such a cycle set is square-free if and
only if~$\theta$ is the identity map.  These descriptions imply $N_0 = 1$ and
$N_1 = 2$.
\end{example}

See~\cite{MR1722951} for more examples of and details on MP cycle sets. 

In \cite[Thm. 4]{MR3177933} Ced\'o, Jespers, and
Okni\'nski constructed finite square-free MP solutions of arbitrary 
level. Our extension theory yields a similar result.  

\begin{thm}\label{T:level_m}
Any square-free multipermutation cycle set of level~$m$ and size $N$ admits an
extension of size $2N$ which is square-free and multipermutation of
level~$m+1$.
\end{thm}

This theorem comes with an important corollary:
\begin{cor}\label{C:level_m}
	For any $m \geqslant 0$,
	\begin{enumerate}
	\item there exists a square-free MP cycle set of level~$m$ and size $2^{m}$;
	\item one has $N_{m+1} \leqslant 2N_{m}$.
	\end{enumerate}	
\end{cor}

Estimation $N_{m+1} \leqslant 2N_{m}+1$ was obtained earlier by Cameron and Gateva-Ivanova \cite{MR2885602}.

\begin{proof}[Proof of Theorem~\ref{T:level_m}]
Given a square-free MP cycle set $(X,\, \cdot)$ of level~$m$ and size~$N$, consider its $2$-cocycle
\begin{align*}
f \colon X \times X &\to \Z/2\Z, &f(x,y)& = \begin{cases} 0 & \text{if } x=y, \\ 1 & \text{if } x \neq y, \end{cases}
\end{align*} 
from Example~\ref{EX:DeltaCocycles}. Form the extension $\extension{\Z/2\Z\times_fX}{p_X}{X}{\Z/2\Z}$ (Example~\ref{E:central}). Explicitly, the cycle set operation on $\Z/2\Z\times_fX$ reads 
\begin{align*}
(\alpha,x)\cdot (\beta,x) &= (\beta+f(x,x),x\cdot x) = (\beta,x),\\
(\alpha,x)\cdot (\beta,y) &= (\beta+f(x,y),x\cdot y) = (\beta+1,x\cdot y) \qquad \text{ for } x \neq y.
\end{align*}
This clearly defines a square-free cycle set $(\Z/2\Z\times_fX,\, \cdot)$ of size~$2N$. We will now show that the map $(\alpha,x) \mapsto x$ induces a cycle set isomorphism 
$$\Ret(\Z/2\Z\times_fX,\, \cdot) \overset{\sim}{\to} (X,\, \cdot),$$ 
which implies that the cycle set $(\Z/2\Z\times_fX,\, \cdot)$ is MP of level~$m+1$. 

Concretely, we have to prove that the equality $(\alpha,x)\cdot (\beta,y)=(\alpha',x')\cdot (\beta,y)$ for all $(\beta,y) \in \Z/2\Z\times_fX$ is equivalent to $x = x'$. Indeed, for distinct $x$ and~$x'$ one has $(\alpha,x)\cdot (\beta,x) = (\beta,x)$ while $(\alpha',x')\cdot (\beta,x) = (\beta+1,x' \cdot x)$, and for $x = x'$ one compares $(\alpha,x)\cdot (\beta,y) = (\beta+f(x,y),x\cdot y) = (\alpha',x)\cdot (\beta,y)$.
\end{proof}

In fact, the retraction-extension 
interplay in our proof is more than a mere coincidence. Below is a more conceptual example of a connexion between the two constructions:   

\begin{pro}
    Let $(X,\, \cdot)$ be a non-degenerate cycle set. Then the natural projection $(X,\, \cdot) \twoheadrightarrow \Ret(X,\, \cdot)$ from~$X$ to its retraction
    factors through any extension $\extension{X}{p}{Y}{A}$ (Definition~\ref{D:Ext}) with the total cycle set~$X$.
\end{pro}

\begin{proof}
It suffices to check that any $x_1,x_2$ from the same fiber $p^{-1}(y)$ have
identical left translations, i.e., satisfy $x_1 \approx x_2$. Indeed, by the
definition of an extension, for any $x_1,x_2 \in p^{-1}(y)$ there exists a
$\alpha \in A$ with $x_1 = \alpha x_2$. But then one has $x_1 \cdot x = (\alpha
x_2) \cdot x = x_2 \cdot x$, which was to be proved.
\end{proof}

In Rump's identification between finite cycle sets and finite non-degenerate involutive braidings~\cite{MR2132760}, square-freeness is equivalent to the diagonal preservation property $\sigma(x,x) =(x,x)$. Thus an inspection of the non-degenerate involutive braiding  
list from~\cite{MR1722951} allows one to compute the first values of the sequence~$N_m$. The results are
given in Table~\ref{tab:N_m}. This computation answers in the negative  \cite[Open Question 6.13I(3)]{MR2885602}, which asks if the
relation $N_m = 2^{m-1}+1$, valid for the first values of~$m$, in fact holds
for all~$m$.  It also shows that our estimations $N_{m+1} \leqslant 2N_{m}$ are
not optimal.
\begin{table}[h]
	\caption{Some values of $N_m$.}
	\centering
	\begin{tabular}{|c|c|c|c|c|c|c|}
		\hline
		$m$ & $0$ & $1$ & $2$ & $3$ & $4$ & $5$\\
		\hline
		$N_m$ & $1$ & $2$ & $3$ & $5$ & $6$ & $8$\\
		\hline
	\end{tabular}
	\label{tab:N_m}
\end{table}

\section{Relation with group cohomology}

This section proposes an interpretation of the second cohomology group
$H^2(X,A)$ of a left non-degenerate (=LND) braided set $(X,\, \sigma)$ (in the
sense of Remark~\ref{R:LNDHom}, the star version) in terms of group cohomology.
This generalizes the analogous result for racks established by Etingof and
Gra{\~n}a~\cite{MR1948837}. For our second favourite example---that of cycle
sets---this can be helpful in studying extensions, in the light of
Theorem~\ref{T:Ext=H2}.

Recall that the \emph{group cohomology} $H^\ast(G,M)$ of a group~$G$ with coefficients in a right $G$-module~$M$ is the cohomology of the complex $(C^n(G,M),\partial_{\Group}^n)$ with $C^n(G,M) = \Fun(G^n,M)$ and
\begin{align*}
\partial_{\Group}^n f (g_1,\ldots,g_{n+1}) =& f(g_2,\ldots,g_{n+1}) + (-1)^{n+1} f(g_1,\ldots,g_{n})g_{n+1}\\
 &+ \sum\nolimits_{i=1}^n (-1)^{i} f(g_1,\ldots, g_{i-1}, g_i g_{i+1}, g_{i+2}, \ldots,  g_{n+1}) 
\end{align*}
(here the multiplication and action symbols are omitted for simplicity). The group we are interested in here is the structure group $G_{(X, \sigma)}$ (Definition~\ref{D:StructureSemiGroup}). Using Notation~\ref{N:sideways}, it can be defined as the free group on the set~$X$ modulo the relation $(a\cdot b) a = (b \wdot a)b$. The $G$-module we will consider comes from

\begin{lem}	\label{L:FunAsMod}
	Let $(X,\, \sigma)$ be a LND braided set, and~$A$ an abelian group. The map
	\begin{align*}
	\Fun(X,A) \times X &\to \Fun(X,A),\\
	(\gamma, a) &\mapsto (\gamma \cdot a \colon b \mapsto \gamma(a \wdot b))
	\end{align*}
	extends to a $G_{(X, \sigma)}$-module structure on the abelian group $\Fun(X,A)$.
\end{lem}

\begin{proof}
	The map $\lambda(a,b) = a \wdot b$ defines a left $(X,\, \sigma)$-module
	structure on itself (Remark~\ref{R:AdjActionBis}). By the left
	non-degeneracy, all the maps $b \mapsto a \wdot b$ are bijective,
	hence~$\lambda$ extends to a unique $G_{(X, \sigma)}$-module structure
	on~$X$ (Lemma~\ref{L:StructureSemiGroup}). This induces a right $G_{(X,
	\sigma)}$-module structure on $\Fun(X,A)$, which satisfies the desired
	property.
\end{proof}

The group cohomology with these choices turns out to be useful in studying the
cohomology of our braided set:

\begin{thm}\label{T:GroupCohom}
	Let $(X,\, \sigma)$ be a left non-degenerate braided set, and~$A$ an
	abelian group. Then one has the following abelian group isomorphism:
	\[
	H^2(X,A)\simeq H^1(G_{(X, \sigma)},\Fun(X,A)),
	\]
	where on the left $H^2$ stands for the cohomology theory from
	Remark~\ref{R:LNDHom} (the star version), and on the right group cohomology
	is used (the module structure on $\Fun(X,A)$ is described above).
\end{thm}

\begin{proof}
	Consider two maps 
	\begin{align*}
	\omega\colon Z^1(G_{(X, \sigma)},\Fun(X,A)) &\to Z^2(X,A),\\
	\theta &\mapsto ((x,y) \mapsto \theta(x)(y));\\
	\nu\colon Z^1(G_{(X, \sigma)},\Fun(X,A)) &\leftarrow Z^2(X,A),\\
	(x \mapsto f(x,-)) &\mapsfrom f.
	\end{align*}
We have to show that
\begin{enumerate}
\item for a $1$-cocycle~$\theta$, $\omega(\theta)$ is indeed a $2$-cocycle;
\item the map~$\nu$ is well defined;
\item $\omega$ sends $1$-coboundaries to $2$-coboundaries;
\item $\nu$ sends $2$-coboundaries to $1$-coboundaries.
\end{enumerate}
Since~$\omega$ and~$\nu$ are clearly mutually inverse, this will imply that they induce isomorphism in cohomology.

\begin{enumerate}
\item Let~$\theta$ be a map in $Z^1(G_{(X, \sigma)},\Fun(X,A))$. It means that it satisfies the relation 
\begin{align*}
\theta(g_1g_2) &= \theta(g_2) + \theta(g_1)\cdot g_2, & g_1,g_2 \in G_{(X, \sigma)}.
\end{align*} 
For $f = \omega(\theta)$, one needs to check the relation
$$f(x,z)+f(x\wdot y,x\wdot z)=f(y,z)+f(y\cdot x,y\wdot z)$$
for all $x,y,z \in X$, which rewrites as
$$\theta(x)(z)+\theta(x\wdot y)(x\wdot z)=\theta(y)(z)+\theta(y\cdot x)(y\wdot z).$$
Recalling the definition of the $G_{(X, \sigma)}$-action on $\Fun(X,A)$, one transforms this into
$$\theta(x)(z)+(\theta(x\wdot y) \cdot x)(z)=\theta(y)(z)+(\theta(y\cdot x)\cdot y)(z).$$
The $1$-cocycle property for~$\theta$ simplifies it to
$$\theta((x\wdot y) x)= \theta((y\cdot x) y),$$
which follows from the relation $(x\wdot y) x = (y\cdot x) y$ valid in the structure group $G_{(X, \sigma)}$.
\item Recall that a map $\theta\colon G_{(X, \sigma)}\to\Fun(X,A)$ is a $1$-cocycle if and only if the map
	$G_{(X, \sigma)}\to G_{(X, \sigma)}\ltimes\Fun(X,A)$ given by $g\mapsto (g,\theta(g))$ is a group
	morphism; here $G_{(X, \sigma)}\ltimes\Fun(X,A)$ is the set  $G_{(X, \sigma)}\times\Fun(X,A)$ endowed with the group multiplication $(g,\gamma)(g',\gamma')=(gg',\gamma'+\gamma\cdot g')$. The verification of this well-known property is elementary. We will now show that, for $f \in Z^2(X,A)$, the assignment $\iota_f \colon X \to G_{(X, \sigma)}\ltimes\Fun(X,A)$, $x \mapsto (x,f(x,-))$, extends to a unique group morphism $G_{(X, \sigma)} \to G_{(X, \sigma)}\ltimes\Fun(X,A)$. For this it suffices to check the property $\iota_f(x\wdot y) \iota_f(x) = \iota_f(y\cdot x) \iota_f(y)$ for all $x,y \in X$. Explicitly, it reads
$$(x\wdot y,f(x\wdot y,-))(x,f(x,-))=(y\cdot x,f(y\cdot x,-))(y,f(y,-)),$$
which simplifies as	
\begin{align*}
&((x\wdot y)x,f(x\wdot y,-)\cdot x + f(x,-))=\\
&((y\cdot x)y,f(y\cdot x,-)\cdot y+f(y,-)).
\end{align*}
Since the relation $(x\wdot y) x = (y\cdot x) y$ always holds in $G_{(X, \sigma)}$, it remains to show the equation
$$ f(x\wdot y, x \wdot -) + f(x,-)= f(y\cdot x,y \wdot -)+f(y,-),$$
which is precisely the definition of a $2$-cocycle.
\item A $1$-coboundary in $C^1(G_{(X, \sigma)},\Fun(X,A))$ is a map of the form 
$$\partial_{\Group}^0 \gamma \colon g \mapsto \gamma - \gamma \cdot g \quad \text{ for some } \gamma \in \Fun(X,A).$$ 
Its image $\omega(\partial_{\Group}^0 \gamma)$ is then the map sending $(x,y)$ to
$$(\gamma - \gamma \cdot x)(y) = \gamma(y) - \gamma(x \wdot y) = - \partial^1 \gamma(x,y),$$
yielding $\omega(\partial_{\Group}^0 \gamma) = \partial^1 (-\gamma)$.
\item A $2$-coboundary in $C^2(X,A)$ is a map of the form $\partial^1 \gamma$ for some $\gamma \in \Fun(X,A)$. Since~$\omega$ and~$\nu$ are mutually inverse, the computation above implies the relation $\nu(\partial^1 \gamma) = \partial_{\Group}^0 (-\gamma)$.
\qedhere
\end{enumerate}
\end{proof}

One could wonder if a similar result holds true for higher cohomology groups. A first step in this direction is the identification
	\[
	H^n(X,A)\simeq H_{tr}^{n-1}(X,\Fun(X,A)),
	\]
which follows from the obvious isomorphism of cochain complexes.	Here $H^n$ stands for the cohomology theory from Remark~\ref{R:LNDHom} (the star version), and $H_{tr}^{n-1}$ is the cohomology of $(X,\, \sigma)$	with trivial coefficients, acting on $\Fun(X,A)$ on the right as in Lemma~\ref{L:FunAsMod}. This latter cohomology, described for instance in~\cite{LebedIdempot}, mimics group cohomology. Our identification generalizes Etingof and Gra{\~n}a's result for racks~\cite[Proposition 5.1]{MR1948837}. Note that they used $H_{tr}^n(X,A)$ instead of our $H^n(X,A)$; the two cohomology theories coincide for racks, but differ for more general braided sets. Unfortunately, precise relations between $H_{tr}^{n-1}(X,\Fun(X,A))$ and the group cohomology $H^{n-1}(G_{(X, \sigma)},\Fun(X,A))$ are as for now very poorly understood. To our knowledge, the only connection between the two is the quantum symmetrizer map, known to be an isomorphism in some very particular cases \cite{FarinatiGalofre,LebedIdempot}.

\bibliographystyle{abbrv}
\bibliography{refs}
 
\end{document}